%% file: main.tex
\documentclass[12pt]{article}
\usepackage{amsmath}
\usepackage{amsfonts}
\usepackage{amssymb}
\usepackage{amsthm}
\usepackage[totalwidth=5.35in, totalheight=8.3in]{geometry}

\begin{document}

\newcommand{\C}{{\mathbb{C}}}
\newcommand{\R}{{\mathbb{R}}}
\newcommand{\Q}{{\mathbb{Q}}}
\newcommand{\Z}{{\mathbb{Z}}}
\newcommand{\N}{{\mathbb{N}}}
\newcommand{\q}{\left}
\newcommand{\w}{\right}
\newcommand{\Vol}[1]{\mathrm{Vol}\q(#1\w)}
\newcommand{\B}[4]{B_{\q(#1,#2\w)}\q(#3,#4\w)}
\newcommand{\Bsub}[5]{B_{\q(#1,#2\w)_{#3}}\q(#4,#5\w)}
\newcommand{\Btsub}[5]{\widetilde{B}_{\q(#1,#2\w)_{#3}}\q(#4,#5\w)}
\newcommand{\Bs}[3]{B_{#1}\q(#2,#3\w)}
\newcommand{\Bt}[4]{\widetilde{B}_{\q(#1,#2\w)}\q(#3,#4\w)}
\newcommand{\Bts}[3]{\widetilde{B}_{#1}\q(#2,#3\w)}
\newcommand{\Bm}[2]{B_{\q(X^1,d^1\w),\ldots, \q(X^\nu,d^\nu\w)}\q(#1,#2\w)}
\newcommand{\Bmt}[2]{\widetilde{B}_{\q(X^1,d^1\w),\ldots, \q(X^\nu,d^\nu\w)}\q(#1,#2\w)}
\newcommand{\sI}[2]{\mathcal{I}\q( #1,#2\w)}
\newcommand{\Det}[2]{\det_{#1\times #1}\q( #2\w)}
\newcommand{\Span}[1]{\mathrm{span}{\q\{#1\w\}}}
\newcommand{\Lpp}[2]{L^{#1}\q(#2\w)}
\newcommand{\Lppn}[3]{\q\|#3\w\|_{\Lpp{#1}{#2}}}
\newcommand{\Texp}[1]{\stackrel{\longrightarrow}{\exp}\q( #1 \w)}
\newcommand{\Cj}[2]{C^{#1}\q(#2\w)}
\newcommand{\Cjn}[3]{\q\|#3\w\|_{C^{#1}\q(#2\w)}}
\newcommand{\sL}{\mathcal{L}}
\newcommand{\pd}[1]{\frac{\partial}{\partial #1}}
\newcommand{\ct}{\tilde{c}}
\newcommand{\ch}{\hat{c}}
\newcommand{\rp}[2]{\rho\q(#1,#2\w)}
\newcommand{\grad}{\bigtriangledown}
\newcommand{\sM}{\mathcal{M}}
\newcommand{\sMt}{\widetilde{\mathcal{M}}}
\newcommand{\ad}[1]{\mathrm{ad}\q(#1\w)}
\newcommand{\sd}{\sum d}
\newcommand{\prel}{\precsim}
\newcommand{\prea}{\simeq}
\newcommand{\numsub}{\nu}
\newcommand{\Bmus}[2]{{B_{\q(X^1, d^1\w),\ldots, \q(X^\nu, d^\nu\w)}\q(#1,#2\w)}}
\newcommand{\Bmust}[2]{{B_{\q(X^1, d^1\w),\q(X^2, d^2\w)}\q(#1,#2\w)}}
\newcommand{\hd}{\hat{d}}
\newcommand{\hX}{\widehat{X}}
\newcommand{\dm}{d_{q+1}^{\vee}}
\newcommand{\sPo}[5]{\mathcal{P}_{1}^{#1}\q(#2,#3,#4,#5\w)}
\newcommand{\sPt}[4]{\mathcal{P}_{2}^{#1}\q(#2,#3,#4\w)}
\newcommand{\sPr}[4]{\mathcal{P}_{3}^{#1}\q(#2,#3,#4\w)}
\newcommand{\sPn}[2]{\mathcal{P}_{#1}^{#2}}
\newcommand{\Ho}{\mathbb{H}^1}
\newcommand{\OpL}[1]{ {\mathrm{Op}_L} \q(#1\w) }
\newcommand{\OpR}[1]{ {\mathrm{Op}_R} \q(#1\w) }
\newcommand{\sC}{\mathcal{C}}
\newcommand{\sA}{\mathcal{A}}
\newcommand{\sInp}{\mathcal{I}}
\newcommand{\sD}{\mathcal{D}}

\newtheorem{thm}{Theorem}[section]
\newtheorem{cor}[thm]{Corollary}
\newtheorem{prop}[thm]{Proposition}
\newtheorem{lemma}[thm]{Lemma}
\newtheorem{conj}[thm]{Conjecture}

\theoremstyle{remark}
\newtheorem{rmk}[thm]{Remark}

\theoremstyle{definition}
\newtheorem{defn}[thm]{Definition}

\theoremstyle{remark}
\newtheorem{example}[thm]{Example}

\numberwithin{equation}{section}

\title{Multi-parameter Carnot-Carath\'eodory balls and the theorem of Frobenius}
\author{Brian Street\footnote{The author was partially supported by NSF DMS-0802587.}}
\date{}

\maketitle
\tableofcontents

\begin{abstract}
\input{abstract}
\end{abstract}

%\section{Introduction}
%\input{intro}

\section{Introduction}
\input{intro3}

	\subsection{Informal statement of results and outline of the paper}\label{SectionResults}
	\input{results}

	\subsection{Past work}\label{SectionPastWork}
	\input{pastwork}

		\subsubsection{Single-parameter balls and the work of Nagel, Stein, and Wainger}\label{SectionNSW}
		\input{nsw}

		\subsubsection{Weakly comparable balls, the work of Tao and Wright, and a motivating example}\label{SectionTW}
		\input{taowright}

		\subsubsection{Nonsmooth H\"ormander vector fields}\label{SectionNonsmooth}
		\input{nonsmooth}

		\subsubsection{Some multi-parameter singular integrals}\label{SectionStreet}
		\input{streetsing}

		\subsubsection{The classical theorem of Frobenius}\label{SectionClassicalFrob}
		\input{classicalfrob}

%\section{Motivation and interpretation of results}\label{SectionMotiv}
%	\input{motivation}

\section{Basic definitions}
\input{defn2}

%\section{Basic definitions}
%\input{defn}

\section{The (uniform) theorem of Frobenius}\label{SectionThmOfFrob}
\input{frob}

%	\subsection{The classical theorem of Frobenius}\label{SectionClassicalFrob}
%		\input{classicfrob}

%\section{Carnot-Carath\'eodory balls at the unit scale Old}%\label{SectionUnitScale}
%\input{unitscale}

\section{Carnot-Carath\'eodory balls at the unit scale}\label{SectionUnitScale}
\input{unitscalenew}

	%\subsection{Unit operators at the unit scale}%\label{SectionUnitOpsAtUnitScale}
	%	\input{unitops}

	\subsection{Control of vector fields}\label{SectionControlUnitScale}
		\input{controlunitscale}

	\subsection{Unit operators at the unit scale}\label{SectionUnitOpsAtUnitScale}
		\input{unitopsnew}

\section{Multi-parameter Carnot-Carath\'eodory balls}\label{SectionMultiParamBalls}
	\input{multiintronew}

	\subsection{The main theorem}\label{SectionBallsAtAPoint}
		\input{ballsatpointthird}

	\subsection{Applications and examples}\label{SectionExamplesOfMainThm}
		\input{examples}

		\subsubsection{Weakly comparable balls}\label{SectionWeaklyCompBalls}
			\input{weaklycompnew}

		\subsubsection{Multiple lists that span}\label{SectionMultipleLists}
			\input{multiplelists}

		\subsubsection{An example where the methods of \cite{NagelSteinWaingerBallsAndMetricsDefinedByVectorFields} do not apply}
			\input{badnsw}

		\subsubsection{Lifting results from the single parameter case and the Campbell-Hausdorff formula}\label{SectionLiftingSingleParam}
			\input{lifting}

	\subsection{Control of vector fields}\label{SectionControlEveryScale}
		\input{controleveryscale}

		\subsubsection{Examples of control}\label{SectionExamplesOfControl}
			\input{exampofcontrol}

%	\subsection{Multi-parameter balls at a point}%\label{SectionBallsAtAPoint}
%		\input{ballsatpointnew}

%		\subsubsection{No restrictions on $\delta$}\label{SectionNoRestrictions}
%			\input{norestrict}

%	\subsection{Multi-parameter balls on a compact set}\label{SectionBallsOnACptSet}
%		\input{ballsoncptnew}

%	\subsection{Weakly comparable balls}%\label{SectionWeaklyCompBalls}
%		\input{weaklycomp}

\section{Unit operators and maximal functions}\label{SectionMaximalFuncs}
	\input{maxfuncsnew}

	\subsection{Some comments on metrics}\label{SectionMetrics}
		\input{metrics}

	\subsection{Foliations whose leaves are locally spaces of homogeneous type}\label{SectionHomogeneousFoliations}
		\input{homogfoli}

\appendix
\section{Two results from calculus}\label{AppendCalc}
\input{appendcalc}

%\section{The theorem of Frobenius}\label{AppendFrob}
%\input{appendfrob}

\section{The Cauchy-Binet formula}\label{AppendixLinear}
\input{appendlinear}

\bibliographystyle{amsalpha}

\bibliography{multiparam}

\center{\it{University of Wisconsin-Madison, Department of Mathematics, 480 Lincoln Dr., Madison, WI, 53706}}

\center{\it{street@math.wisc.edu}}

\center{MSC2000: 53C17 (primary), 53C12, 42B25 (secondary)}

\center{Keywords: Sub-Riemannian geometry, Carnot-Carath\'eodory geometry, Frobenius theorem, multi-parameter, spaces of homogeneous type, maximal functions}

\end{document}

%% file: abstract.tex
%Combining ideas of Tao and Wright with ideas of Nagel, Stein, and Wainger,
%we study multi-parameter Carnot-Carath\'eodory balls.  
We study multi-parameter Carnot-Carath\'eodory balls, generalizing
results due to Nagel, Stein, and Wainger in the single parameter
setting.
The main technical
result is seen as a uniform version of the theorem of Frobenius.
In addition, we study maximal functions associated to certain
multi-parameter families of Carnot-Carath\'eodory balls.

%% file: intro3.tex
In the seminal paper \cite{NagelSteinWaingerBallsAndMetricsDefinedByVectorFields},
Nagel, Stein, and Wainger gave a detailed study of Carnot-Carath\'eodory
balls.  The main purpose of this paper is to develop an analogous theory
of {\it multi-parameter} Carnot-Carath\'eodory balls:  a situation where
the methods of \cite{NagelSteinWaingerBallsAndMetricsDefinedByVectorFields}
do not apply in general.
We will see that the main results for multi-parameter Carnot-Carath\'eodory
balls follow from a certain ``uniform'' version of the theorem
of Frobenius on involutive distributions.\footnote{Here, and in the rest
of the paper, we are considering (possibly) singular distributions.  That is, the dimension of the distribution may vary from point to point.}
We will prove this version of the theorem of Frobenius
by building on the work of \cite{NagelSteinWaingerBallsAndMetricsDefinedByVectorFields}
along with work of Tao and Wright \cite{TaoWrightLpImprovingBoundsForAverages}.
Our primary motivation is to obtain the properties of multi-parameter
balls which are relevant for developing a theory of multi-parameter
singular integrals, which will be the subject of a future paper.
To this end, we will estimate the volume of certain multi-parameter
balls, and we will study maximal functions associated to certain
families of multi-parameter balls.  In addition, we will study
the composition of certain ``unit operators.''

We begin by introducing the notion of a Carnot-Carath\'eodory ball.
Suppose we are given $q$ $C^1$ vector fields, $X_1,\ldots, X_q$ on 
an open set $\Omega\subseteq \R^n$; denote this list of vector fields by $X$.
We define the Carnot-Carath\'eodory ball of unit radius, centered at $x_0\in \Omega$,  with
respect to the list $X$ by:\footnote{Here, and in the rest of the paper, we write
$\gamma'\q( t\w) = Z\q( t\w)$ to mean
$\gamma\q( t\w) = \gamma\q( 0\w) + \int_0^t Z\q(  s\w)\: ds$.}
\begin{equation*}
\begin{split}
B_X\q( x_0\w):=\bigg\{y\in \Omega \:\bigg|\: &\exists \gamma:\q[0,1\w]\rightarrow \Omega, \gamma\q( 0\w) =x_0, \gamma\q( 1\w) =y,  \\
&\gamma'\q( t\w) = \sum_{j=1}^q a_j\q( t\w) X_j\q( \gamma\q( t\w)\w),
 a_j\in L^\infty\q( \q[0,1\w]\w), \\
&\Lppn{\infty}{\q[0,1\w]}{\q(\sum_{1\leq j\leq q} \q|a_j\w|^2\w)^{\frac{1}{2}}}<1\bigg\}.
\end{split}
\end{equation*}
Now that we have the definition for Carnot-Carath\'eodory balls
with unit radius, we may define Carnot-Carath\'eodory balls of any
radius merely by scaling the vector fields.  This leads us
directly to multi-parameter balls, of which the single parameter
balls of \cite{NagelSteinWaingerBallsAndMetricsDefinedByVectorFields}
are a special case.

Fix $\nu\geq 1$, an integer.  We will discuss $\nu$-parameter balls.
To each vector field $X_j$, we associate a formal degree
$0\ne d_j\in\q[0,\infty\w)^\nu$.  We denote by $\q( X,d\w)$ the list of vector fields
$$\q(X_1,d_1 \w),\ldots, \q( X_q, d_q\w).$$
Furthermore, for $\delta\in \q[ 0,\infty\w)^\nu$, we denote by $\delta^{d} X$
the list of vector fields:
$$\delta^{d_1} X_1, \ldots, \delta^{d_q} X_q$$
where $\delta^{d_j}$ is defined by the standard multi-index notation.
That is, $\delta^{d_j} = \prod_{\mu=1}^\nu \delta_\mu^{d_j^\mu}$.
Then we define the multi-parameter Carnot-Carath\'eodory ball
centered at $x_0\in \Omega$ of radius $\delta$ by:
$$\B{X}{d}{x_0}{\delta}:= B_{\delta^d X}\q( x_0 \w).$$
The theory in \cite{NagelSteinWaingerBallsAndMetricsDefinedByVectorFields}
concerns the case when $\nu=1$ (see Section \ref{SectionNSW} for a discussion of their results).
One of the main goals of this paper is to develop appropriate conditions
on the list $\q( X,d\w)$ to allow for a general theory of such
multi-parameter balls.

It has long been understood that singular integrals corresponding
to the single parameter balls of \cite{NagelSteinWaingerBallsAndMetricsDefinedByVectorFields}
play a fundamental role in many questions in the regularity of 
linear partial differential operators that are defined by vector fields;
in particular, they arise in many questions in several complex variables.
This began in \cite{FollandSteinEstimatesForTheDbarComplex,RothschildSteinHypoelltipicDifferentialOperators},
and was followed by \cite{SanchezCalleFundamentalSolutionsAndGeometry,FeffermanSanchezCalleFundamentalSolutionsForSecondOrderSubellipticOperators,JerisonSanchezCalleSubellipticSecondOrderDifferential}.  
These works were followed by many others; too many to offer a detailed
account here.  In the area of several complex variables, some examples are
\cite{ChristRegularityPropertiesOfTheDbarbEquation,NagelRosaySteinWaingerEstimatesForTheBergmanAndSzegoKernels,ChangNagelSteinEstimatesForTheDbarNeumannProblem,KoenigOnMaximalSobolevAndHolderEstimatesForTheTangentialCR}.
Recently, however,
{\it multi-parameter }singular integrals, where the underlying geometries
are non-Euclidean, have been shown to arise in
various special cases in several complex variables
and in the parametricies for certain linear partial differential
operators.  Moreover, these examples are not even amenable to the
usual product theory of singular integrals (as is covered in, for example, \cite{NagelSteinOnTheProductTheoryOfSingularIntegrals}):  the geometries overlap
in a non-trivial way.  In this vein see 
\cite{MullerRicciSteinMarcinMultipliersAndMultiParam,NagelRicciSteinSingularIntegralsWithFlagKernels,NagelSteinTheDbarComplexOnDecoupled,StreetAdvances}.
It is our hope that this paper will help play a role in unlocking
more general theories.

%% file: results.tex
In this section, we offer a brief overview of some of the key results
of the paper.  One of the main aspects of the proofs, and of the
interrelationships between the results, is keeping careful track
of parameters the constants in the results depend on.  This
makes the rigorous formulation of these results somewhat
technical.  Because of this, in this section, we state
the results only in the $C^\infty$ category (while we will
later deal with less smoothness) and are not precise about
what parameters the constants depend on.  
After each result
we will refer the reader to the part of the paper which contains
the precise formulation of the result.
In addition, Theorem \ref{ThmIntroMainThm} represents only a special case of the
main result of the paper (Theorem \ref{ThmMultiBallsAtPoint}).

Before we begin, we need a few pieces of notation.
Given two integers $1\leq m\leq n$, we let $\sI{m}{n}$ be the set of
all lists of integers $\q( i_1,\ldots, i_m\w)$, such that:
$$1\leq i_1<i_2<\cdots<i_m\leq n.$$
Furthermore, suppose $A$ is an $n\times q$ matrix, and suppose $1\leq n_0\leq n\wedge q$, for $I\in \sI{n_0}{n}$, $J\in \sI{n_0}{q}$ define the $n_0\times n_0$
matrix $A_{I,J}$ by using the rows from $A$ which are listed in $I$
and the columns of $A$ which are listed in $J$.
We define:
$$\det_{n_0\times n_0} A = \q( \det A_{I,J}\w)_{\substack{I\in \sI{n_0}{n}\\J\in \sI{n_0}{q}}}.$$
In particular, $\det_{n_0\times n_0} A$ is a {\it vector}.
It will not be important to us in which order the coordinates are arranged.
For further information on this object, see Appendix \ref{AppendixLinear}.

For a vector $v\in \R^n$, we write $\q|v\w|$ for the usual $\ell^2$
norm, and $\q|v\w|_\infty$ and $\q|v\w|_1$ for the $\ell^\infty$
and $\ell^1$ norms, respectively.  For a matrix $A$, we write
$\q\|A\w\|$ for the usual operator norm.
Finally, we write $B_n\q( \eta\w)$ for the ball in $\R^n$,
centered at $0$, of
radius $\eta>0$ in the $\q|\cdot\w|$ norm.

The setup of the main result is as follows.  We are given $q$
$C^\infty$ vector fields $X_1,\ldots, X_q$
defined on a fixed open set $\Omega\subseteq \R^n$.
Corresponding to each vector field we are given a formal
degree $0\ne d_j \in \q[0,\infty\w)^\nu$, where $\nu$ is a fixed
positive integer.  We let $\q( X,d\w)$ denote the list
of vector fields with formal degrees $\q( X_1,d_1\w), \ldots, \q( X_q,d_q\w)$,
and we let $X$ denote the list of vector fields $X_1,\ldots, X_q$.
At times, we will identify $X$ 
with the $n\times q$ matrix whose columns are given by $X_1,\ldots, X_q$ (similarly for other lists of vector fields).
Our main assumption is that for every $\delta\in \q[0,1\w)^\nu$,
with $\q|\delta\w|$ sufficiently small,\footnote{Throughout the rest of this introduction, $\delta$ will always denote a small element of $\q[0,1\w)^\nu$.}
 we have:%\footnote{Actually, \eqref{EqnIntroMainAssump} need only hold near each point; on a neighborhood that .  See Section \ref{SectionBallsAtAPoint} for details.}
\begin{equation}\label{EqnIntroMainAssump}
\q[\delta^{d_j}X_j, \delta^{d_k}X_k\w] = \sum_{l=1}^q c_{j,k}^{l,\delta} \delta^{d_l}X_l.
\end{equation}
We assume that $c_{j,k}^{l,\delta}\in C^\infty$ {\it uniformly} in $\delta$;
i.e., that as $\delta$ varies, $c_{j,k}^{l,\delta}$ varies over a bounded
subset of $C^\infty$.\footnote{Even in the smooth case, the assumptions
in Section \ref{SectionBallsAtAPoint} require less than we outline here.}

\begin{rmk}\label{RmkIntroLeaves}
Note that we have not assumed that the list of vector fields $X$ spans the tangent space.
This will prove the be an essential point in much of what follows.
One thing to observe is that while $X$ may not span the
tangent space, \eqref{EqnIntroMainAssump} implies that the
distribution spanned by $\delta^d X$
is involutive, and therefore the classical
theorem of Frobenius applies to show that these vector fields
foliate $\Omega$ into leaves (see Section \ref{SectionClassicalFrob} for a review of
the classical theorem of Frobenius).  The Carnot-Carath\'eodory ball
$\B{X}{d}{x_0}{\delta}$ is then an open subset of the leaf passing
through $x_0$ generated by this distribution.
In what follows, we will estimate the volume of this ball (denoted
by $\Vol{\B{X}{d}{x_0}{\delta}}$).  This volume is taken in the
sense of the induced Lebesgue measure on the leaf.
\end{rmk}

For $n_0\leq q$ and $J=\q(j_1,\ldots, j_{n_0}\w)\in \sI{n_0}{q}$, we write $\q( X,d\w)_{J}$ to denote
the list of vector fields with formal degrees $\q( X_{j_1}, d_{j_1}\w), \ldots, \q( X_{j_{n_0}}, d_{j_{n_0}}\w)$,
and we write $X_{J}$ to denote the list of vector fields
$X_{j_1},\ldots, X_{j_{n_0}}$, similarly we write $d_J$ for the list of
formal degrees $d_{j_1},\ldots, d_{j_{n_0}}$.
For each $x\in \Omega$, let $n_0\q( x,\delta\w)=\dim\Span{\delta^{d_1}X_1\q( x\w) ,\ldots, \delta^{d_q}X_q\q( x\w)}$,\footnote{Note,
the dependence of $n_0\q( x,\delta\w)$ 
on $\delta$ only involves which of the coordinates of $\delta$ are $0$.}
and for each $x\in \Omega$,
 and $\delta$ sufficiently small
pick $J\q( x,\delta\w)\in \sI{n_0\q( x,\delta\w)}{q}$ such that:
$$\q|\det_{n_0\q( x,\delta\w)\times n_0\q(x,\delta\w)} \delta^{d_{J\q(x,\delta\w)}} X_{J\q( x,\delta\w)}\q( x\w) \w|_\infty= \q|\det_{n_0\q(x,\delta\w)\times n_0\q(x,\delta\w)} \delta^d X\q( x\w) \w|_\infty.$$

For $u\in \R^{n_0\q( x,\delta\w)}$ with $\q| u\w|$ sufficiently small, define the map\footnote{If $Z$ is a $C^1$ vector field, then $e^Z x$ is defined in the following way.  Let $E\q(t\w)$ be the unique solution to the ODE $\frac{d}{dt} E\q(t\w) = Z\q(E\q(t\w)\w)$, $E\q(0\w)=x$.  Then, $e^{Z}x$ is defined to be $E\q(1\w)$, provided this solution exists up to $t=1$ (which it will if $Z$ has sufficiently small $C^1$ norm).  See Appendix \ref{AppendCalc} for further details.}
$$\Phi_{x,\delta}\q( u\w) = e^{u\cdot \q(\delta^d X\w)_{J\q( x,\delta\w)}}x=e^{u\cdot \delta^{d_{J\q( x,\delta\w)}}X_{J\q( x,\delta\w)}}x.$$
Our main theorem is:
\begin{thm}\label{ThmIntroMainThm}
Let $K$ be a compact subset of $\Omega$.  Then, there exist constants $\eta, \xi\approx 1$ such that for all $\delta$
sufficiently small and all $x\in K$:
$$\B{X}{d}{x}{\xi \delta}\subseteq \Phi_{x,\delta}\q(B_{n_0\q( x,\delta\w)}\q(\eta\w)\w)\subseteq \B{X}{d}{x}{\delta}$$
and
\begin{enumerate}
\item\label{ItemOnetoOne} $\Phi_{x,\delta}: B_{n_0\q( x,\delta\w)}\q( \eta\w) \rightarrow \B{X}{d}{x_0}{\delta}$ is one-to-one.
\item\label{ItemConstDet} For all $u\in B_{n_0\q( x,\delta\w)}\q( \eta\w)$, $\q|\det_{n_0\q( x,\delta\w) \times n_0\q( x,\delta\w)}d\Phi_{x,\delta}\q(u\w)\w|\approx \q|\det_{n_0\q(x,\delta\w) \times n_0\q( x,\delta\w)} \delta^d X\q( x\w)\w|$.
\item\label{ItemVol} $\Vol{\B{X}{d}{x}{\delta}}\approx \q|\det_{n_0\q( x,\delta\w)\times n_0\q( x,\delta\w)} \delta^d X\q( x\w)\w|.$  This is essentially a consequence of Items \ref{ItemOnetoOne} and \ref{ItemConstDet}.
\item\label{ItemDoubling} $\Vol{\B{X}{d}{x}{2\delta}}\lesssim \Vol{\B{X}{d}{x}{\delta}}$.  This is essentially a consequence of Item \ref{ItemVol}.
\end{enumerate}
\end{thm}
In addition to what is stated in Theorem \ref{ThmIntroMainThm}, a number
of other technical results hold which are essential for applications.
In particular, the map $\Phi_{x,\delta}$ can be used as a ``scaling''
map.  This is because the pullback of the vector fields $\delta^d X$
via the map $\Phi_{x,\delta}$ to $B_{n_0\q( x,\delta\w)}\q( \eta\w)$ satisfy
good properties uniformly in $x$ and $\delta$.\footnote{See
Section \ref{SectionLiftingSingleParam} to see a scaling technique
in action.}
  We refer the
reader to Section \ref{SectionMultiParamBalls} for a discussion
of these results along with the rigorous statement of Theorem \ref{ThmIntroMainThm}.
Note that in the single parameter case, Item \ref{ItemDoubling} is the
main inequality that must be satisfied for the balls $\B{X}{d}{x}{\delta}$
to form a space of homogeneous type (when paired with Lebesgue measure).
This is the first sign that these multi-parameter balls, in this generality,
will yield analogs to some results from the single-parameter
Calder\'on-Zygmund theory.

As was mentioned earlier, Theorem \ref{ThmIntroMainThm} will follow
from a ``uniform'' version of the theorem of Frobenius.
To understand this connection, one must first understand
the connection between multi-parameter balls and single parameter balls.
Given the multi-parameter formal degrees $0\ne d_j\in \q[0,\infty\w)^\nu$,
we obtain corresponding single parameter degrees, which we denote by $\sum d$,
and are defined by $\q(\sum d\w)_j:= \sum_{\mu=1}^\nu d_j^{\mu}=\q|d_j\w|_1$.
Given $\delta$, we decompose $\delta=\delta_0\delta_1$,
where $\delta_0\in \q[0,\infty\w)$ and $\delta_1\in \q[0,\infty\w)^\nu$.\footnote{Of course this decomposition is not unique.}
Then, directly from the definition, we obtain:
$$\B{X}{d}{x_0}{\delta} = \B{\delta_1^d X}{\sum d}{x_0}{\delta_0}=\B{\delta^d X}{\sum d}{x_0}{1}.$$
Because of this, to prove Theorem \ref{ThmIntroMainThm} for a fixed
$x\in K$ and a fixed $\delta$, it suffices to prove a result for
a list of vector fields with single-parameter formal degrees:  the vector fields
$\q( \delta^d X, \sum d\w)$.

At this point, we change notation.  We now work in the single-parameter
case $\nu=1$.  We suppose we are given $q$ $C^\infty$ vector fields
$X_1,\ldots, X_q$ on a fixed open set $\Omega\subseteq \R^n$ and
associated to each $X_j$ we are given a formal degree $d_j\in \q(0,\infty\w)$.
We further suppose that we are given a fixed point $x_0\in \Omega$.\footnote{In addition, we need to assume that $x_0$ is not too close to the boundary of $\Omega$, but we ignore such technicalities in this introduction.}
One should think of this single-parameter list $\q( X,d\w)$ as coming
from a multi-parameter list via $\q( \delta^d X,\sum d\w)$ as in Theorem \ref{ThmIntroMainThm}.
Our main assumption is that we have:
\begin{equation}\label{EqnIntroAssumpFrob}
\q[X_j, X_k\w] =\sum_{l=1}^q c_{j,k}^l X_l
\end{equation}
where $c_{j,k}^l\in C^\infty$.

Let $n_0=\dim \Span{X_1\q( x_0\w),\ldots, X_q\q( x_0\w)}$, and
pick $J\in \sI{n_0}{q}$ such that:
$$\q|\det_{n_0\times n_0} X_J\q( x_0\w)\w|_\infty= \q|\det_{n_0\times n_0} X\q( x_0\w)\w|_\infty.$$
For $u\in \R^{n_0}$ with $\q|u\w|$ sufficiently small, define the map
$$\Phi\q( u\w) = e^{u\cdot X_J}x_0.$$
In what follows, the constants 
%may depend on (for instance) upper bounds
%for a finite number of the $C^m$ norms of the vector fields $X_j$ and upper bounds for
% a finite number of the $C^m$ norms of $c_{j,k}^l$, but they 
can be chosen uniformly as the $X_j$ and $c_{j,k}^l$ vary over bounded
subsets of $C^{\infty}$. 
The constants do {\it not} depend
on a lower bound for, say, $\q|\det_{n_0\times n_0} X\q( x_0\w)\w|$.
Our ``uniform'' version of the theorem of Frobenius is:
\begin{thm}\label{ThmIntroFrob}
There exist $\eta, \xi\approx 1$ such that:
$$\B{X}{d}{x_0}{\xi} \subseteq \Phi\q( B_{n_0}\q( \eta\w)\w) \subseteq \B{X}{d}{x_0}{1}$$
and
\begin{itemize}
\item $\Phi:B_{n_0}\q( \eta\w) \rightarrow \B{X}{d}{x_0}{1}$ is one-to-one.
\item For all $u\in B_{n_0}\q( \eta\w)$, $\q|\det_{n_0\times n_0} d\Phi\q( u\w)\w|\approx \q|\det_{n_0\times n_0} X\q( x_0\w)\w|$.
\end{itemize}
Furthermore, if we let $Y_j$ be the pullback of $X_j$ via the map $\Phi$,
then the list of vector fields $Y_1,\ldots, Y_q$ satisfy good estimates.
See Theorem \ref{ThmMainUnitScale} for details.
\end{thm}

Let us now describe why Theorem \ref{ThmIntroFrob} can be viewed
as a version of the theorem of Frobenius (see Section \ref{SectionClassicalFrob} for further discussion on this point).
Indeed, our main assumption \eqref{EqnIntroAssumpFrob}
is exactly the main assumption of the theorem of
Frobenius.
Hence under the hypotheses of Theorem \ref{ThmIntroFrob}, the
vector fields $X_1,\ldots, X_q$ foliate $\Omega$ into leaves.
As mentioned in Remark \ref{RmkIntroLeaves}, $\B{X}{d}{x_0}{\xi}$
is an open neighborhood of $x_0$ on this leaf.
Moreover $\Phi:B_{n_0}\q( \eta\w)\rightarrow \B{X}{d}{x_0}{1}$
is one-to-one.  Thus, $\Phi$ can be considered as a coordinate
chart on the leaf in a neighborhood of $x_0$.  Hence for each point
$x_0\in \Omega$, Theorem \ref{ThmIntroFrob} yields a coordinate
chart near $x_0$ on the leaf passing through $x_0$.  In this way,
Theorem \ref{ThmIntroFrob} {\it implies} the classical theorem
of Frobenius.  The main point is that not only does Theorem \ref{ThmIntroFrob}
yield a coordinate chart, but it also allows one to take $\xi, \eta\approx 1$ and it gives good estimates on this
coordinate chart; estimates which do not follow from the standard
proofs of the theorem of Frobenius (see Remark \ref{RmkDoesntFollowFromFrobProofs}), nor from the methods of
\cite{NagelSteinWaingerBallsAndMetricsDefinedByVectorFields} (see the discussion
in Section \ref{SectionNSW}).

In Section \ref{SectionPastWork}, we discuss a number of previous,
related works, and relate our results to these works.
In Section \ref{SectionThmOfFrob} we state and prove a precise version
of Theorem \ref{ThmIntroFrob} in the special case when
$X_1\q( x_0\w), \ldots, X_q\q( x_0\w)$ are linearly independent,
and it is in this section that the main technicalities of the
paper lie.
We refer to this result as a uniform theorem of Frobenius.
The proof heavily uses methods from
Section 4 of \cite{TaoWrightLpImprovingBoundsForAverages} 
and methods from  \cite{NagelSteinWaingerBallsAndMetricsDefinedByVectorFields},
but these need to be significantly generalized to adapt them to our situation.
In Section \ref{SectionUnitScale} we use the results of Section
\ref{SectionThmOfFrob} to prove the more general version of
Theorem \ref{ThmIntroFrob} in the case when $X_1\q( x_0\w), \ldots, X_q\q( x_0\w)$ are not necessarily linearly independent.  We refer
to this as studying Carnot-Carath\'eodory balls 
``at the unit scale.''\footnote{Here we mean at the unit scale with respect
to the vector fields $X_j$.  Thus, if the $X_j$ are very small
(as is the case when $X_j=\delta^{d_j} W_j$, where
$\delta$ is small), then one can think of it as being at a small scale.
In addition, we could have equally well referred to this as
a theorem of Frobenius.  We chose this name, though, to
emphasize its role in the proof of Theorem \ref{ThmIntroMainThm}.}
In addition, we use these results to define smooth bump functions
supported on these balls, along the lines of those
used in \cite{NagelSteinDifferentiableControlMetrics}.
In Section \ref{SectionMultiParamBalls} we state and
prove the rigorous version of Theorem \ref{ThmIntroMainThm}.
In Section \ref{SectionMaximalFuncs}, we use the results
from Section \ref{SectionUnitOpsAtUnitScale} to
study multi-parameter maximal functions associated to a certain
subclass of our multi-parameter balls.  In Section
\ref{SectionMaximalFuncs} we also discuss compositions of certain
unit operators (see Corollary \ref{CorSchwartzKer}
and Section \ref{SectionStreet}), and in Section
\ref{SectionMetrics} we use these unit operators
to discuss the relationship between certain quasi-metrics
that arise.

From the discussion proceeding Theorem \ref{ThmIntroFrob}, it is clear
why Theorem \ref{ThmIntroFrob} implies Theorem \ref{ThmIntroMainThm},
provided one has appropriate control over the implicit constants in Theorem \ref{ThmIntroFrob}.
Hence a main aspect of this paper is to keep track of the
appropriate constants in Theorem \ref{ThmIntroFrob}.
At times, this will be quite technical.  In addition, we will state our
main results with only a finite amount of smoothness, further
complicating our notations.\footnote{While it does complicate notation,
working with only a explicit finite amount of smoothness does not
complicate our proof.}  
In the past, there has been some interest in results in the single
parameter case, using as low regularity as possible.  Even
in this single parameter context, our results are new in this
direction.  See Section \ref{SectionNonsmooth} for a discussion
of this.

In an effort to ease the notation in the paper, at the start
of many of the sections of this paper, we will define
a notion of ``admissible constants.''
These will be constants that only depend on certain parameters.
This notion of admissible constant may change from section
to section, but we will be explicit about what it means each time.
In addition, if $\kappa$ is another parameter, and we
say ``there exists an admissible constant $C=C\q( \kappa\w)$,''
we mean that $C$ is allowed to depend on everything an
admissible constant may depend on, and is also allowed
to depend on $\kappa$.
We use the notation $A\lesssim B$ to mean $A\leq CB$,
where $C$ is an admissible constant; so that, in particular,
the meaning of $\lesssim$ may change from section to section.
We use $A\approx B$ to mean $A\lesssim B$ and $B\lesssim A$.
In some sections, we will use different levels of smoothness
assumptions.  In these sections, we will also define a
notion of $m$-admissible constants, where $m\in \N$ denotes
the level of assumed smoothness.  We will write $A \lesssim_m B$
for $A\leq C B$, where $C$ is an $m$-admissible constant,
and we define $\approx_m$ in a similar manner.

We write $Q_n\q( \eta\w)$ to denote the unit ball in 
$\R^n$, centered at $0$, of radius $\eta$ in the $\q|\cdot\w|_\infty$ norm.
All functions in this paper are assumed to be real valued.
Given a, possibly not closed, set $U\subseteq \R^n$, we write:
\begin{equation*}
\Cjn{m}{U}{f} = \sup_{x\in U} \sum_{\q|\alpha\w|\leq m} \q|\partial_x^{\alpha} f\q(x\w)\w|.
\end{equation*}
Finally, $v_1,v_2\in \R^m$ are two vectors, we write $v_1\leq v_2$
to mean that the inequality holds for each coordinate.

\begin{rmk}
Throughout the paper we work on an open subset $\Omega\subseteq \R^n$,
endowed with Lebesgue measure.  At first glance, it might seem useful
to work more generally on a Riemannian manifold (where Lebesgue
measure is replaced by the volume element); and replace the set of vector fields $X_1,\ldots, X_q$ with a locally finitely generated distribution (endowed with an appropriate (multi-parameter) filtration taking the place of the formal degrees).  However, our results
are local in nature, and working in such a setting offers
no new generality and only serves to complicate notation.
\end{rmk}

%% file: pastwork.tex
In this section we discuss other results from the literature which
are related to the results in this paper.
In particular, we discuss the work of
\cite{NagelSteinWaingerBallsAndMetricsDefinedByVectorFields}
and the work in Section 4 of \cite{TaoWrightLpImprovingBoundsForAverages}.
Next, we discuss other results concerning Carnot-Carath\'eodory
balls in the case when the vector fields are not smooth.
In particular, we discuss the recent works
\cite{BarmantiBrandoliniPedroniNonsmoothHormander}
and
\cite{MontanariMorbidelliNonsmoothHormanderVectorFields}.
Third, as motivation for our study of ``unit operators'' (and maximal
functions)
in Section \ref{SectionMaximalFuncs}, we discuss the singular integrals from
 \cite{StreetAdvances}.
Finally, we discuss the classical theorem of Frobenius and make some
further remarks on how Theorem \ref{ThmIntroFrob} can be seen
as a ``uniform'' version.

%% file: nsw.tex
In this section, we discuss the main results of
\cite{NagelSteinWaingerBallsAndMetricsDefinedByVectorFields}.
In fact, their main results can be seen as a special case
of Theorem \ref{ThmIntroMainThm}, in the single-parameter
case ($\nu=1$).

We are given an open set $\Omega\subseteq \R^n$ and $C^\infty$ vector
fields $X_1,\ldots, X_q$ on $\Omega$, with corresponding
formal degrees $d_1,\ldots, d_q\in \q( 0,\infty\w)$.
\cite{NagelSteinWaingerBallsAndMetricsDefinedByVectorFields}
assumes two properties of the vector fields and formal degrees:
\begin{enumerate}
\item\label{ItemMainNSWAssump} There exist $c_{i,j}^k\in C^\infty$ such that
\begin{equation}\label{EqnIntroNSWIntegCond}
\q[X_i, X_j\w] = \sum_{d_k\leq d_i+d_j} c_{i,j}^k X_k.
\end{equation}
\item\label{ItemDetNSWAssump} The vector fields $X_1,\ldots, X_q$ span the tangent space at every point.
\end{enumerate}
In this context, Nagel, Stein, and Wainger prove Theorem \ref{ThmIntroMainThm}
(for a fixed compact set $K\Subset \Omega$).
Note that Item \ref{ItemMainNSWAssump} is a special case of
\eqref{EqnIntroMainAssump}.  Indeed, one may take
$$c_{i,j}^{k,\delta} = 
\begin{cases}
\delta^{d_i+d_j-d_k}c_{i,j}^k & \text{if } d_k\leq d_i+d_j,\\
0 & \text{otherwise}.
\end{cases}$$ 
The implicit constants are allowed to depend not only on
upper bounds for a finite number of the $C^m$ norms of the $X_j$ and the $c_{i,j}^k$
(as in Theorem \ref{ThmIntroMainThm}), but also a lower bound
for:
\begin{equation}\label{EqnNSWLowerBound}
\inf_{x\in K}\q|\det_{n\times n} X\q( x\w)\w|.
\end{equation}
This is the fundamental difference between the results
of \cite{NagelSteinWaingerBallsAndMetricsDefinedByVectorFields}
and Theorem \ref{ThmIntroMainThm}.

Indeed, using the connection between single-parameter and multi-parameter
balls discussed in Section \ref{SectionResults}, it is not hard to
see that Theorem \ref{ThmIntroMainThm} is essentially {\it equivalent}
to obtaining the results of \cite{NagelSteinWaingerBallsAndMetricsDefinedByVectorFields}
without allowing the constants to depend on a lower bound for
\eqref{EqnNSWLowerBound}.
Of course, if one does not allow the constants to depend on a lower
bound for \eqref{EqnNSWLowerBound}, one should also consider
the limiting result when the quantity in \eqref{EqnNSWLowerBound}
is equals $0$.  I.e., when the vector fields do not span
the tangent space at every point.
This is precisely the statement of Theorem \ref{ThmIntroMainThm}
in the single parameter case.

Use of a lower bound for \eqref{EqnNSWLowerBound} is essential
to the methods of \cite{NagelSteinWaingerBallsAndMetricsDefinedByVectorFields}.
It is used, for instance, every time the error term in the Campbell-Hausdorff
formula
is estimated.\footnote{See the appendix of
\cite{NagelSteinWaingerBallsAndMetricsDefinedByVectorFields}
for an introduction to the Campbell-Hausdorff formula.}
To explain this, we outline a proof of (a result similar to) Lemma 2.13
of \cite{NagelSteinWaingerBallsAndMetricsDefinedByVectorFields}.
We take the setting as above, and consider the map ($B_{q}\q( \eta\w) \rightarrow\Omega$,
for some small $\eta>0$):
$$\theta_{\delta}\q( s\w) = e^{s_1\delta^{d_1}X_1+\ldots+s_q\delta^{d_q}X_q}x_0.$$
Then one has:
$$ d\theta_\delta\q( \partial_{s_j}\w) = \sum_{j=1}^q c_{j}^{k,\delta} \delta^{d_k}X_k,$$
with $c_{j}^{k,\delta}$ bounded uniformly for $\delta>0$ small.
Indeed, the Campbell-Hausdorff formula allows one to compute the
Taylor series for $d\theta_\delta \q( \partial_{s_j}\w)$.
One has, for every $N>0$,
\begin{equation*}
\begin{split}
d\theta_\delta\q( \partial_{s_j}\w) = &\delta^{d_j}X_j + a_1 \q[s\cdot \delta^d X, \delta^{d_j}X_j\w] + a_2 \q[s\cdot \delta^d X, \q[s\cdot \delta^d X, \delta^{d_j}X_j\w] \w] + \ldots \\
&+ a_{N-1} \q\{\text{commutators of order }N-1\w\} + O\q(\q|\delta^d s\w|^N\w),
\end{split}
\end{equation*}
where the $a_j$ are constants and $\delta^d s = \q(\delta^{d_1} s_1,\ldots, \delta^{d_q} s_q\w)$.  The first $N$ terms are of the desired form
by \eqref{EqnIntroNSWIntegCond} (or more generally, 
\eqref{EqnIntroMainAssump}).  Thus, the goal is to see
that $O\q(\q|\delta^d s\w|^N\w)$ is of the desired form.
This can be seen directly, by taking $N$ so large that
$N\min_j\q\{d_j\w\}\geq \max_j \q\{d_j\w\}$, and using 
the lower bound for \eqref{EqnNSWLowerBound}.  However,
this procedure does not work in the multi-parameter
situation.  Indeed, consider the two-parameter situation.
In the case when $\delta_1<<\delta_2$, then the best one can
say about the error term $O\q(\q|\delta^d s\w|^N\w)$ is that
it is bounded by as large a power of $\delta_2$ as we like (by
taking $N$ large).  However, we would need it to be bounded by
a large power of $\delta_1$ to generalize the above proof.
It turns out that, even in the multi-parameter situation,
the error term {\it is} of the desired form.  This follows
{\it a fortiori} from the results of this paper.
Because of this, one can use the results of this paper to apply
the proofs in \cite{NagelSteinWaingerBallsAndMetricsDefinedByVectorFields}
to the multi-parameter situation.  However, since the results
in \cite{NagelSteinWaingerBallsAndMetricsDefinedByVectorFields} follow
from the results in this paper, this idea does not
improve the main results of this paper.  This idea
does have some uses, though:  one can often ``lift'' results
from the single-parameter setting to the multi-parameter setting by
using the results from this paper.  This is discussed
in more detail in Section \ref{SectionLiftingSingleParam}.

%\begin{rmk}
%Use of a lower bound for \eqref{EqnNSWLowerBound} is essential
%to the methods of \cite{NagelSteinWaingerBallsAndMetricsDefinedByVectorFields}.
%It is used, for instance, every time the error term in the Campbell-Hausdorff
%is estimated.
%\end{rmk}

At first glance, one might think that the proper generalization
of \eqref{EqnIntroNSWIntegCond}
to the multi-parameter situation would be:
\begin{equation}\label{EqnNSWGuessIntCond}
\q[X_i,X_j\w]=\sum_{d_k\leq d_i+d_j} c_{i,j}^k X_k,
\end{equation}
where $d_j\in \q[0,\infty\w)^\nu$ and the inequality $d_k\leq d_i+d_j$ is meant coordinatewise.
Just as before this is a special case of
\eqref{EqnIntroMainAssump}, and one may take
$$c_{i,j}^{k,\delta}=
\begin{cases}\delta^{d_i+d_j-d_k} c_{i,j}^k & \text{if }d_k\leq d_i+d_j,\\
0 & \text{otherwise}.
\end{cases}$$
However, unlike in the single-parameter case, \eqref{EqnNSWGuessIntCond}
does not encapsulate a large fraction of the interesting examples.
This is explained in more detail in
Section \ref{SectionControlEveryScale}.

%% file: taowright.tex
In this section, we discuss the work in Section 4
of \cite{TaoWrightLpImprovingBoundsForAverages}
on ``weakly-comparable'' Carnot-Carath\'eodory balls.
While the results discussed in this section do not follow
from Theorem \ref{ThmIntroMainThm}, they do follow
from the more general Theorem \ref{ThmMultiBallsAtPoint}--this
is discussed in Section \ref{SectionWeaklyCompBalls}.

To understand these results, we must first understand the main
motivating example of
\cite{NagelSteinWaingerBallsAndMetricsDefinedByVectorFields}.
Suppose we are given $C^\infty$ vector fields $W_1,\ldots, W_r$ on an open
subset $\Omega\subseteq \R^n$.  Suppose further that
these vector fields satisfy H\"ormander's condition:  i.e.,
$W_1,\ldots,W_r$ along with their commutators up to some fixed finite
order (say, up to order $m\in \N$)
span the tangent space at every point.
We assign to each vector field $W_j$ the formal degree $1$.  We assign
to each commutator $\q[W_i, W_j\w]$ the formal degree $2$.  
We assign to each commutator $\q[W_i, \q[W_j,W_k\w]\w]$ the formal
degree $3$.
We continue
this process up to degree $m$, and we obtain a list of vector fields
with one parameter formal degrees $\q( X_1,d_1\w),\ldots, \q( X_q,d_q\w)$.
As usual, we denote this list by $\q( X,d\w)$.  It is easy to check
that this list of vector fields satisfies the assumptions
in Section \ref{SectionNSW}.
It is also shown in \cite{NagelSteinWaingerBallsAndMetricsDefinedByVectorFields}
that the single-parameter balls
$$B_{\delta W_1,\ldots, \delta W_r}\q( x_0\w)$$
are comparable to the single-parameter balls
$$\B{X}{d}{x_0}{\delta}.$$
Because of this, one can use the results of
\cite{NagelSteinWaingerBallsAndMetricsDefinedByVectorFields}
to study the balls
$$B_{\delta W_1,\ldots, \delta W_r}\q( x_0\w).$$

We now turn to discussing two-parameter weakly-comparable balls.
The restriction to two-parameters is not essential, see
Section \ref{SectionWeaklyCompBalls}.
We suppose again that we are given a list of $C^\infty$
vector fields satisfying H\"ormander's condition, $W_1,\ldots, W_r$.
We now separate this list into two lists:
$$W_1',\ldots, W_{r_1}', \quad W_1'',\ldots, W_{r_2}'',$$
so that the two lists together satisfy H\"ormander's condition,
but they may not satisfy H\"ormander's condition separately.
Suppose we wish to study the two-parameter balls given by
$$B_{\delta_1 W', \delta_2 W''}\q( x_0\w),$$
where $\delta_1,\delta_2\in \q(0,1\w)$ are small.  Thus, when
$\delta_1=\delta_2$ this reduces to the single-parameter case discussed
above.
It is natural to wish for an estimate of the form
\begin{equation}\label{EqnIntroWishFor}
\Vol{B_{2\delta_1 W', 2\delta_2 W''}\q( x_0\w)} \lesssim \Vol{B_{\delta_1 W', \delta_2 W''}\q(x_0\w)}.
\end{equation}
Unfortunately, \eqref{EqnIntroWishFor} does not hold in general (See Example \ref{ExampleWeaklyCompNecessary}).
To obtain \eqref{EqnIntroWishFor}, there are three options:
\begin{enumerate}
\item We could restrict the vector fields we consider.  This is
the perspective taken up in Theorem \ref{ThmIntroMainThm}.
\item We could restrict the form of $\delta=\q( \delta_1,\delta_2\w)$.
This is the perspective taken up in Section 4
of \cite{TaoWrightLpImprovingBoundsForAverages}.
\item We could, more generally, do a combination of the above two methods.
This is taken up in Theorem \ref{ThmMultiBallsAtPoint}.
\end{enumerate}
We briefly discuss the first two methods, and refer the reader to
Theorem \ref{ThmMultiBallsAtPoint} for the third.

The first method is quite straight forward given the
Theorem \ref{ThmIntroMainThm}.
We assign to each of the vector fields $W_1',\ldots, W_{r_1}'$
the formal degree $\q( 1,0\w)$.  We assign to each of the vector
fields $W_1'',\ldots, W_{r_1}''$ the formal degree $\q( 0,1\w)$.
If we have assigned a vector field $Z_1$ the formal degree $d_1$
and $Z_2$ the formal degree $d_2$, we assign to the commutator
$\q[ Z_1,Z_2\w]$ the formal degree $d_1+d_2$.
One then uses this procedure to take commutators of the $W'$ and $W''$
to some large finite order, thereby yielding a list of vector fields
with two-parameter formal degrees $\q( X_1,d_1\w), \ldots, \q( X_q,d_q\w)$.
The restriction we put on the vector fields $W'$ and $W''$ is
merely that this list of vector fields satisfies the conditions
of Theorem \ref{ThmIntroMainThm}.\footnote{It 
is a consequence of the results in Section \ref{SectionControlEveryScale}
that the balls one obtains in this manner are essentially independent
of the order of commutators one takes.  That is, if the vector fields $\q( X,d\w)$ were obtained by taking commutators up to order $M$, and if $\q(X,d\w)$ satisfies the assumptions of 
Theorem \ref{ThmIntroMainThm}, then the vector fields obtained by taking
commutators up to order $M+1$ also satisfy the assumptions of Theorem \ref{ThmIntroMainThm}, and yield comparable balls.}
One can then apply Theorem \ref{ThmIntroMainThm} to study these balls.

For the second method, Tao and Wright noted that one does not
have to restrict the vector fields one considers, provided one
restricts attention to $\delta$ which are ``weakly-comparable.''
To define this notion, fix large constants $\kappa, N$.  We then
restrict our attention to $\delta=\q( \delta_1, \delta_2\w)$
that satisfy:
$$\frac{1}{\kappa} \delta_1^N \leq \delta_2 \leq \kappa \delta_1^{\frac{1}{N}}.$$
In this case, one can prove develop a very satisfactory
theory of the balls with these radii.  In particular,
one has \eqref{EqnIntroWishFor}.  See
Section \ref{SectionWeaklyCompBalls} for more details.

Despite the fact, as is mentioned in
\cite{TaoWrightLpImprovingBoundsForAverages},
that the proofs from
\cite{NagelSteinWaingerBallsAndMetricsDefinedByVectorFields}
generalize to show everything they needed, Tao and Wright
put forth another proof method.
That these methods can be rephrased, generalized, and combined
with the methods of \cite{NagelSteinWaingerBallsAndMetricsDefinedByVectorFields}
and classical methods to prove more general results is one
of the main points of this paper.

\begin{rmk}
One point we have skipped over in this section is to compare the balls
$B_{\delta W}\q( x_0\w)$ to the balls $\B{X}{d}{x_0}{\delta}$ in
the multi-parameter situation (they
turn out to be comparable).
This is taken up in Section \ref{SectionLiftingSingleParam}.
\end{rmk}

%% file: nonsmooth.tex
The results in \cite{NagelSteinWaingerBallsAndMetricsDefinedByVectorFields}
were stated only for $C^\infty$ vector fields.
However, it is clear from their work that there is a finite
number $M$ (depending on various quantities) such that one
need only consider $C^M$ vector fields.
Unfortunately, this $M$ is quite large.
This is due to the fact that the uses of the Campbell-Hausdorff
formula in \cite{NagelSteinWaingerBallsAndMetricsDefinedByVectorFields}
require using very high order Taylor approximations of many
of the functions involved.

Much work has been done, in this single-parameter setting, to reduce the required regularity
in the results of \cite{NagelSteinWaingerBallsAndMetricsDefinedByVectorFields}.
Quite recently, and independently of this paper, two
works have made great strides on this problem:
the work of Barmanti, Brandolini, and Pedroni
\cite{BarmantiBrandoliniPedroniNonsmoothHormander}
and
the work of Montanari and Morbidelli
\cite{MontanariMorbidelliNonsmoothHormanderVectorFields}.
We refer the reader to these to works for the long list
of works that proceeded them and for a description
of applications for such results.

To describe these results, suppose we are given vector fields
$W_1,\ldots, W_r$ satisfying H\"ormander's condition at step $s\geq 1$.
That is, $W_1,\ldots, W_r$ along with their commutators up to order $s$
span the tangent space at each point (see the discussion at the
start of Section \ref{SectionTW}).
Then, \cite{BarmantiBrandoliniPedroniNonsmoothHormander}
shows that one can recreate much of the theory
of \cite{NagelSteinWaingerBallsAndMetricsDefinedByVectorFields}
provided one assumes the vector fields are $C^{s-1}$.
\cite{MontanariMorbidelliNonsmoothHormanderVectorFields}
achieves the same thing assuming the vector fields lie in a 
space that is between $C^{s-2,1}$ and $C^{s-1,1}$ (see
\cite{MontanariMorbidelliNonsmoothHormanderVectorFields} for
a precise statement).

The regularity assumptions in this paper are
incomparable to those discussed above.\footnote{Of course, our results
also apply to the multi-parameter situation, which is not true
of \cite{BarmantiBrandoliniPedroniNonsmoothHormander,MontanariMorbidelliNonsmoothHormanderVectorFields}.}
As far as {\it isotropic} estimates go, the work of
\cite{BarmantiBrandoliniPedroniNonsmoothHormander,MontanariMorbidelliNonsmoothHormanderVectorFields}
requires less regularity than ours.  However, our estimates
are non-isotropic in nature.  To understand this, use the vector fields
$W_1,\ldots, W_r$ to generate a list of vector fields with single-parameter
formal degrees $\q( X_1,d_1\w),\ldots, \q( X_q,d_q\w)$ as in
the start of Section \ref{SectionTW}.
We then assume that each $X_j$ is $C^2$, and assume a non-isotropic
estimate on the $c_{i,j}^k$ which is weaker than assuming the $c_{i,j}^k$ are
in $C^2$ (here, the $c_{i,j}^k$ are as in \eqref{EqnIntroNSWIntegCond}).
Note that if one were to replace this with an isotropic estimate,
we would require that $W_j$ be in $C^{s+2}$, which is much worse
than the results in \cite{BarmantiBrandoliniPedroniNonsmoothHormander,MontanariMorbidelliNonsmoothHormanderVectorFields}.
However, the point here is that we do not need to take derivatives
of $W_j$ in {\it every} direction up to order $s+2$, but instead we
can mostly restrict our attention to derivatives that arise
from taking commutators.

It is likely that the regularity required in this paper is not minimal--even
for the methods we use.
Indeed, we often show that a subset of $C^1$ is precompact by showing
that it is bounded in $C^2$, leaving much room for improvement.
Improving this would require an even more detailed study of the
various ODEs that arise than is already undertaken in this paper,
and this would take us quite afield of the main purpose of this paper (to
understand the multi-parameter situation).
Ideally, one would like to unify the non-isotropic estimates
in this paper with the isotropic estimates of
\cite{BarmantiBrandoliniPedroniNonsmoothHormander,MontanariMorbidelliNonsmoothHormanderVectorFields}, we do not attempt to do so here but we hope that
the results in this paper will help to motivate future work in this direction.

%% file: streetsing.tex
In \cite{StreetAdvances}, an algebra of singular integral operators 
was developed which contained both the left and right invariant Calder\'on-Zygmund
operators on a stratified Lie group (see \cite{StreetAdvances} for
a precise statement).
In this section, we outline the main technical estimate
that was key to the work of \cite{StreetAdvances}.
One of the main results of Section \ref{SectionMaximalFuncs}
is a generalization\footnote{See Corollary \ref{CorSchwartzKer} for the statement of this generalization.} of this estimate, and we refer to the operators
that come into play in this estimate as ``unit operators.''

For the purposes of this section, we discuss only the case
of the three dimensional Heisenberg group $\Ho$, though
all of the results discussed here hold more generally
on stratified Lie groups.
As a manifold $\Ho$ is diffeomorphic to $\R^3$.  For an
introduction to $\Ho$, the reader may consult
Chapter XII of \cite{SteinHarmonicAnalysis}.
If we write $\q( x,y,t\w)\in \R^3$ for coordinates on $\Ho$,
then the group law is given by:
$$\q(x,y,t \w)\q( x',y',t'\w) = \q( x+x', y+y', t+t' +2\q(yx'-xy'\w) \w).$$
With this group law, $\Ho$ is a three dimensional, nilpotent Lie group.
As such, it has a three dimensional Lie algebra.  The left invariant
vector fields are spanned by:
$$X_L = \partial_x + 2y\partial_t, \quad Y_L=\partial_y-2x\partial_t, \quad T=\partial_t.$$
The right invariant vector fields are spanned by:
$$X_R = \partial_x - 2y\partial_t, \quad Y_R=\partial_y+2x\partial_t, \quad T=\partial_t.$$
Note that we have:
$$\q[X_L,Y_L\w]= -4T,\quad \q[X_R,Y_R\w]=4T,$$
and $T$ commutes with all of the vector fields.  Moreover, the left
invariant vector fields commute with the right invariant vector fields.

Using these vector fields, we may create a two-parameter list of vector fields
given by:
$$\q( X_L,\q( 1,0\w)\w), \q( Y_L, \q( 1,0\w)\w), \q( T,\q( 2,0\w)\w), \q( X_R,\q( 0,1\w)\w), \q( Y_R,\q( 0,1\w)\w), \q( T,\q( 0,2\w)\w),$$
and we denote this list by $\q( X,d\w)$.
It is easy to see that $\q( X,d\w)$ satisfies the assumptions
of Theorem \ref{ThmIntroMainThm}; we therefore obtain a theory
of the two-parameter Carnot-Carath\'eodory balls
$\B{X}{d}{x_0}{\delta}$.

Let $\chi$ be the characteristic function of the unit ball
in $\R^3$, and for $r\in \q( 0,\infty\w)$, define
$$\chi_r\q( x,y,t\w) = r^{4}\chi\q( rx,ry,r^2t\w).$$
It is easy to see that $\chi_r\q( \xi^{-1} \zeta\w)$ is supported
for $\xi$ essentially\footnote{By this we mean it is supported in a comparable ball.} 
in $\B{X}{d}{\zeta}{\q(\frac{1}{r},0\w)}$.
Moreover, it is bounded by a constant times $\Vol{\B{X}{d}{\zeta}{\q(\frac{1}{r},0\w)}}^{-1}$ (and these results are sharp).
For $\chi_r\q(\zeta\xi^{-1}\w)$, the same is true, but one must use
the radius $\q( 0,\frac{1}{r}\w)$ instead of $\q( \frac{1}{r},0\w)$.

Define a left invariant operator and a right invariant operator by:
$$\OpL{\chi_r}:f\mapsto f*\chi_r,\quad \OpR{\chi_r}:f\mapsto \chi_r*f.$$
A key ingredient of the theory in \cite{StreetAdvances} was a
study of the Schwartz kernel of the operator $\OpL{\chi_{r_1}}\OpR{\chi_{r_2}}$
(see Section 5.1 of \cite{StreetAdvances}).
Let $K_{r_1,r_2}\q( \zeta, \xi\w)$ denote this Schwartz kernel.
The results in \cite{StreetAdvances} show:
\begin{itemize}
\item $K_{r_1,r_2}\q( \zeta,\xi\w)$ is supported essentially in $\B{X}{d}{\zeta}{\q(\frac{1}{r_1},\frac{1}{r_2}\w)}$.
\item $K_{r_1,r_2}\q( \zeta, \xi\w)\lesssim \Vol{\B{X}{d}{\zeta}{\q(\frac{1}{r_1},\frac{1}{r_2}\w)}}^{-1}$.
\item The above two results are sharp.  In particular, there is an $\eta>0$ such that for $\xi\in \B{X}{d}{\zeta}{\q(\frac{\eta}{r_1}, \frac{\eta}{r_2}\w)}$, we have
$$K_{r_1,r_2}\q( \zeta,\xi\w)\approx \Vol{\B{X}{d}{\zeta}{\q(\frac{1}{r_1}, \frac{1}{r_2} \w)}}^{-1}.$$
\end{itemize}

Using these results, one can study maximal operators.  Indeed,
define three maximal operators:
\begin{equation*}
\begin{split}
\sM f\q( \zeta\w) &= \sup_{\delta_1,\delta_2>0} \frac{1}{\Vol{\B{X}{d}{\zeta}{\q( \delta_1,\delta_2\w)}}} \int_{\B{X}{d}{\zeta}{\q(\delta_1,\delta_2\w)}} \q|f\q( \xi\w) \w| \: d\xi,\\
\sM_L f\q( \zeta\w) &= \sup_{\delta>0} \frac{1}{\Vol{\B{X}{d}{\zeta}{\q( \delta,0\w)}}} \int_{\B{X}{d}{\zeta}{\q(\delta,0\w)}} \q|f\q( \xi\w) \w| \: d\xi,\\
\sM_R f\q( \zeta\w) &= \sup_{\delta>0} \frac{1}{\Vol{\B{X}{d}{\zeta}{\q( 0,\delta\w)}}} \int_{\B{X}{d}{\zeta}{\q(0,\delta\w)}} \q|f\q( \xi\w) \w| \: d\xi.
\end{split}
\end{equation*}
The results above show
$$\sM f \lesssim \sM_L\sM_R f.$$
However, it is well-known that the one-parameter maximal functions
$\sM_L$ and $\sM_R$ are bounded on $L^p$ ($1<p\leq \infty$); due
to the fact that $\B{X}{d}{\zeta}{\q( \cdot, 0\w)}$ and $\B{X}{d}{\zeta}{\q( 0,\cdot\w)}$ give rise to spaces of homogeneous type.
It follows, then, that $\sM$ is also bounded on $L^p$ ($1<p\leq \infty$).

The goal of Section \ref{SectionMaximalFuncs} is to see how far
this proof (in its entirety) can be generalized.
It was used heavily in \cite{StreetAdvances} that the left invariant
vector fields commuted with the right invariant vector fields.
In Section \ref{SectionMaximalFuncs} we will see that we do not
need the relevant vector fields to commute, but can instead just assume
that they ``almost commute.''  This is made precise
in Section \ref{SectionMaximalFuncs}.

It is extremely likely maximal results hold for a larger
class of our multi-parameter balls than what is
shown in Section \ref{SectionMaximalFuncs}--but the study
of unit operators seems very tied to the (rather strong)
assumptions in Section \ref{SectionMaximalFuncs}.
We content ourselves, in this paper, with studying maximal operators
under these hypotheses.  It would be interesting to generalize these
results further.

%% file: classicalfrob.tex
In this section, we remind the reader of the statement of the
theorem of Frobenius.  We keep the exposition brief since
we use the classical theorem of Frobenius only tangentially
in this paper, and this section is more to fix terminology.  Suppose $M$ is a connected manifold, and
$X_1,\ldots, X_q$ are $C^\infty$ vector fields on $M$.
Suppose, further, that for each $i,j$ there exist
$C^\infty$ functions $c_{i,j}^k$ such that:
\begin{equation}\label{EqnAppendFrobInteg}
\q[X_i,X_j\w]=\sum_{k} c_{i,j}^k X_k.
\end{equation}
Conditions like (\ref{EqnAppendFrobInteg}) are referred to
as ``integrability conditions.''  In this case, we have the classical theorem of Frobenius:
\begin{thm}\label{ThmAppendFrob}
For each $x\in M$, there exists a unique, 
%closed, connected 
maximal, connected, injectively immersed
submanifold $L\subseteq M$
such that:
\begin{itemize}
\item $x\in L$,
\item For each $y\in L$, $T_y L=\Span{X_1\q(y\w),\ldots, X_q\q(y\w)}$.
\end{itemize}
$L$ is called a ``leaf.''
\end{thm}

\begin{rmk}
Often, one sees an additional assumption in Theorem \ref{ThmAppendFrob}.
Namely, that $\dim \Span{X_1,\ldots, X_q}$ is constant.  This assumption
is not necessary, and the usual proofs (for instance, the one in \cite{ChevalleyTheoryOfLieGroupsI}) give the stronger result
in Theorem \ref{ThmAppendFrob}.  This was noted in \cite{HermanTheDifferentialGeometryOfFoliations}.
\end{rmk}

\begin{rmk}
Let $\sD$ be a $C^\infty$ module of vector fields on an open set $\Omega\subseteq \R^n$.  We call $\sD$ a (generalized) distribution.  Suppose that $\sD$ satisfies two conditions:
\begin{enumerate}
\item $\sD$ is involutive.  That is, if $X,Y\in \sD$, then $\q[X,Y\w]\in \sD$.
\item $\sD$ is locally finitely generated as a $C^\infty$ module.  That is, for each $x\in \Omega$,
there is a neighborhood $U$ containing $x$ such that there exist a finite
set of vector fields $X_1,\ldots, X_q\in \sD$ such that every $Y\in \sD$,
when restricted to $U$,
can be written as a linear combination (with coefficients in $C^\infty$)
of $X_1,\ldots, X_q$ on $U$.
\end{enumerate}
Note, under the above hypotheses, $X_1,\ldots, X_q$ satisfy
\eqref{EqnAppendFrobInteg} on $U$ (since $\q[X_i, X_j\w]\in \sD$).
Thus, one may apply Theorem \ref{ThmAppendFrob} to foliate
$\Omega$ into leaves, with each leaf $L$ satisfying $T_y L=\sD_y$, $\forall y\in L$.
Often, the Frobenius theorem is stated in terms of
such an ``involutive disribution which is locally finitely generated as a $C^\infty$ module,''
instead of stated in terms of an explicit
choice of generators as we have done in Theorem \ref{ThmAppendFrob}.
The reason we have chosen to state the theorem with an explicit
choice of generators is that we will need to discuss
how various constants depend on the generators.
\end{rmk}

\begin{rmk}
Above we have stated the result assuming the vector fields are
$C^\infty$.  In fact, an analogous result (using, again, the usual proofs) holds only assuming
that the vector fields are $C^1$.  In fact, there are even
results when the vector fields are assumed to be merely Lipschitz (see \cite{RampazzoFrobeniusTypeTheoremsForLipschitzDist}).
However $C^1$ will be sufficient for our purposes (and most of the
applications we have in mind require only $C^\infty$).
\end{rmk}

We close this section with a discussion of the relationship between
Theorem \ref{ThmAppendFrob} and Theorem \ref{ThmIntroFrob}.
As we mentioned before, Theorem \ref{ThmIntroFrob} implies
Theorem \ref{ThmAppendFrob}.  To understand the philosophy behind
Theorem \ref{ThmIntroFrob}, let $\sInp$ be an index set, and suppose
for each $\alpha\in \sInp$ we are given $C^\infty$ vector fields
$X_1^\alpha, \ldots, X_q^\alpha$ on a fixed open set $\Omega$.  Here,
both $q$ and $\Omega$ are independent of $\alpha$.  Suppose
further that for every $\alpha\in \sInp$ we have,
$$\q[X_i^\alpha, X_j^\alpha\w]=\sum_{k} c_{i,j}^{k,\alpha} X_k^\alpha.$$
Suppose, finally, that as $\alpha$ varies over $\sInp$, $X_j^{\alpha}$ and
$c_{i,j}^{k,\alpha}$ vary over bounded (and therefore pre-compact) subsets of $C^\infty$.
Since Theorem \ref{ThmAppendFrob} applies for each $\alpha\in \sInp$, one
might hope that it applies uniformly\footnote{We mean uniformly
in the sense that the coordinate charts which define the leaves
can be chosen to satisfy good estimates which are uniform in $\alpha$.}
for $\alpha\in \sInp$.
Indeed, this is the case, and is essentially the statement
of Theorem \ref{ThmIntroFrob}.  Hence, Theorem \ref{ThmIntroFrob}
may be informally restated as saying that the theorem of
Frobenius holds ``uniformly on compact sets'' in the above sense.

As it turns out, the classical proofs of Theorem \ref{ThmAppendFrob}
do not work uniformly in $\alpha$ in the above sense (this is
discussed in Remark \ref{RmkDoesntFollowFromFrobProofs}).
If we fix $x_0\in \Omega$ and define $n_0^\alpha=\dim\Span{X_1^{\alpha}\q( x_0\w),\ldots, X_q^{\alpha}\q( x_0\w)}$,
then the classical proofs also depend on a lower bound for
$$\q|\det_{n_0^\alpha\times n_0^\alpha} X^\alpha\q( x_0\w)\w|,$$
which may not be bounded below uniformly for $\alpha\in \sInp$.\footnote{It
is not a coincidence that the failure of the classical
proofs of the theorem of Frobenius to be uniform in an appropriate sense
lies in the use of a lower bound of a determinant, just as in the
work of Nagel, Stein, and Wainger (see Section \ref{SectionNSW}).  Indeed,
these two issues are closely related.}

There is another way to view Theorem \ref{ThmIntroFrob}
in relation to Theorem \ref{ThmAppendFrob}.
Let $X_1,\ldots, X_q$ be $C^\infty$ vector fields
satisfying
\eqref{EqnAppendFrobInteg}.
Notice we have not assumed
that $n_0\q( x\w)=\dim \Span{X_1\q( x\w),\ldots, X_q\q( x\w)}$
is constant in $x$.  
The foliation associated to the involutive
distribution generated by $X_1,\ldots, X_q$
is called ``singular'' if $n_0\q( x\w)$
is not constant in $x$; and if $n_0\q( x\w)$
is not constant near a point $x_0$ then
$x_0$ is called a singular point.

In the classic proofs of Theorem \ref{ThmAppendFrob},
the coordinate charts defining the leaves degenerate
as one approaches a singular point.
Theorem \ref{ThmIntroFrob} avoids this.
This is an essential point
in Section \ref{SectionHomogeneousFoliations}.

%% file: defn2.tex
Fix, for the rest of the paper, a connected open set $\Omega\subseteq \R^n$.
Suppose we are given a list of $C^1$ vector fields $X_1,\ldots, X_q$
defined on $\Omega$, and let $X$ denote this list.
As mentioned in Section \ref{SectionResults}, we will often identify
this list with the $n\times q$ matrix whose columns are given
by the vector fields $X_1,\ldots, X_q$.
In addition, we will define (when it makes sense) $X^\alpha$, where
$\alpha$ is an ordered multi-index, in the usual way.\footnote{For instance,
if $\alpha$ were the list $\q( 1,2,1,3\w)$, then $X^\alpha = X_1 X_2X_1X_3$
and $\q|\alpha\w|=4$, the length of the list.}
Thus, $X^\alpha$
is an $\q|\alpha\w|$th order partial differential operator.
In the introduction, we defined 
the Carnot-Carath\'eodory ball of unit radius
centered at $x_0\in \Omega$.  We denoted this ball by $B_X\q( x_0\w)$.

It will often be convenient to assume that $B_X\q( x_0\w)$ lies
``inside'' of $\Omega$.  More precisely, we make the following
definition:
\begin{defn}\label{DefnCondNoDeg}
Given $x_0\in \Omega$, we say $X$ satisfies $\sC\q(x_0\w)$
if for every $a=\q( a_1,\ldots, a_q\w)\in \q(L^\infty\q(\q[0,1\w] \w)\w)^q$,
with:
$$\q\|\q|a\w|\w\|_{L^\infty\q(\q[0,1\w]\w)}=\q\|\q(\sum_{j=1}^q\q|a_j\w|^2\w)^{\frac{1}{2}}\w\|_{L^\infty\q(\q[0,1\w]\w)}<1,$$
there exists a solution
 $\gamma:[0,1]\rightarrow \Omega$ to the ODE:
$$\gamma'\q( t\w) =\sum_{j=1}^q a_j\q( t\w) X_j\q( \gamma\q( t\w)\w),\quad \gamma\q(0\w) =x_0.$$
Note, by Gronwall's inequality, when this solution exists, it is unique.
\end{defn}

As in the introduction, to define Carnot-Carath\'eodory balls
of (possibly multi-parameter) radii, we assign to each vector
field $X_j$ a formal degree $0\ne d_j\in \q[0,\infty\w)^\nu$.
Here $\nu\in \N$ is a fixed number, independent of $j$, representing the number of parameters.
We denote the list $\q( X_1,d_1\w),\ldots, \q( X_q,d_q\w)$ by $\q( X,d\w)$.
In the introduction, we defined (for $\delta\in \q[0,\infty\w)^\nu$)
the list $\delta^d X$ to be the list of vector fields $\delta^{d_1} X_1,\ldots, \delta^{d_q} X_q$.  Then, we defined the multi-parameter Carnot-Carath\'eodory
ball $\B{X}{d}{x_0}{\delta}:=B_{\delta X}\q( x_0\w)$.
Just as in Definition \ref{DefnCondNoDeg} it will often be useful
to assume $\B{X}{d}{x_0}{\delta}$ lies ``inside'' of $\Omega$,
and so we make the following definition:
\begin{defn}
Given $x_0\in \Omega$ and $\delta\in \q[0,\infty\w)^\nu$, we say
$\q( X,d\w)$ satisfies $\sC\q( x_0,\delta\w)$ if $\delta^d X$
satisfies $\sC\q( x_0\w)$.
\end{defn}

In addition to the balls $\B{X}{d}{x_0}{\delta}$ it will be useful
to define some smaller balls.  Given $x_0\in \Omega$ and $\delta\in \q[0,\infty\w)^\nu$,
we define
$$\Bt{X}{d}{x_0}{\delta} = \q\{y\in \Omega: \exists a\in \R^q, \q|a\w|\leq 1, y=\exp\q(a\cdot \delta^d X\w)x_0\w\}.$$
Note that $\Bt{X}{d}{x_0}{\delta}\subseteq \B{X}{d}{x_0}{\delta}$.

Given a list of vector fields along with formal degrees $\q( X,d\w)$
and $J=\q( j_1,\ldots, j_{n_0}\w)\in \sI{n_0}{q}$, we defined
in the introduction the list of vector fields with formal degrees
$\q( X,d\w)_J$ and the list of vector fields $X_J$.  Namely,
$\q( X,d\w)_J$ is the list $\q( X_{j_1},d_{j_1}\w), \ldots, \q( X_{j_{n_0}}, d_{j_{n_0}}\w)$ and $X_J$ is the list $X_{j_1},\ldots, X_{j_{n_0}}$,
while $d_J$ is the list $d_{j_1},\ldots, d_{j_{n_0}}$.

Note that if $\q( X,d\w)$ satisfies $\sC\q(x_0,\delta\w)$, then so
does $\q( X,d\w)_J$.  In addition, we have,
$$\Bsub{X}{d}{J}{x_0}{\delta}\subseteq \B{X}{d}{x_0}{\delta}, \quad \Btsub{X}{d}{J}{x_0}{\delta}\subseteq \Bt{X}{d}{x_0}{\delta}.$$

Often, it will be convenient for our estimates to state
the definition of $\B{X}{d}{x_0}{\delta}$ in a slightly different way.
Thus, given the formal degrees $d_1,\ldots, d_q$ and given
a $a=\q(a_1,\ldots, a_q\w)\in \R^q$, $\delta\in \q[0,\infty\w)^\nu$, we define:
$$\delta^{d} a = \q(\delta^{d_1}a_1,\ldots, \delta^{d_q}a_q \w), $$
$$ \delta^{-d} a =\q(\delta^{-d_1}a_1,\ldots, \delta^{-d_q}a_q\w).$$
Then we have:
\begin{equation*}
\begin{split}
\B{X}{d}{x_0}{\delta} = \big\{y\in \Omega : &\exists \gamma:\q[0,1\w]\rightarrow \Omega, \gamma\q( 0\w)=x_0, \gamma\q( 1\w) =y\\
&\gamma'\q(t\w)=a\q(t\w)\cdot X\q(\gamma\q(t\w)\w), \q\|\q|\delta^{-d} a\w|\w\|_{L^\infty\q(\q[0,1\w]\w)}<1 \big\}.
\end{split}
\end{equation*}

%% file: frob.tex
In this section, we present a uniform version of the theorem of Frobenius:
the special case of Theorem \ref{ThmIntroFrob} when the
vector fields are assumed to be linearly independent.
%To do so, we will rely heavily on the methods in
%Section 4 of \cite{TaoWrightLpImprovingBoundsForAverages}
%and on the methods in \cite{NagelSteinWaingerBallsAndMetricsDefinedByVectorFields}.
The work in this section was heavily influenced by the methods in
Section 4 of \cite{TaoWrightLpImprovingBoundsForAverages} and
those in \cite{NagelSteinWaingerBallsAndMetricsDefinedByVectorFields}.
In fact, a result similar to a special case of Theorem
\ref{ThmMainFrobThm} is contained in \cite{TaoWrightLpImprovingBoundsForAverages},
though the result there is stated somewhat differently
(see Section \ref{SectionWeaklyCompBalls} for a discussion of their results).
%We will not be explicit each time we use an idea from one of these papers,
%but instead remark that a result similar to a special case of Theorem \ref{ThmMainFrobThm}
%is contained in \cite{TaoWrightLpImprovingBoundsForAverages}, though the result there is
%stated somewhat differently (see Section \ref{SectionWeaklyCompBalls} for a discussion of their results).  
Our goal, in this section, is to rephrase
and generalize the proof methods from these two papers to suit our needs.
%We will be explicit every time a proof is taken more or less directly
%from one of these two papers.

In our context, we are faced
with a few difficulties not addressed in \cite{TaoWrightLpImprovingBoundsForAverages}.
A main difficulty we face is that we will not assume an \it a priori \rm
smoothness that was assumed in that paper.  This will
require us to provide a more detailed study of an ODE that arises
in that paper.  This difference in difficulty here, is that while
in that paper existence for a certain ODE was proved via the contraction
mapping principle, we must also prove smooth dependence
on parameters.  Furthermore, we will generalize their results to
vector fields that do not necessarily span the tangent space.  While
this may seem like an artificial generalization, it will
prove to be essential to our study of maximal functions and unit operators in
Sections \ref{SectionUnitOpsAtUnitScale} and \ref{SectionMaximalFuncs}.  Finally, we must also combine these methods with the methods
in \cite{NagelSteinWaingerBallsAndMetricsDefinedByVectorFields} to
prove the relationships between the various balls we will define.

Let $X=\q(X_1, \ldots, X_{n_0}\w)$ be $n_0$ $C^1$ vector fields
with single-parameter formal degrees $d=\q(d_1,\ldots, d_{n_0}\w)\in \q(0,\infty\w)^{n_0}$ defined on
the fixed connected open set $\Omega\subseteq \R^n$.
% (here, the $X$s play the
%role of the $Z$s in the introduction).
Fix $1\geq \xi>0$,
$x_0\in \Omega$,
and suppose that $\q( X,d\w)$ satisfies $\sC\q( x_0, \xi\w)$.
Suppose further that the $X_j$s satisfy an integrability condition
on $\B{X}{d}{x_0}{\xi}$
given by:
\begin{equation*}
\q[ X_j, X_k \w] = \sum_{l} c_{j,k}^l X_l.
\end{equation*}
In this section, we will assume that:
\begin{itemize}
\item $X_1\q( x_0\w) ,\ldots, X_{n_0}\q( x_0\w)$ are linearly independent.
\item $\Cjn{1}{\B{X}{d}{x_0}{\xi}}{X_j}<\infty$, for every $1\leq j\leq n_0.$
\item For $\q|\alpha\w|\leq 2$, $X^\alpha c_{j,k}^l\in \Cj{0}{\B{X}{d}{x_0}{\xi}}$, and 
$$\sum_{\q|\alpha\w|\leq 2} \Cjn{0}{\B{X}{d}{x_0}{\xi}}{X^\alpha c_{j,k}^l}<\infty,$$ for all $j,k,l$.
\end{itemize}

We will say that $C$ is an admissible constant if $C$ can be chosen
to depend only on a fixed upper bound, $d_{max}<\infty$, for $d_1,\ldots, d_{n_0}$, 
a fixed lower bound $d_{min}>0$ for $d_1,\ldots, d_{n_0}$, a fixed upper bound for $n$ (and
therefore for $n_0$), a fixed lower bound, $\xi_0>0$, for $\xi$, and a fixed upper
bound for the quantities:
\begin{equation*}
\Cjn{1}{\B{X}{d}{x_0}{\xi}}{X_j}, \quad \sum_{\q|\alpha\w|\leq 2} \Cjn{0}{\B{X}{d}{x_0}{\xi}}{X^\alpha c_{j,k}^l}.
\end{equation*}

Furthermore, if we say that $C$ is an $m$-admissible constant, we mean that
in addition to the above, we assume that:
%\item $\Cjn{m}{\B{X}{d}{x_0}{\xi}}{X_j}<\infty$, for every $1\leq j\leq n_0$
$$\sum_{\q|\alpha\w|\leq m}\Cjn{0}{\B{X}{d}{x_0}{\xi}}{X^{\alpha}c_{j,k}^l}<\infty,$$ 
for every $j,k,l$
(in particular, these derivatives up to order $m$ exist
and are continuous).
$C$ is allowed to depend on $m$, all the quantities an admissible constant
is allowed to depend on, and a fixed upper bound for the above quantity.
%We write $A\lesssim_m B$ for $A\leq C B$ where $C$ is an $m$-admissible constant, and $A\approx_m B$ for $A\lesssim_m B$ and $B\lesssim_m A$.
Note that $\lesssim_0, \lesssim_1, \lesssim_2$, and $\lesssim$ all
denote the same thing.

For $\eta>0$, a sufficiently small admissible constant, define the map:
$$\Phi: B_{n_0}\q(\eta\w) \rightarrow \Bt{X}{d}{x_0}{\xi}$$
by
$$\Phi\q(u\w) = \exp\q(u\cdot X\w)x_0.$$
Note that, by Theorem \ref{ThmExpReg}, $\Phi$ is $C^1$.
The main theorem of this section is the following:
\begin{thm}\label{ThmMainFrobThm}
There exist admissible constants $\eta_1>0$, $\xi_1>0$, such that:
\begin{itemize}
\item $\Phi: B_{n_0}\q( \eta_1\w) \rightarrow \Bt{X}{d}{x_0}{\xi}$ is one-to-one.
\item For all $u\in B_{n_0}\q( \eta_1 \w)$, $\q| \det_{n_0\times n_0} d\Phi\q(u\w) \w| \approx \q|\det_{n_0\times n_0} X\q( x_0\w)\w|$. 
\item $\B{X}{d}{x_0}{\xi_1}\subseteq \Phi\q( B_{n_0}\q( \eta_1\w) \w) \subseteq \Bt{X}{d}{x_0}{\xi} \subseteq \B{X}{d}{x_0}{\xi}$.
\end{itemize} 
Furthermore, if we let $Y_j$ be the pullback of $X_j$ under the map $\Phi$,
then we have:
\begin{equation}\label{EqnRegularityForYj}
\Cjn{m}{B_{n_0}\q(\eta_1\w)}{Y_j}\lesssim_m 1
\end{equation}
in particular,
$$\Cjn{2}{B_{n_0}\q(\eta_1\w)}{Y_j}\lesssim 1.$$
Finally, if for $u\in B_{n_0}\q( \eta_1\w)$ we define the $n_0\times n_0$ matrix
$A\q( u\w)$ by:\footnote{Here we are thinking of $\grad_u$ as the vector $\q(\partial_{u_1},\ldots, \partial_{u_{n_0}}\w)$.}
$$\q( Y_1,\ldots, Y_{n_0}\w) = \q( I+A\w) \grad_u$$
then,
$$\sup_{u\in B_{n_0}\q( \eta_1\w) } \q\| A\q( u\w)\w\|\leq \frac{1}{2}.$$
%and,
%$$\Cjn{m}{B_{n_0}\q(\eta_1\w)}{Y_j}\lesssim_m 1$$
\end{thm}

This section will be devoted to the proof of Theorem \ref{ThmMainFrobThm}.
%See Section \ref{SectionClassicalFrob} for a discussion on how
%this is a ``uniform'' version of the theorem of Frobenius.

\begin{rmk}
In \cite{TaoWrightLpImprovingBoundsForAverages}, the map
$\Phi$ was defined with a large parameter $K$.  
Then,
a result like
(\ref{EqnRegularityForYj}) was proven by taking
$K$ large depending on $m$.  
It is important for the
applications we have in mind that this procedure is not necessary.
In our setup, this procedure is similar to taking the parameter
$\kappa$ in
Theorem \ref{ThmExistUniqDiffEq} small depending on $m$; however
we will see that we will be able to fix $\kappa=\frac{1}{2}$ throughout.
\end{rmk}

\begin{rmk}\label{RmkdsDontMatter}
The formal degrees, $d_1,\ldots, d_{n_0}$ do not play an essential
role in this section.  
Indeed note that they do not play a role in the assumptions
for Theorem \ref{ThmMainFrobThm}.  Moreover, since
$\xi_1,\xi\approx 1$, they do not play a role in the conclusion
either.
Indeed, Theorem \ref{ThmMainFrobThm}  
with any choice of $d_1,\ldots, d_{n_0}\in \q( 0,\infty\w)$ is
equivalent to the theorem with any other choice (though
the various constants in the conclusion of Theorem \ref{ThmMainFrobThm}
will depend on the choice of the $d$s).
%However, we have stated Theorem \ref{ThmMainFrobThm}
%with the $d_j$s to facilitate the applications in
%Section \ref{SectionMultiParamBalls}.
The reason we have chosen to state Theorem \ref{ThmMainFrobThm}
with an arbitrary choice of $d$s (instead of taking,
say, $d_1=\cdots=d_{n_0}=1$) is that when we prove Theorem 
\ref{ThmMultiBallsAtPoint} we will be, in effect,
applying Theorem \ref{ThmMainFrobThm} infinitely many times.
Having stated Theorem \ref{ThmMainFrobThm} for general $d$
will allow us to seamlessly apply the results
here without any hand-waving about how various constants
depend on the formal degrees.
%It will be much more convenient to know, in advance, how the
%constants in the conclusion of Theorem \ref{ThmMainFrobThm}
%depend on the $d$s.
\end{rmk}

\begin{rmk}\label{RmkDoesntFollowFromFrobProofs}
As was discussed in Section \ref{SectionNSW}, the methods
in \cite{NagelSteinWaingerBallsAndMetricsDefinedByVectorFields}
fail to prove Theorem \ref{ThmMainFrobThm}.
It is also worth noting that the methods usually used to prove
the theorem of Frobenius are insufficient to prove Theorem \ref{ThmMainFrobThm}.
For simplicity, we discuss the proof in \cite{LundellAShortProofOfTheFrobeniusTheorem}, but similar remarks hold for all previous proofs we know of.
In \cite{LundellAShortProofOfTheFrobeniusTheorem}, an invertible
linear transformation was applied to $X_1,\ldots, X_{n_0}$ (call the
resulting vector fields $V_1,\ldots, V_{n_0}$).  This was done
in such a way that $\q[V_i, V_j\w] =0$ for every $i,j$.
Because of this, the map:
$$u\mapsto e^{u\cdot V} x_0$$
is easy to study.  Unfortunately, we know of no {\it a priori} way
to create such an invertible linear transformation without
destroying the admissible constants.  {\it A fortiori}, however,
we may just push forward the linear transformation $\q( I+A\w)^{-1}$
via the map $\Phi$ to obtain such a linear transformation.
This idea seems to yield no nontrivial new information.
\end{rmk}
%\begin{rmk}\label{RmkFailureOfNSW}
%The methods in \cite{NagelSteinWaingerBallsAndMetricsDefinedByVectorFields}
%fail to prove Theorem \ref{ThmMainFrobThm}.  In that paper, it is used
%heavily that the vector fields $X_j$ are of the form $\delta^{d_j}W_j$
%for some fixed vector fields $W_j$, and some small $\delta>0$.  The constants in that paper
%depend on a {\it lower} bound for 
%$$\q|\det_{n_0\times n_0} \q( W_1\q( x_0\w) | \cdots |W_{n_0}\q( x_0\w) \w)\w|$$
%This is used, for instance, whenever one bounds the error term
%in the Campbell-Hausdorff formula in that paper.
%It is also worth noting that the methods usually used to prove
%the theorem of Frobenius are insufficient to prove Theorem \ref{ThmMainFrobThm}.
%For simplicity, we discuss the proof in \cite{LundellAShortProofOfTheFrobeniusTheorem}, but similar remarks hold for all previous proofs we know of.
%In \cite{LundellAShortProofOfTheFrobeniusTheorem}, an invertible
%linear transformation was applied to $X_1,\ldots, X_{n_0}$ (call the
%resulting vector fields $V_1,\ldots, V_{n_0}$).  This was done
%in such a way that $\q[V_i, V_j\w] =0$ for every $i,j$.
%Because of this, the map:
%$$u\mapsto e^{u\cdot V} x_0$$
%is easy to study.  Unfortunately, we know of no {\it a priori} way
%to create such an invertible linear transformation without
%destroying the admissible constants.  {\it A fortiori}, however,
%we may just push forward the linear transformation $\q( I+A\w)^{-1}$
%via the map $\Phi$ to obtain such a linear transformation.
%This idea seems to yield no nontrivial new information.
%\end{rmk}

\begin{rmk}
Morally, Theorem \ref{ThmMainFrobThm} (along with Theorems \ref{ThmMainUnitScale} and \ref{ThmMultiBallsAtPoint}) is a compactness result.  
This is discussed at the end of Section \ref{SectionClassicalFrob}.
The use of this compactness can be
seen every time we apply Theorem \ref{ThmInverseFunctionThm}.
Moreover, this compactness perspective was taken up in
Section 4 of \cite{StreetAdvances}.  In fact, one of the main consequences of 
this paper is that one may remove condition 4 of Definition 4.4 of \cite{StreetAdvances}, and still obtain the relevant results
(this is tantamount to saying that we do not require
a lower bound for a determinant as discussed in Section \ref{SectionNSW}).  Thus, from the 
remarks in that paper, one can easily see the results in this
paper from the perspective of compactness.
\end{rmk}

The next two lemmas we state in slightly greater generality than
we need, since we will refer to the proofs later in the paper.
\begin{lemma}\label{LemmaLieDerivOfDet}
Fix $1\leq n_1\leq n_0$.  Then, for $1\leq j\leq n_0$, $I\in \sI{n_1}{n}$,
$J\in \sI{n_1}{n_0}$, $x\in \B{X}{d}{x_0}{\xi}$,
\begin{equation*}
\q| X_j  \det X\q(x\w)_{I,J}\w| \lesssim \q| \det_{n_1\times n_1} X\q(x\w)\w|.
\end{equation*}
\end{lemma}
\begin{proof}
We use the notation $\sL_U$ to denote the Lie derivative with respect
to the vector field $U$, and $i_V$ to denote the interior
product with the vector field $V$.  $\sL_U$ and $i_V$ have the following, well-known, properties:
\begin{itemize}
\item $\sL_U f = Uf$ for functions $f$.
\item $\q[\sL_U, i_V\w] = i_{\q[U,V\w]}$.
\item $\sL_U \omega = i_U d\omega + d i_U \omega $, for forms $\omega$.
\item $\sL_U \q(\omega_1\wedge \omega_2\w) = \q( \sL_U \omega_1 \w)\wedge \omega_2 + \omega_1\wedge \q( \sL_U \omega_2\w)$ for forms $\omega_1,\omega_2$. 
\item If $U=\sum_k b_k \frac{\partial}{\partial x_k}$, then,
$$\sL_U d x_k = d i_U d x_k = d b_k = \sum \frac{\partial b_k}{\partial x_j}d x_j.$$
\end{itemize}

Fix $I=\q(i_1,\ldots,i_{n_1}\w),J=\q(j_1,\ldots,j_{n_1}\w)$ as in the statement of the lemma.  Then,
$$\det X\q(x\w)_{I,J} = i_{X_{j_{n_1}}} i_{X_{j_{n_1-1}}} \cdots i_{X_{j_1}} d x_{i_1} \wedge d x_{i_2} \wedge \cdots \wedge d x_{i_{n_1}}.$$
Thus, we see:
\begin{equation}\label{EqnLieDerivsFirst}
\begin{split}
X_j \det X\q( x\w)_{I,J} &= \sL_{X_j} i_{X_{j_{n_1}}} i_{X_{j_{n_1-1}}} \cdots i_{X_{j_1}} d x_{i_1} \wedge d x_{i_2} \wedge \cdots \wedge d x_{i_{n_1}}\\
&= i_{\q[ X_j, X_{j_{n_1}}\w]} i_{X_{j_{n_1-1}}} \cdots i_{X_{j_1}} d x_{i_1} \wedge d x_{i_2} \wedge \cdots \wedge d x_{i_{n_1}}\\
&\quad +i_{X_{j_{n_1}}} i_{\q[ X_j, X_{j_{n_1-1}}\w]} \cdots i_{X_{j_1}} d x_{i_1} \wedge d x_{i_2} \wedge \cdots \wedge d x_{i_{n_1}}\\
&\quad+\cdots+i_{X_{j_{n_1}}} i_{X_{j_{n_1-1}}} \cdots i_{\q[ X_j, X_{j_{1}}\w]} d x_{i_1} \wedge d x_{i_2} \wedge \cdots \wedge d x_{i_{n_1}}\\
&\quad + i_{X_{j_{n_1}}} i_{X_{j_{n_1-1}}} \cdots i_{X_{j_1}} \sL_{X_j} \q( d x_{i_1} \wedge d x_{i_2} \wedge \cdots \wedge d x_{i_{n_1}} \w).
\end{split}
\end{equation}
Every term, except the last term, on the RHS of (\ref{EqnLieDerivsFirst})
is easy to estimate.  We do the first term as an example, and all of the others
work in the same way:
\begin{equation*}
\begin{split}
&\q| i_{\q[ X_j, X_{j_{n_1}}\w]} i_{X_{j_{n_1-1}}} \cdots i_{X_{j_1}} d x_{i_1} \wedge d x_{i_2} \wedge \cdots \wedge d x_{i_{n_1}}\w|\\
&\quad = \q| \sum_{k=1}^{n_0} c_{j,j_{n_1}}^k i_{X_k} i_{X_{j_{n_1-1}}} \cdots i_{X_{j_1}} d x_{i_1} \wedge d x_{i_2} \wedge \cdots \wedge d x_{i_{n_1}} \w|\\
&\quad \lesssim \q| \det_{n_1\times n_1} X\q( x\w) \w|.
\end{split}
\end{equation*}
Since, for each $k$, $i_{X_k} i_{X_{j_{n_1-1}}} \cdots i_{X_{j_1}} d x_{i_1} \wedge d x_{i_2} \wedge \cdots \wedge d x_{i_{n_1}}$ is either $0$ or of
the form $\pm \det X\q( x\w)_{I,J'}$ for some $J'\in \sI{n_1}{n_0}$.

We now turn to the last term on the RHS of (\ref{EqnLieDerivsFirst}).
We have:
\begin{equation*}
\begin{split}
&\sL_{X_j} \q( d x_{i_1} \wedge d x_{i_2} \wedge \cdots \wedge d x_{i_{n_1}} \w)\\
&\quad =\q( \sL_{X_j} d x_{i_1} \w) \wedge d x_{i_2} \wedge \cdots \wedge d x_{i_{n_1}} + d x_{i_1} \wedge \q(\sL_{X_j} d x_{i_2} \w) \wedge \cdots \wedge d x_{i_{n_1}}\\
&\quad\quad+ \cdots+d x_{i_1} \wedge d x_{i_2} \wedge \cdots \wedge \q(\sL_{X_j} d x_{i_{n_1}} \w).
\end{split}
\end{equation*}
So we may separate the last term on the RHS of (\ref{EqnLieDerivsFirst}) into
a sum of $n_1$ terms.  We bound just the first, the bounds of the
others being similar.  To to this, let $X_j = \sum_{k} b_j^k \frac{\partial}{\partial x_k}$.  Note that $\Cjn{1}{\B{X}{d}{x_0}{\xi}}{b_j^k}\lesssim 1$.

\begin{equation*}
\begin{split}
&\q| i_{X_{j_{n_1}}} i_{X_{j_{n_1-1}}} \cdots i_{X_{j_1}} \q( \sL_{X_j} d x_{i_1} \w) \wedge d x_{i_2} \wedge \cdots \wedge d x_{i_{n_1}}\w|\\
&\quad = \q| \sum_l \frac{\partial b_j^{i_1}}{\partial x_l} i_{X_{j_{n_1}}} i_{X_{j_{n_1-1}}} \cdots i_{X_{j_1}} d x_l \wedge d x_{i_2} \wedge \cdots \wedge d x_{i_{n_1}}\w|\\
&\quad \lesssim \q| \det_{n_1\times n_1} X\q( x\w) \w|.
\end{split}
\end{equation*}
since each of the terms $i_{X_{j_{n_1}}} i_{X_{j_{n_1-1}}} \cdots i_{X_{j_1}} d x_l \wedge d x_{i_2} \wedge \cdots \wedge d x_{i_{n_1}}$ is either $0$ or of the form
$\pm \det X\q( x\w)_{I',J}$ for some $I'\in \sI{n_1}{n}$.
\end{proof}

\begin{rmk}
The reader wishing to avoid the use
of Lie derivatives in Lemma \ref{LemmaLieDerivOfDet}
should consult Lemma 2.6 of
\cite{NagelSteinWaingerBallsAndMetricsDefinedByVectorFields}
where a similar result in the special case $n_1=n$ is 
shown directly, without the use of Lie derivatives.
However, the proof we give in Lemma \ref{LemmaLieDerivOfDet}
is easily adapted to other situations that will arise in this paper
(e.g. Lemmas \ref{LemmaLieDerivOfDetCarnot}
and \ref{LemmaQuotientOfDets}),
while the proof in \cite{NagelSteinWaingerBallsAndMetricsDefinedByVectorFields}
becomes progressively more complicated to generalize.
\end{rmk}

\begin{lemma}\label{LemmaDetsDontChange}
For $y\in \B{X}{d}{x_0}{\xi}$, $1\leq n_1\leq n_0$, we have
$$\q| \det_{n_1\times n_1} X\q( y\w) \w| \approx \q| \det_{n_1\times n_1} X\q( x_0\w) \w|$$
In particular, since $\q| \det_{n_0\times n_0} X\q( x_0\w) \w|\ne 0$, $\q| \det_{n_0\times n_0} X\q( y\w) \w| \ne 0$.
\end{lemma}
\begin{proof}
Since $y\in \B{X}{d}{x_0}{\xi}$, there exists $\gamma:\q[0,1\w]\rightarrow \B{X}{d}{x_0}{\xi}$
with 
\begin{itemize}
\item $\gamma\q( 0\w) =x_0$, $\gamma\q( 1\w) = y$,
\item $\gamma'\q( t\w) =  a\q( t\w)\cdot X\q( \gamma\q(t\w)\w)$,
\item $a\in \q(\Lpp{\infty}{\q[0,1\w]}\w)^{n_0}$,
\item $\Lppn{\infty}{\q[0,1\w]}{\q|\xi^{-d} a\w|}<1$.
\end{itemize}
But, then consider:
\begin{equation*}
\begin{split}
\frac{d}{dt} \q| \det_{n_1\times n_1} X\q( \gamma\q(t\w)\w) \w|^2 &= 2\sum_{\substack{I\in \sI{n_1}{n}\\J\in \sI{n_1}{n_0}}} \det X_{I,J}\q( \gamma\q(t\w)\w) \frac{d}{dt} \det X_{I,J}\q( \gamma \q(t\w)\w)
\\&= 2\sum_{\substack{I\in \sI{n_1}{n}\\J\in \sI{n_1}{n_0}}} \det X_{I,J}\q( \gamma\q(t\w)\w) \q(\q( a\cdot X\w)  \det X_{I,J}\w)\q(\gamma\q( t\w)\w)\\
&\lesssim \q|\det_{n_1\times n_1} X\q( \gamma\q(t\w)\w)\w|^2
\end{split}
\end{equation*}
where, in the last step, we have applied Lemma \ref{LemmaLieDerivOfDet}.
Hence, Gronwall's inequality shows:
\begin{equation*}
\q| \det_{n_1\times n_1} X\q( y\w) \w| = \q| \det_{n_1\times n_1} X\q( \gamma\q(1\w)\w) \w| \lesssim \q| \det_{n_1\times n_1} X\q( \gamma\q(0\w)\w) \w| = \q| \det_{n_1\times n_1} X\q( x_0\w) \w|.
\end{equation*}
Reversing the path $\gamma$ and applying the same argument, we see that:
$$\q| \det_{n_1\times n_1} X\q( x_0\w) \w| \lesssim \q| \det_{n_1\times n_1} X\q( y\w) \w|,$$
completing the proof.
\end{proof}

Now consider the map $\Phi:B_{n_0}\q( \eta\w)\rightarrow \B{X}{d}{x_0}{\xi}$.
$d\Phi\q( 0\w) = X\q( x_0\w)$, and it follows that
$\det_{n_0\times n_0} d\Phi\q( 0\w)\ne 0$.  Hence, if we consider $\Phi$
as a map to the leaf generated by $X$ passing through the point $x_0$,
the inverse function theorem shows that there is a (non-admissible) $\delta>0$
such that:
$$\Phi: B_{n_0}\q( \delta\w)\rightarrow \Phi\q(B_{n_0}\q( \delta\w)\w)$$
is a $C^1$ diffeomorphism.  Pullback the vector field $X_j$ via the map
$\Phi$ to $B_{n_0}\q( \delta\w)$.  Call this $C^0$ vector field $\widehat{Y}_j$.

Clearly $\widehat{Y}_j\q(0\w) =\frac{\partial}{\partial u_j}$.  Write:
\begin{equation}\label{EqnDefnAj}
\widehat{Y}_j = \pd{u_j} + \sum_k \hat{a}_j^k \pd{u_k}
\end{equation}
with $\hat{a}_j^k\q( 0\w)=0$.  Moreover, in polar coordinates, for $\omega$ fixed,
Remark \ref{RmkExpMoreReg} shows that $\hat{a}_j^k\q( r\omega\w)$ is $C^1$
in the $r$ variable, and it follows that for $\omega$ fixed,
$\hat{a}_j^k \q( r\omega\w) = O\q( r\w)$.
We will now show that $\hat{a}_j^k$ satisfies an ODE in the $r$
variable.  
The derivation of this ODE is classical (see, for instance, page 155 of \cite{ChevalleyTheoryOfLieGroupsI}, though we follow the presentation of \cite{TaoWrightLpImprovingBoundsForAverages}), and is the main starting point for this 
entire section.
%The derivation of this ODE is taken from \cite{TaoWrightLpImprovingBoundsForAverages}, and is the main starting point for this entire section.
We include the derivation here, since it is not very long, and is
of fundamental importance to the rest of the paper.

Continuing in polar coordinates,
$$\Phi\q( r,\omega\w) = \exp\q( r\q(\omega\cdot X\w)\w)x_0.$$
Hence,
\begin{equation*}
d\Phi\q( r\partial_r \w)\q( \Phi\q( r,\omega\w)\w) = rd\Phi\q( \partial_r\w) \q( \Phi\q( r, \omega\w)\w) = r \omega\cdot X \q( \Phi\q( r,\omega\w)\w).
\end{equation*}
Writing this in Cartesian coordinates, we have the following vector field
identity on $B_{n_0}\q( \delta\w)$:
\begin{equation}\label{EqnVFIdentity}
\sum_{j=1}^{n_0} u_j \pd{u_j} = \sum_{j=1}^{n_0} u_j \widehat{Y}_j.
\end{equation}
Taking the lie bracket of (\ref{EqnVFIdentity}) with $\widehat{Y}_i$,
we obtain:
\begin{equation}\label{EqnFirstYjCommute}
\begin{split}
\sum_{j=1}^{n_0} \q( \widehat{Y_i}\q( u_j\w) \partial_{u_j} + u_j \q[\widehat{Y}_i, \partial_{u_j}\w] \w) &= \sum_{j=1}^{n_0} \q( \widehat{Y}_i\q(u_j\w) \widehat{Y}_j + u_j \q[\widehat{Y}_i, \widehat{Y}_j\w] \w)
\\& = \sum_{j=1}^{n_0} \q(\widehat{Y}_i\q( u_j \w) \widehat{Y}_j + u_j\sum_{l=1}^{n_0} \ct_{i,j}^l\q(u\w) \widehat{Y}_l  \w),
\end{split}
\end{equation}
where $\ct_{i,j}^k\q( u\w) = c_{i,j}^k \q( \Phi\q( u\w)\w)$, and
we have used the fact that $\q[\widehat{Y}_i, \widehat{Y}_j\w]=\sum \ct_{i,j}^k \widehat{Y}_k$.

\begin{rmk}
Since $\widehat{Y}_i$ is not $C^1$, one might worry about our
manipulations in (\ref{EqnFirstYjCommute}).  This turns out
to not be a problem.  Indeed, it makes sense to take the above
commutator, since $\widehat{Y}_i$ is $C^1$ in the $r$ variable (and
we are commuting it with $r\partial_r$).  Then, the computations
on the LHS of (\ref{EqnFirstYjCommute}) may be done in the
sense of distributions, while the computations on the RHS
may be done by pushing everything forward via the map $\Phi$.
We leave the details to the reader.
\end{rmk}

We re-write (\ref{EqnFirstYjCommute}) as:
\begin{equation}\label{EqnSecondYjCommute}
\begin{split}
&\q(\sum_{j=1}^{n_0} u_j \q[\partial_{u_j}, \widehat{Y}_i-\partial_{u_i}\w]\w) + \widehat{Y}_i-\partial_{u_i} \\&\quad= -\q(\sum_{j=1}^{n_0} \q(\widehat{Y}_i -\partial_{u_i}\w)\q(u_j\w)\q(\widehat{Y}_j-\partial_{u_j}\w)\w) - \sum_{j=1}^{n_0}\sum_{l=1}^{n_0} u_j \ct_{i,j}^l\q(u\w) \widehat{Y}_l.
\end{split}
\end{equation}

Plugging (\ref{EqnDefnAj}) into (\ref{EqnSecondYjCommute}) we have:
\begin{equation}\label{Eqn27}
\begin{split}
&\sum_{j=1}^{n_0} \sum_{k=1}^{n_0} u_j \q( \partial_{u_j} \hat{a}_j^k\w)\partial_{u_k} + \sum_{k=1}^{n_0} \hat{a}_i^k \partial_{u_k} \\
&\quad= -\q( \sum_{j=1}^{n_0} \sum_{k=1}^{n_0} \hat{a}_i^j \hat{a}_j^k \partial_{u_k} \w) - \sum_{k=1}^{n_0} \q(\sum_{j=1}^{n_0} u_j \ct_{i,j}^k\w) \partial_{u_k} - \sum_{l=1}^{n_0}\sum_{k=1}^{n_0} \q(\sum_{j=1}^{n_0} t_j \ct_{i,j}^l\w) \hat{a}_l^k \partial_{u_k}.
\end{split}
\end{equation}

Taking the $\partial_{u_k}$ component, and writing $\sum_{j=1}^{n_0} u_j \partial_{u_j} +1 = \partial_r r$, we have from (\ref{Eqn27}):
\begin{equation}\label{EqnTakingComps27}
\partial_r r\hat{a}_i^k = -\sum_{j=1}^{n_0} \hat{a}_i^j\hat{a}_j^k -\sum_{j=1}^{n_0} u_j \ct_{i,j}^k - \sum_{l=1}^{n_0}\q( \sum_{j=1}^{n_0} u_j \ct_{i,j}^l \w) \hat{a}_l^k.
\end{equation}
Define two $n_0\times n_0$ matrices, $\widehat{A}, C_u$ by:
\begin{equation*}
\widehat{A}_{i,k} := \q( \hat{a}_i^k\w), \quad \q(C_u\w)_{i,k} := \q(\sum_{j=1}^{n_0} u_j \ct_{i,j}^k\w), \quad 1\leq i,k\leq n_0.
\end{equation*}
Using this, (\ref{EqnTakingComps27}) may be re-written as the matrix
valued ODE:
\begin{equation}\label{EqnHatDiffEq}
\partial_r r \widehat{A} = -\widehat{A}^2 -C_u\widehat{A}-C_u.
\end{equation}

\begin{thm}\label{ThmExistUniqDiffEq}
Fix $\frac{1}{2}\geq \kappa>0$ (throughout the paper we will choose $\kappa=\frac{1}{2}$).  Consider the differential equation:
\begin{equation}\label{EqnThmDiffEq}
\partial_r r A\q( r\omega\w) = - A\q( r\omega\w)^2 -C_u\q( r\omega\w) A\q( r\omega\w) - C_u\q( r\omega\w),
\end{equation}
defined for $A:B_{n_0}\q( \eta\w) \rightarrow \mathbb{M}_{n_0\times n_0}\q( \R\w)$, where $\mathbb{M}_{n_0\times n_0}\q( \R\w)$ 
denotes the set of $n_0\times n_0$ real matrices.  Then, there exists an admissible constant $\eta_1=\eta_1\q( \kappa\w)>0$ such
that there exists a unique solution $A\in C\q( B_{n_0}\q(\eta_1 \w); \mathbb{M}_{n_0\times n_0}\q( \R\w) \w)$ to (\ref{EqnThmDiffEq}) satisfying $A\q( r\omega\w) = O\q( r\w)$ for each fixed $\omega$.  Moreover, this solution satisfies:
\begin{itemize}
\item $\q\|A\q( t\w)\w\|\lesssim \q| t\w|$.
\item $\sup_{t\in B_{n_0}\q( \eta_1\w)} \q\|A\q(t\w)\w\|\leq \kappa$.
\end{itemize}
Furthermore, if $\ct_{i,j}^k\in \Cj{m}{B_{n_0}\q( \eta_1\w)}$
with $\Cjn{m}{B_{n_0}\q( \eta_1\w)}{\ct_{i,j}^k}<\infty$, then
$A\in C^m\q(B_{n_0}\q( \eta\w); \mathbb{M}_{n_0\times n_0}\q( \R\w)\w)$,
and if $\widetilde{C}_{m,\eta_1}$ is a fixed upper bound for:
$$\Cjn{m}{B_{n_0}\q( \eta_1\w)}{\ct_{i,j}^k},\quad 1\leq i,j,k\leq n_0,$$
then, there exists an admissible constant $C_m=C_m\q( m,\widetilde{C}_{m,\eta_1} \w)$
such that:
\begin{equation}\label{EqnCmBoundForA}
\q\|A\w\|_{C^m\q(B_{n_0}\q(\eta_1\w); \mathbb{M}_{n_0\times n_0}\q( \R\w)\w)}\leq C_m.
\end{equation}
\end{thm}

Note that (\ref{EqnThmDiffEq}) is not a standard ODE (due to the
factor of $r$ on the left hand side), and so
we cannot apply the standard theorems for existence and
dependence on parameters.
Fortunately, though, we will be able to prove Theorem \ref{ThmExistUniqDiffEq}, by adapting the methods
of \cite{IzzoCrConvergenceOfPicardsSuccessiveApprox}.  
In \cite{TaoWrightLpImprovingBoundsForAverages}, the solution
$A$ was assumed to be \it a priori \rm $C^\infty$, thereby
removing many of the difficulties in the proof of Theorem \ref{ThmExistUniqDiffEq}.
Before we begin
the proof, we need two preliminary lemmas:

\begin{lemma}\label{LemmaDefineContract}
Fix $\epsilon>0$.  Suppose $g\in C^m\q( B_{n_0}\q( \epsilon\w)\w)$.
Define $h$ on $B_{n_0}\q( \epsilon\w)$ by:
\begin{equation}\label{EqnDefnContract}
h\q(r\omega\w) = \begin{cases}
\frac{1}{r} \int_0^r g\q( s\omega\w) ds & \text{if $r\ne 0$},\\
g\q( 0\w) & \text{if $r=0$}.
\end{cases}
\end{equation}
Then, $h\in C^m\q( B_{n_0}\q( \epsilon\w)\w)$.  Moreover, if $\alpha$ is
a multi-index with $\q|\alpha\w|\leq m$, we have:
\begin{equation}\label{EqnDerivContract}
\q(\partial_u^\alpha h\w) \q(r\omega \w) = \begin{cases}
\frac{1}{r^{\q|\alpha\w|+1}} \int_0^r s^{\q|\alpha\w|} \q(\partial_u^\alpha g\w)\q( s\omega \w) ds & \text{if $r\ne 0$},\\
\frac{1}{\q|\alpha\w|+1}\q(\partial_u^\alpha g\w)\q(0\w) & \text{if $r=0$}.
\end{cases}
\end{equation}
\end{lemma}
\begin{proof}
Note that, since $g\in C^m$, the right hand sides of (\ref{EqnDefnContract})
and (\ref{EqnDerivContract}) are both continuous in $r$.  Note, also,
that to prove the lemma, it suffices to prove the formula (\ref{EqnDerivContract})
for $g\in C^\infty$, as then the linear map $g\mapsto h$ will extend as a
map $C^\infty\rightarrow C^m$ to a map $C^m\rightarrow C^m$.  Hence,
we prove the lemma just under the assumption $g\in C^\infty$ (this
reduction is not necessary for our proof, but it simplifies notation a bit).

First, we prove the lemma for $r\ne 0$.  Away from
$r=0$, $h$ is clearly $C^\infty$, and so we need only verify the formula
(\ref{EqnDerivContract}).
$h$ satisfies the formula:
$$\partial_r r h\q( r\omega\w) = g\q( r\omega\w).$$
Apply $\partial_u^\alpha$ to both sides of this formula.  Using
the fact that $\q[\partial_u^\alpha,\partial_r r\w]= \q|\alpha\w|\partial_u^\alpha$, we have:
$$\partial_r r \q(\partial_u^\alpha h\w) \q( r\omega\w) +\q|\alpha\w|\q( \partial_u^\alpha h\w) \q( r\omega\w) = \q( \partial_u^{\alpha} g\w) \q( r\omega\w).$$
Multiplying both sides by $r^{\q|\alpha\w|}$, we obtain:
$$\partial_r r^{\q|\alpha\w|+1} \q( \partial_u^\alpha h\w)\q( r\omega\w) = r^{\q|\alpha\w|} \q( \partial_u^{\alpha} g\w) \q( r\omega\w)$$
and (\ref{EqnDerivContract}) follows for $r\ne 0$.

Hence, to complete the proof, we need only show that $\partial_u^{\alpha} h$ exists
at $0$ and is given by the $\frac{1}{\q|\alpha\w|+1}\q(\partial_u^\alpha g\w)\q( 0\w)$.  We first consider the case when:
$$\partial_u^\beta g\q( 0\w) =0, \quad 0\leq\q|\beta\w|\leq m,$$
and we prove the result by induction on the order of $\alpha$,
our base case being the trivial case $\q|\alpha\w|=0$.  Thus, suppose
we have the result for some $\alpha$, $\q|\alpha\w|<m$ and we wish
to show that the following derivative exists, and equals $0$:
\begin{equation*}
\partial_{u_j} \partial_u^{\alpha} h\q( r\omega\w)\bigg|_{r=0} = \partial_{u_j} \begin{cases}
\frac{1}{r^{\q|\alpha\w|+1}} \int_0^r s^{\q|\alpha\w|} \q(\partial_u^\alpha g\w)\q(s\omega\w) ds & \text{if $r\ne 0$}\\
0 &\text{if $r=0$}
\end{cases}\bigg|_{r=0}
\end{equation*}
And this will follow if we can show that:
\begin{equation}\label{EqnContractToShow}
\frac{1}{r^{\q|\alpha\w|+1}} \int_0^r s^{\q|\alpha\w|} \q(\partial_u^\alpha g\w)\q(s\omega\w) ds = o\q(r\w).
\end{equation}
But, by our assumption on $g$, $\q(\partial_u^\alpha g\w)\q(s\omega\w) = O\q( s^2\w)$ and (\ref{EqnContractToShow}) follows, completing the proof in this case.

Now turn to the general case $g\in C^\infty$.  We may write:
$$g\q( u\w) = \sum_{\q|\beta\w|\leq m} \frac{1}{\beta!}\q(\partial_u^\beta g\w) \q(0\w) u^\beta + g_e\q( u\w)$$
where $g_e$ vanishes to order $m$ at $0$.  Thus, by linearity of the map
$g\mapsto h$, it suffices to prove the lemma for monomials $u^{\beta}$.
Since we know (\ref{EqnDerivContract}) holds away from $r=0$ and we know
the RHS of (\ref{EqnDerivContract}) is continuous, it suffices to show that
if $g=u^\beta$, then $h\in C^\infty$.  But in this case, $h=\frac{1}{\q|\beta\w|+1} u^\beta\in C^\infty$, completing the proof.
\end{proof}

\begin{lemma}[\cite{IzzoCrConvergenceOfPicardsSuccessiveApprox}, p. 2060]\label{LemmaIzzoLemma}
Suppose $\q( M, \rho\w)$ is a metric space, and suppose $\q( Q_n\w)_{n=0}^\infty$ is a sequence of contractions on $M$ for which there exists a number $c<1$
such that:
$$\rho\q(Q_n\q(x\w),Q_n\q(y\w)\w)\leq c\rho\q(x,y\w)$$
for all $x,y\in M$ and all $n$.  Suppose also that there is a point $x_\infty\in M$
such that $Q_n\q( x_\infty\w) \rightarrow x_\infty$ as $n\rightarrow \infty$.
Let $x_0\in M$ be arbitrary, and define a sequence $\q( x_n\w)$ by setting:
$$x_{n+1}=Q_n\q( x_n\w).$$
Then, $x_n\rightarrow x_\infty$ as $n\rightarrow \infty$.
\end{lemma}

\begin{proof}[Proof of Theorem \ref{ThmExistUniqDiffEq}]
It is easy to see from the definition $C_u$ that:
$$\q\|C_u\q(r\omega\w)\w\|\leq Dr$$
where $D$ is an admissible constant.  Take $\eta_1=\eta_1\q( \kappa\w)>0$ to be
an admissible constant so small that:
\begin{equation*}
\kappa^2+\frac{D\eta_1}{2}\q( \kappa+1\w) \leq \kappa, \quad \kappa+\frac{D\eta_1}{3}\leq \frac{3}{4}.
\end{equation*}

Our first step will be to show the existence of $A$ using the contraction
mapping principle.  Moreover, this contraction mapping principle
may be considered the base case in an induction we will use
at the end of the proof, to establish the regularity of $A$.
Consider the metric space:
\begin{equation*}
\begin{split}
M:=\bigg\{A\in C\q(B_{n_0}\q(\eta_1\w); \mathbb{M}_{n_0\times n_0}\q(\R\w)\w):& A\q(0\w)=0, \sup_{\substack{0<r\leq \eta_1\\\omega\in S^{n_0-1}}} \q\|\frac{1}{r}A\q(r\omega\w) \w\|<\infty,\\
& \sup_{t\in B_{n_0}\q(\eta_1\w)} \q\|A\q(t\w)\w\|\leq \kappa\bigg\}
\end{split}
\end{equation*}
with the metric:
$$\rp{A}{B} = \sup_{\substack{0<r\leq \eta_1\\\omega\in S^{n_0-1}}} \q\|\frac{1}{r}\q(A\q(r\omega\w)-B\q(r\omega\w)\w)\w\|.$$
Note that $M$ is complete with respect to the metric $\rho$.
Define the map $T: M\rightarrow C\q( B_{n_0}\q( \eta_1\w);\mathbb{M}_{n_0\times n_0}\q( \R\w)  \w)$, by:
\begin{equation*}
TA \q( r\omega\w) = \begin{cases}
\frac{1}{r} \int_0^r -A\q(s\omega\w)^2-C_u\q(s\omega\w)A\q( s\omega\w)-C_u\q(s\omega\w) ds & \text{if $r\ne 0$},\\
0 & \text{if $r=0$}.
\end{cases}
\end{equation*}
Note that, by Lemma \ref{LemmaDefineContract}, $TA\in C\q( B_{n_0}\q( \eta_1\w);\mathbb{M}_{n_0\times n_0}\q( \R\w)  \w)$.

Our first goal is to show that $T:M\rightarrow M$.  Consider, for
$0<r\leq \eta_1$, $\omega\in S^{n_0-1}$, $A\in M$,
\begin{equation*}
\begin{split}
\q\|TA\q(r\omega\w)\w\| &\leq \frac{1}{r}\int_0^r \q\|A\q(s\omega\w)\w\|^2 +\q\|C_u\q(s\omega\w)\w\|\q\|A\q(s\omega\w)\w\| + \q\|C_u\q(s\omega\w)\w\| ds\\
&\leq \frac{1}{r} \int_0^r \q(\kappa^2 + Ds\kappa + Ds \w)ds\\
&\leq \kappa^2 + \frac{D\eta_1}{2}\kappa + \frac{D\eta_1}{2}\leq \kappa.
\end{split}
\end{equation*}
Thus, by the definition of $TA$, $\sup_{t\in B_{n_0}\q(\eta_1\w)} \q\|TA\q(t\w)\w\|\leq \kappa$.

Next, we have:
\begin{equation*}
\begin{split}
\q\|\frac{1}{r} TA\q(r\omega\w)\w\| &\leq \frac{1}{r^2} \int_0^r \q(s\kappa \rp{0}{A} + D s \kappa + Ds \w)ds\\
&= \frac{\kappa}{2}\rp{0}{A}+ \frac{D\kappa}{2}+\frac{D}{2}<\infty.
\end{split}
\end{equation*}
Hence, $T:M\rightarrow M$.

Next, we wish to show that $T$ is a contraction.  Consider, suppressing
the dependence on $s\omega$ in the integrals,
\begin{equation*}
\begin{split}
\q\|\frac{1}{r}\q(TA\q(r\omega\w) - TB\q(r\omega\w)\w)\w\| & 
= \q\| \frac{1}{r^2} \int_0^r -\q(A-B\w)A-B\q(A-B\w)-C_u\q(A-B\w)   \w\|\\
&\leq \frac{1}{r^2}\int_0^r \q(2s\kappa\rp{A}{B} + Ds^2\rp{A}{B} \w)ds\\
&\leq \kappa \rp{A}{B}+ \frac{D\eta_1}{3} \rp{A}{B}\\
&\leq \frac{3}{4}\rp{A}{B}
\end{split}
\end{equation*}
where the last line follows by our choice of $\eta_1$.
Thus, we have $\rp{TA}{TB}\leq \frac{3}{4}\rp{A}{B}$.

Applying the contraction mapping principle, there exists a unique
fixed point $A\in M$ such that $TA=A$.  This is the desired solution to
(\ref{EqnThmDiffEq}).  Since $A\in M$, we have $\sup_{t\in B_{n_0}\q(\eta_1\w)} \q\|A\q(t\w)\w\|\leq \kappa$.  Moreover, since
$A=\lim_{n\rightarrow \infty} T^n\q( 0\w)$, we have:
\begin{equation*}
\rp{0}{A}=\lim_{n\rightarrow \infty} \rp{0}{T^n0}\leq \sum_{n=1}^\infty \rp{T^{n-1}0}{T^{n}0} \leq \sum_{n=0}^\infty \q(\frac{3}{4}\w)^n \rp{0}{T0}=4\rp{0}{T0}
\end{equation*}
and for $r\ne 0$, we have:
\begin{equation*}
\q\|\frac{1}{r}T0\q( r\omega\w)\w\| \leq \frac{1}{r^2} \int_0^r Ds ds\leq \frac{D}{2}
\end{equation*}
and so $\rp{0}{T0}\lesssim 1$ and therefore $\rp{0}{A}\lesssim 1$.  This can
be rephrased as $\q\|A\q( t\w)\w\| \lesssim \q|t\w|$.

We now turn to uniqueness of the solution $A$.  Suppose $B$ is another solution
(we are not, necessarily, assuming $B\in M$).  Suppose that, for $\omega$
fixed, $\q\| B\q( r\omega\w)\w\| = O\q(r\w)$.  Then, we have:
\begin{equation*}
\q\|r\q(A\q(r\omega\w) -B\q(r\omega\w)\w)\w\|\leq \int_0^r \q( \q\|s\q(A-B\w)\w\|\q[\q\|\frac{A}{s}\w\|+\q\|\frac{B}{s}\w\| + \q\|\frac{C_u}{s}\w\|  \w] \w)ds
\end{equation*}
And applying the integral form of Gronwall's inequality to $\q\|r\q( A-B\w)\w\|$ shows
that $A=B$.

To conclude the proof, we need to show that if $\ct_{i,j}^k\in C^m$, then
$A\in C^m$, and to estimate the $C^m$ norm of $A$.  First, we show that
$A\in C^m$.  To do this, we will show that $T^n0\rightarrow A$ in $C^m\q(B_{n_0}\q(\eta_1\w);\mathbb{M}_{n_0\times n_0}\q( \R\w)\w)$ (here we mean the Banach space of those $C^m$ functions all of whose derivatives up to order $m$ are bounded on $B_{n_0}\q(\eta_1\w)$).  We proceed by induction on $m$, our base case being
$m=0$, which we have already proven, by the contraction mapping
principle.  Thus, suppose $\Cjn{m}{B_{n_0}\q( \eta_1\w)}{\ct_{i,j}^k}<\infty$
for $1\leq i,j,k\leq n_0$ and suppose
\begin{equation*}
\lim_{n\rightarrow \infty} \q\|T^n0 -A\w\|_{C^{m-1}\q(B_{n_0}\q(\eta_1\w);\mathbb{M}_{n_0\times n_0}\q( \R\w)\w)}=0.
\end{equation*}
Fix $\q|\alpha\w|=m$.  We will show that
$$\partial_u^\alpha T^n0$$
converges in $C^{0}\q(B_{n_0}\q(\eta_1\w);\mathbb{M}_{n_0\times n_0}\q( \R\w)\w)$, and this will complete the induction.  Note that, by Lemma \ref{LemmaDefineContract}
we know that, for each $n$, $T^n 0\in C^m$.
%, and so it suffices to show
%that $\q(\partial_u^\alpha T^n0\w)\q( u\w)$ converges uniformly for $u\ne 0$.
Fix $r\ne 0$, $\omega\in S^{n_0-1}$.

Define $\gamma_n = T^n\q( 0\w)$, $\gamma_\infty=A$.  By Lemma \ref{LemmaDefineContract}, we have, for $n<\infty$,
\begin{equation}\label{EqnDerivT}
\begin{split}
\partial_u^\alpha T\q( \gamma_n\w) \q( r\omega\w) &= \frac{1}{r^{m+1}}\int_0^r s^m \partial_u^\alpha \q(-\gamma_n^2-C_u\gamma_n-C_u\w)ds\\
&= \sum_{\alpha_1+\alpha_2=\alpha} \frac{1}{r^{m+1}} \int_0^r s^m\q(-\q(\partial_u^{\alpha_1} \gamma_n\w)\q(\partial_u^{\alpha_2}\gamma_n\w)-\q(\partial_u^{\alpha_1}C_u\w)\q(\partial_u^{\alpha_2}\gamma_n\w)\w)ds\\
&\quad - \frac{1}{r^{m+1}} \int_0^r s^m \partial_u^\alpha C_u ds.
\end{split}
\end{equation}
Define, for $l\in C^0\q(B_{n_0}\q(\eta_1\w);\mathbb{M}_{n_0\times n_0}\q(\R\w)\w)$, and for $0\leq n\leq \infty$,
\begin{equation}\label{EqnDefnQn}
\begin{split}
Q_n\q( l\w)\q( r\omega\w) &=  -\sum_{\substack{\alpha_1+\alpha_2=\alpha\\ \alpha_1\ne 0\\\alpha_2\ne 0}} \frac{1}{r^{m+1}} \int_0^r s^m\q(\partial_u^{\alpha_1} \gamma_n\w)\q(\partial_u^{\alpha_2}\gamma_n\w) ds\\
&\quad -\sum_{\substack{\alpha_1+\alpha_2=\alpha\\ \alpha_1\ne 0}} \frac{1}{r^{m+1}} \int_0^r s^m\q(\partial_u^{\alpha_1}C_u\w)\q(\partial_u^{\alpha_2}\gamma_n\w)ds\\
&\quad -\frac{1}{r^{m+1}} \int_0^r s^m \partial_u^\alpha C_u ds\\
&\quad -\frac{1}{r^{m+1}} \int_0^r s^m\q(l\gamma_n+\gamma_nl+C_u l\w)ds.
\end{split}
\end{equation}
Note that $Q_n\q( l\w)\q(u\w)$ extends continuously to $u=0$ and we have:
$$Q_n:C^0\q(B_{n_0}\q(\eta_1\w);\mathbb{M}_{n_0\times n_0}\q(\R\w)\w)\rightarrow C^0\q(B_{n_0}\q(\eta_1\w);\mathbb{M}_{n_0\times n_0}\q(\R\w)\w).$$
Putting (\ref{EqnDerivT}) and (\ref{EqnDefnQn}) together, we see, for $0\leq n<\infty$,
\begin{equation}\label{EqnQnWithT}
Q_n\q( \partial_u^\alpha \gamma_n\w) = \partial_u^{\alpha} T\q( \gamma_n\w).
\end{equation}

Our next goal is to show that $Q_n$ is a contraction ($n\leq \infty$),
as a map
$$Q_n:C^0\q(B_{n_0}\q(\eta_1\w);\mathbb{M}_{n_0\times n_0}\q(\R\w)\w)\rightarrow C^0\q(B_{n_0}\q(\eta_1\w);\mathbb{M}_{n_0\times n_0}\q(\R\w)\w).$$
Consider, for $r\ne 0$, $\omega\in S^{n_0-1}$,
and using that $\gamma_n\in M$ for all $n$,
\begin{equation*}
\begin{split}
&\q\|Q_n\q(l_1\w)\q(r\omega\w) -Q_n\q( l_2\w)\q(r\omega\w) \w\| \\
&\quad= \q\|\frac{1}{r^{m+1}}\int_0^r s^m \q[ \q(l_1-l_2\w)\gamma_n + \gamma_n\q(l_1-l_2\w) +C_u\q(l_1-l_2\w)\w]\w\|\\
&\quad \leq \q\|l_1-l_2\w\|_{C^0\q(B_{n_0}\q(\eta_1\w);\mathbb{M}_{n_0\times n_0}\q(\R\w)\w)} \frac{1}{r^{m+1}} \int_0^r s^m \q( 2\kappa + Ds\w)\\
&\quad \leq \q\|l_1-l_2\w\|_{C^0\q(B_{n_0}\q(\eta_1\w);\mathbb{M}_{n_0\times n_0}\q(\R\w)\w)}  \q(\frac{2\kappa}{m+1}+\frac{D\eta_1}{m+2}\w)\\
&\quad \leq \frac{3}{4} \q\|l_1-l_2\w\|_{C^0\q(B_{n_0}\q(\eta_1\w);\mathbb{M}_{n_0\times n_0}\q(\R\w)\w)}
\end{split}
\end{equation*}
where the last line follows by our choice of $\eta_1$.

Next, fix $l\in C^0\q(B_{n_0}\q(\eta_1\w);\mathbb{M}_{n_0\times n_0}\q(\R\w)\w)$.
We wish to show that $Q_n\q( l\w) \rightarrow Q_\infty\q( l\w)$ in $C^0\q(B_{n_0}\q(\eta_1\w);\mathbb{M}_{n_0\times n_0}\q(\R\w)\w)$.
Consider,
\begin{equation*}
\begin{split}
Q_n\q(l\w)\q(r\omega\w) - Q_\infty \q( l\w) \q(r\omega\w) &= 
-\sum_{\substack{\alpha_1+\alpha_2=\alpha\\ \alpha_1\ne 0\\\alpha_2\ne 0}} \frac{1}{r^{m+1}} \int_0^r s^m\q(\partial_u^{\alpha_1} \q(\gamma_n-\gamma_\infty\w)\w)\q(\partial_u^{\alpha_2}\gamma_n\w) ds\\
&\quad-\sum_{\substack{\alpha_1+\alpha_2=\alpha\\ \alpha_1\ne 0\\\alpha_2\ne 0}} \frac{1}{r^{m+1}} \int_0^r s^m\q(\partial_u^{\alpha_1} \gamma_\infty\w)\q(\partial_u^{\alpha_2}\q(\gamma_n-\gamma_\infty\w)\w) ds\\
&\quad -\sum_{\substack{\alpha_1+\alpha_2=\alpha\\ \alpha_1\ne 0}} \frac{1}{r^{m+1}} \int_0^r s^m\q(\partial_u^{\alpha_1}C_u\w)\q(\partial_u^{\alpha_2}\q(\gamma_n-\gamma_\infty\w)\w)ds\\
&\quad -\frac{1}{r^{m+1}} \int_0^r s^m\q(l\q(\gamma_n-\gamma_\infty\w)+\q(\gamma_n-\gamma_\infty\w)l+C_u l\w)ds.
\end{split}
\end{equation*}
Using our inductive hypothesis that $\gamma_n\rightarrow \gamma_\infty$
in $C^{m-1}\q(B_{n_0}\q(\eta_1\w);\mathbb{M}_{n_0\times n_0}\q(\R\w)\w)$
it is easy to show that the above goes to $0$ uniformly in $\q(r,\omega\w)$
as $n\rightarrow\infty$.

In particular, if we let $l_\infty$ be the unique fixed point
of the strict contraction $Q_\infty$ we have that $Q_n\q( l_\infty\w)\rightarrow l_\infty$
in $C^{0}\q(B_{n_0}\q(\eta_1\w);\mathbb{M}_{n_0\times n_0}\q(\R\w)\w)$.
Using (\ref{EqnQnWithT}), we have:
$$\partial_u^\alpha \gamma_{n+1} = \partial_u^\alpha T\q(\gamma_n\w) = Q_n\q( \partial_u^\alpha \gamma_n\w).$$
Hence, Lemma \ref{LemmaIzzoLemma} shows that:
$$\partial_u^{\alpha}\gamma_n \rightarrow l_\infty$$
which shows that $\partial_u^\alpha T^n0$ converges in $C^{0}\q(B_{n_0}\q(\eta_1\w);\mathbb{M}_{n_0\times n_0}\q(\R\w)\w)$.  It follows that $A\in C^{m}\q(B_{n_0}\q(\eta_1\w);\mathbb{M}_{n_0\times n_0}\q(\R\w)\w)$.

Moreover, we have that $\partial_u^\alpha A = l_\infty$, where $l_\infty$
was the unique fixed point of $Q_\infty$.  Hence, by the contraction mapping
principle, $\partial_u^{\alpha}A=\lim_{n\rightarrow \infty} Q_\infty^n 0$.
It follows, by a proof similar to the one we did before for $T$, that we have:
\begin{equation*}
\q\|\partial_u^\alpha A\w\|_{C^{0}\q(B_{n_0}\q(\eta_1\w);\mathbb{M}_{n_0\times n_0}\q(\R\w)\w)} \leq 4 \q\|Q_\infty\q(0\w)\w\|_{C^{0}\q(B_{n_0}\q(\eta_1\w);\mathbb{M}_{n_0\times n_0}\q(\R\w)\w)}.
\end{equation*}
Let us suppose, for induction that we have (\ref{EqnCmBoundForA}) for $m-1$.
Then to prove (\ref{EqnCmBoundForA}) for $m$ it suffices to show that:
$$\q\|Q_\infty\q(0\w)\w\|_{C^{0}\q(B_{n_0}\q(\eta_1\w);\mathbb{M}_{n_0\times n_0}\q(\R\w)\w)} \leq C_m\q(m,\widetilde{C}_{m,\eta_1}\w)$$
but this follows immediately from the inductive hypothesis and the definition
of $Q_\infty$.
\end{proof}

Now fix $\eta_1$ and $A$ as in the conclusion of Theorem \ref{ThmExistUniqDiffEq}, taking $\kappa=\frac{1}{2}$. 
\begin{lemma}\label{LemmaAequalAh}
$A\big|_{B_{n_0}\q(\delta\w)} = \widehat{A}$.
\end{lemma}
\begin{proof}
Using, as remarked before, that for fixed $\omega$, $\q\|\widehat{A}\q(r\omega\w)\w\|=O\q(r\w)$, this follows just as in the proof of uniqueness in
Theorem \ref{ThmExistUniqDiffEq}.
\end{proof}
Lemma \ref{LemmaAequalAh} shows that we may extend the vector fields
$\widehat{Y}_j$ by setting:
$$Y_j = \partial_{u_j} + \sum_{k=1}^{n_0} a_j^k \partial_{u_k}$$
where $A_{j,k}=\q(a_j^k\w)$.

\begin{thm}\label{ThmPushForwardY}
$d\Phi\q( Y_j\w) = X_j$.
\end{thm}

To prove Theorem \ref{ThmPushForwardY}, we need a preliminary lemma:
\begin{lemma}\label{LemmaPushForwardY}
Fix $\omega\in S^{n_0-1}$, $r_0< \eta_1$, and suppose that
for all $r\leq r_0$, $\q| \det_{n_0\times n_0} d\Phi\q( r\omega\w)\w|\ne 0$.
Then, on the line $\q\{r\omega: 0\leq r\leq r_0 \w\}$,
we have $d\Phi\q( Y_j\w) = X_j$, $1\leq j\leq n_0$.
\end{lemma}
\begin{proof}
Suppose not.  Define
$$r_1 = \sup\q\{r\geq 0: d\Phi\q( Y_j\w) = X_j \text{ on the line } \q\{r'\omega: 0\leq r'\leq r\w\}, 1\leq j\leq n_0\w\}.$$
Then we must have $r_1<r_0$ (by continuity).
Since $\widehat{Y}_j=Y_j\big|_{B_{n_0}\q(\delta\w)}$, we know that $r_1>0$.
Since $\q|\det_{n_0\times n_0} d\Phi\q( r_1\omega\w)\w|\ne 0$, the
inverse function theorem implies that there exists a neighborhood
$V$ of $r_1\omega$ such that $\Phi:V\rightarrow \Phi\q( V\w)$ is a
$C^1$ diffeomorphism.

Pick $0<r_2<r_3<r_1<r_4$ such that:
$$\q\{r'\omega: r_2\leq r'\leq r_4\w\}\subset V.$$
Let $\widetilde{Y}_j$ be the pullback of $X_j$ to $V$ via the map $\Phi$.
By our choice of $r_1$, we have that on the line $\q\{r'\omega: r_2\leq r'\leq r_3\w\}$, $\widetilde{Y}_j=Y_j$.  On the other hand, if we write:
$$\widetilde{Y}_j=\partial_{u_j}+\tilde{a}_j^k \partial_{u_k}$$
then the coefficients $\tilde{a}_j^k$ satisfy the differential
equation (\ref{EqnThmDiffEq}) (this follows just as before).  Away from $r=0$ this is a standard
ODE, so standard uniqueness theorems (say using
Gronwall's inequality) show that $\widetilde{Y}_j=Y_j$ on the line
$\q\{r'\omega: r_2\leq r'\leq r_4\w\}$.  This contradicts our
choice of $r_1$.
\end{proof}

From here, Theorem \ref{ThmPushForwardY} will follow immediately from the
following theorem:
\begin{thm}\label{ThmEstJac}
For all $t\in B_{n_0}\q( \eta_1\w)$,
$$\q|\det_{n_0\times n_0} d\Phi\q( t\w)\w| \approx \q|\det_{n_0\times n_0} d\Phi\q( 0\w)\w| = \q|\det_{n_0\times n_0} X\q( x_0\w)\w|.$$
\end{thm}

Theorem \ref{ThmEstJac}, in turn, follows immediately from the following
lemma and a simple continuity argument:
\begin{lemma}
Fix $\omega\in S^{n_0-1}$, $0<r_0<\eta_1$.  Suppose for $0\leq r\leq r_0$,
$\q|\det_{n_0\times n_0} d\Phi \q( r\omega\w) \w|\ne 0$.  Then, for all
$0\leq r\leq r_0$,
$$\q|\det_{n_0\times n_0} d\Phi\q( r\omega\w)\w|\approx \q|\det_{n_0\times n_0} d\Phi \q( 0 \w)\w|$$
Here, the implicit constants depend on neither $r_0$ nor $\omega$.
\end{lemma}
\begin{proof}
By Lemma \ref{LemmaPushForwardY}, for all $0\leq r\leq r_0$, we have 
$$d\Phi\q( Y_j\w)\q( \Phi\q(r\omega\w)\w) = X_j\q( \Phi\q( r\omega\w)\w).$$
Rewriting this in matrix notation, we have:
$$d\Phi\q(\q(I+A\w)\grad_u\w)\q(\Phi\q(r\omega\w)\w) = X\q( \Phi\q(r\omega\w)\w).$$
Thus, by applying (\ref{EqnDetsEqual}), we have at the point $\Phi\q( r\omega\w)$:
\begin{equation*}
\begin{split}
\q|\det_{n_0\times n_0} X\w| &= \q|\det_{n_0\times n_0} d\Phi\q(I+A\w)\w|\\
&= \sqrt{\det\q(\q(I+A\w)^{t} d\Phi^{t} d\Phi \q(I+A\w)\w)}\\
&= \q|\det \q(I+A\w)\w|\sqrt{\det\q(d\Phi^t d\Phi\w)}\\
&= \q|\det \q(I+A\w)\w|\q|\det_{n_0\times n_0} d\Phi\w|.
\end{split}
\end{equation*}
However, we have $\q\|A\w\|\leq \frac{1}{2}$, and so we have:
\begin{equation*}
\q|\det_{n_0\times n_0}d\Phi\q(r\omega\w)\w|\approx \q|\det_{n_0\times n_0} X\q(\Phi\q( r\omega\w)\w)\w|.
\end{equation*}
Applying Lemma \ref{LemmaDetsDontChange}, we see that
$$\q|\det_{n_0\times n_0} X\q(\Phi\q( r\omega\w)\w)\w| \approx \q|\det_{n_0\times n_0} X\q( x_0\w)\w|=\q|\det_{n_0\times n_0} d\Phi\q(0\w)\w|,$$
completing the proof.
\end{proof}

\begin{prop}\label{PropEstCmOfA}
We have, for $m\geq 0$:
$$\Cjn{m+1}{B_{n_0}\q( \eta_1\w)}{f} \approx_{m} \sum_{\q|\alpha\w|\leq m+1} \Cjn{0}{B_{n_0}\q(\eta\w)}{Y^\alpha f}$$
and,
$$\Cjn{m}{B_{n_0}\q(\eta_1\w);\mathbb{M}_{n_0\times n_0}\q(R\w)}{A}\lesssim_m 1.$$
It immediately follows that:
$$\Cjn{m}{B_{n_0}\q(\eta_1\w)}{Y_j}\lesssim_m 1$$
in particular,
$$\Cjn{2}{B_{n_0}\q(\eta_1\w)}{Y_j}\lesssim 1.$$
\end{prop}
\begin{proof}
We prove the result by induction.  Our base case will be $m=0$.
We already know,
$$\Cjn{0}{B_{n_0}\q(\eta_1\w)}{A}\leq \frac{1}{2}\lesssim_0 1.$$
Recall, $\lesssim_0,\lesssim_1,\lesssim_2$, and $\lesssim$ all mean
the same thing.
For notational convenience, write the operator:
$$\grad_Y f = \q( Y_1 f, \ldots, Y_{n_0} f\w).$$
Since $\grad_Y= \q( I+A\w)\grad_u$, and $\Cjn{0}{B_{n_0}\q(\eta_1\w);\mathbb{M}_{n_0\times n_0}\q(\R\w)}{I+A}\lesssim 1$, it follows that:
$$\sum_{\q|\alpha\w|\leq 1} \Cjn{0}{B_{n_0}\q(\eta_1\w)}{Y^\alpha f}\lesssim_0 \Cjn{1}{B_{n_0}\q( \eta_1\w)}{f}.$$
Conversely, since $\grad_u = \q( I+A\w)^{-1}\grad_Y$ and $\Cjn{0}{B_{n_0}\q(\eta_1\w);\mathbb{M}_{n_0\times n_0}\q(\R\w)}{\q(I+A\w)^{-1}}\lesssim 1$ (which can be seen by writing $\q( I+A\w)^{-1}$ as a Neumann series), we have:
$$\Cjn{1}{B_{n_0}\q( \eta_1\w)}{f} \lesssim_{0} \sum_{\q|\alpha\w|\leq 1} \Cjn{0}{B_{n_0}\q(\eta_1\w)}{Y^\alpha f}.$$

Suppose, for induction, that we have:
$$\Cjn{m}{B_{n_0}\q( \eta_1\w)}{f} \approx_{m-1} \sum_{\q|\alpha\w|\leq m} \Cjn{0}{B_{n_0}\q(\eta\w)}{Y^\alpha f}.$$
Then, note,
\begin{equation*}
\Cjn{m}{B_{n_0}\q( \eta_1\w)}{\ct_{i,j}^k}\approx_{m-1} \sum_{\q|\alpha\w|\leq m} \Cjn{0}{B_{n_0}\q(\eta\w)}{Y^\alpha \ct_{i,j}^k}.
\end{equation*}
But, 
$$Y^\alpha\ct_{i,j}^k = \q( X^\alpha c_{i,j}^k\w)\circ \Phi$$
and hence,
\begin{equation*}
\sum_{\q|\alpha\w|\leq m} \Cjn{0}{B_{n_0}\q(\eta\w)}{Y^\alpha \ct_{i,j}^k} \lesssim_m 1
\end{equation*}
and we have that:
$$\Cjn{m}{B_{n_0}\q( \eta_1\w)}{\ct_{i,j}^k}\lesssim_m 1$$
for all $i,j,k$.
It follows from Theorem \ref{ThmExistUniqDiffEq} that:
$$\Cjn{m}{B_{n_0}\q(\eta_1\w);\mathbb{M}_{n_0\times n_0}\q(R\w)}{A}\lesssim_m 1.$$
And thus, we have, using the Neumann series, that:
\begin{equation*}
\Cjn{m}{B_{n_0}\q(\eta_1\w);\mathbb{M}_{n_0\times n_0}\q(\R\w)}{\q(I+A\w)^{-1}}, \quad  \Cjn{m}{B_{n_0}\q(\eta_1\w);\mathbb{M}_{n_0\times n_0}\q(\R\w)}{I+A}\lesssim_m 1.
\end{equation*}
Hence, since $\grad_u =\q( I+A\w)^{-1} \grad_Y$ and $\grad_Y=\q( I+A\w)\grad_u$,
it follows easily that:
$$\Cjn{m+1}{B_{n_0}\q( \eta_1\w)}{f} \approx_{m} \sum_{\q|\alpha\w|\leq m+1} \Cjn{0}{B_{n_0}\q(\eta\w)}{Y^\alpha f}.$$
\end{proof}

Now we turn our attention to showing that if we shrink $\eta_1$ enough, while
still keeping it admissible, we have that $\Phi$ is injective on $B_{n_0}\q( \eta_1\w)$.  This result is essentially contained in \cite{TaoWrightLpImprovingBoundsForAverages} (see p. 622 of that reference), however we recreate
the proof below for completeness, and to make it
clear why each constant is admissible.  Thus, the next
lemma and proposition follow \cite{TaoWrightLpImprovingBoundsForAverages}.

\begin{lemma}\label{LemmaTWNoFixed}
Suppose $Z$ is a $C^1$ vector field on an open subset $V\subseteq \R^n$,
and $U\subseteq V$.
Then, there exists a $\delta>0$, depending only on $n$, such that
if $\Cjn{1}{U}{Z}\leq \delta$, then there does not exist $x_1\in U$ with:
\begin{itemize}
\item $e^{tZ}x_1\in U$, $0\leq t\leq 1$,
\item $e^Z x_1=x_1$,
\item $Z\q( x_1\w) \ne 0$.
\end{itemize}
\end{lemma}
\begin{proof}
Suppose the lemma does not hold, and we have an $x_1$ and $Z$ as above.
In the proof of this lemma, we will use big-$O$ notation--the implicit
constants will only depend on $n$.
Differentiating the identity:
$$\frac{d}{dt} e^{tZ} x_1 = Z\q( e^{tZ} x_1\w)$$
we obtain:
$$\frac{d^2}{dt^2} e^{tZ}x_1 = O\q( \delta \q|\frac{d}{dt} e^{tZ} x_1\w|\w).$$
Thus, by Gronwall's inequality:
$$\frac{d}{dt} e^{tZ}x_1  = O\q(\q|\frac{d}{dt} e^{tZ}x_1\bigg|_{t=0}\w|\w)= O\q(\q|Z\q( x_1\w)\w|\w)$$
for $t\leq 1$.  Hence,
$$\frac{d^2}{dt^2} e^{tZ}x_1 = O\q( \delta \q|Z\q(x_1\w)\w|\w).$$
Integrating, we obtain:
$$\frac{d}{dt} e^{tZ}x_1 = Z\q( x_1\w) + O\q( \delta \q| t\w| \q|Z\q( x_1\w)\w| \w).$$
Integrating again, we obtain:
$$x_1 = e^{Z} x_1 = x_1+ Z\q( x_1\w) + O\q(\delta \q|Z\q( x_1\w)\w|\w).$$
which is impossible if $\delta$ is sufficiently small, completing
the proof.
\end{proof}

\begin{prop}\label{PropPhiInjec}
We may shrink $\eta_1$, while still keeping it admissible, to ensure that
$\Phi$ is injective on $B_{n_0}\q( \eta_1\w)$.
\end{prop}
\begin{proof}
We will construct an admissible constant $\eta_2$ with the properties desired
in the statement of the proposition, and then the proof will be completed
by renaming $\eta_2$, $\eta_1$.  %The proof here is essentially
%the argument used in \cite{TaoWrightLpImprovingBoundsForAverages}.  We
%include it here for clarity and completeness.

Consider the maps $\Psi_{u_0}\q(u\w) = e^{u\cdot Y}u_0$, defined
for $\q|u_0\w|,\q|u\w|\leq \eta'$, where $\eta'>0$ is some sufficiently small admissible constant.  Notice, since
$$\Cjn{2}{B_{n_0}\q(\eta_1\w)}{Y_j}\lesssim 1$$ 
we have by Theorem \ref{ThmExpReg} that $\Psi_{u_0}\in C^2$
with $C^2$ norm admissibly bounded uniformly in $u_0$.
Furthermore, since $d\Psi_{u_0}\q( 0\w) = \q( I+ A\q( u_0\w)\w)$,
and $\q\|A\w\|\leq \frac{1}{2}$,
we have that $\q|\det d\Psi_{u_0}\q( 0\w)\w|\gtrsim 1$, uniformly
in $u_0$.

Hence we apply the uniform inverse function theorem (Theorem \ref{ThmInverseFunctionThm})
to see that there exist admissible constants $\eta_2>0$, $\delta>0$ such
that for all $u_1,u_2\in B_{n_0}\q(\eta_2\w)$ there exists $u_0\in B_{n_0}\q(\delta\w)$
with $u_2=\Psi_{u_1}\q( u_0\w)$.  Moreover, by shrinking $\eta_2$,
we may shrink $\delta$.

Now suppose $\Phi:B\q( \eta_2\w) \rightarrow \Bt{X}{d}{x_0}{\xi}$
is not injective.  Thus, there exist $u_1,u_2\in B_{n_0}\q(\eta_2\w), u_1\ne u_2$
such that
$$\Phi\q( u_1\w) =\Phi\q( u_2\w).$$
But, since there exists $0\ne u_0\in B_{n_0}\q( \delta\w)$ with $u_2= e^{u_0\cdot Y} u_1$, we have that:
$$\Phi\q( u_1\w) = \Phi\q( u_2\w) = e^{u_0\cdot X} \Phi\q( u_1\w).$$
Setting $Z=u_0\cdot X$, we have by Lemma \ref{LemmaDetsDontChange}
that $Z$ is non-zero on $\B{X}{d}{x_0}{\xi}$.  Applying Lemma \ref{LemmaTWNoFixed}, we see
that by taking $\delta$ admissibly small enough (and therefore $\eta_2$
admissibly small enough), we achieve a contradiction.
\end{proof}

Our proof of Theorem \ref{ThmMainFrobThm} will now be completed
by the following proposition:
\begin{prop}[\cite{NagelSteinWaingerBallsAndMetricsDefinedByVectorFields}, Lemma 2.16]\label{PropCanPutBallInPhi}
There exists an admissible constant $\xi_1>0$, such that
$$\B{X}{d}{x_0}{\xi_1}\subseteq \Phi\q(B_{n_0}\q(\eta_1\w)\w).$$
\end{prop}
\begin{proof}
Actually, the proof in \cite{NagelSteinWaingerBallsAndMetricsDefinedByVectorFields} proves something more general.  In our case, though, we have already
shown that $\Phi$ is injective (Proposition \ref{PropPhiInjec}), and this simplifies matters, somewhat.  We include this simplified proof, and refer the reader
to \cite{NagelSteinWaingerBallsAndMetricsDefinedByVectorFields} for the
stronger results.

Fix $\xi_1>0$.  Suppose $y\in \B{X}{d}{x_0}{\xi_1}$.  Thus, there
exists $\phi:\q[0,1\w]\rightarrow \B{X}{d}{x_0}{\xi_1}$, $\phi\q(0\w)=x_0$,
$\phi\q( 1\w) =y$,
$$\phi'\q(t\w) = \q(b\cdot X\w) \q( \phi\q(t\w)\w)$$
with $b\in \Lpp{\infty}{\q[0,1\w]}^{n_0}$, $\Lppn{\infty}{\q[0,1\w]}{\q|\xi_1^{-d} b\w|}<1$.

Define
$$\mathcal{T} = \q\{t\leq 1: \phi\q(t'\w)\in \Phi\q(B_{n_0}\q( \frac{\eta_1}{2}\w)\w), \forall 0\leq t'\leq t\w\}.$$
Let $t_0= \sup{\mathcal{T}}$.  We want to show that, by taking $\xi_1$ admissibly
small enough we have that $t_0=1$ and $\phi\q(1\w)\in \Phi\q(B_{n_0}\q( \frac{\eta_1}{2}\w)\w)$.

Suppose not.  Then, we must have that $\q| \Phi^{-1} \q( \phi\q(t_0\w) \w)\w| =\frac{\eta_1}{2}$.  Then, we have:
\begin{equation*}
\begin{split}
\frac{\eta_1}{2} &= \q| \Phi^{-1} \q( \phi\q(t_0\w) \w)\w|\\
&=\q|\int_0^{t_0} \frac{d}{dt} \Phi^{-1}\q(\phi\q(t\w)\w)\w|\\
&= \q| \int_0^{t_0} \q(\q(b\cdot X\w)\Phi^{-1}\w)\q(\phi\q(t\w)\w) \w|\\
&= \q|\int_0^{t_0} \q( b\cdot Y\w) \q( \Phi^{-1}\q(\phi\q(t\w)\w)\w)\w|\\
&<\frac{\eta_1}{2},
\end{split}
\end{equation*}
provided $\xi_1$ is admissibly small enough.  In the second to last line,
$$\q( b\cdot Y\w) \q( \Phi^{-1}\q(\phi\q(t\w)\w)\w)$$
denotes the vector $b\cdot Y$ evaluated at the point $\Phi^{-1}\q(\phi\q(t\w)\w)$.
This achieves the contradiction and completes the proof.
\end{proof}

%% file: unitscalenew.tex
In this section, we generalize Theorem \ref{ThmMainFrobThm} to the
case when the vector fields may not be linearly independent--thereby
completing the proof of Theorem \ref{ThmIntroFrob}.
Suppose $X=\q( X_1,\ldots, X_q\w)$ are $q$ $C^1$ vector fields
with associated single-parameter formal degrees $d=\q( d_1,\ldots, d_q\w)\in \q(0,\infty\w)^{q}$, defined
on the fixed connected open set $\Omega\subseteq \R^n$.  
%Let $J_0\in \sI{n_0}{q}$, and
Fix $1\geq \xi>0$,
$x_0\in \Omega$.
Let $n_0=\dim \Span{X_1\q(x_0\w),\ldots, X_q\q( x_0\w)}$.
Fix $1\geq \zeta>0$, $J_0\in \sI{n_0}{q}$; we assume that:
\begin{equation}\label{EqnJ0IsMaximal}
\q|\det_{n_0\times n_0} X_{J_0}\q( x_0\w)\w|_{\infty}\geq \zeta \sup_{J\in \sI{n_0}{q}} \q|\det_{n_0\times n_0} X_J\q( x_0\w)\w|_{\infty}.
\end{equation}
Recall if $J_0=\q( j_1,\ldots, j_{n_0}\w)$, $\q(X,d\w)_{J_0}$ denotes the list
with formal degrees $\q( \q( X_{j_1}, d_{j_1}\w),\ldots, \q( X_{j_{n_0}}, d_{j_{n_0}}\w)\w)$.  We also write $X_{J_0}$ to denote the list of vector
field $X_{j_1},\ldots, X_{j_{n_0}}$, and $d_{J_0}$ to denote the list
of formal degrees $d_{j_1},\ldots, d_{j_{n_0}}$.
Suppose, further, that $\q(X,d\w)_{J_0}$ satisfies $\sC\q( x_0,\xi\w)$.
%Recall if $J_0=\q( j_1,\ldots, j_{n_0}\w)$, $\q(X,d\w)_{J_0}$ denotes the list
%with formal degrees $\q( \q( X_{j_1}, d_{j_1}\w),\ldots, \q( X_{j_{n_0}}, d_{j_{n_0}}\w)\w)$.  We will also write $X_{J_0}$ to denote the list of vector
%field $X_{j_1},\ldots, X_{j_{n_0}}$.
In addition, suppose that the $X_j$s satisfy an integrability condition
on $\Bsub{X}{d}{J_0}{x_0}{\xi}$ given by:
\begin{equation}\label{EqnUnitScaleIntegCont}
\q[X_j,X_k\w] = \sum_{l} c_{j,k}^l X_l.
\end{equation}
%Let $n_0=\dim \Span{X_1\q(x_0\w),\ldots, X_q\q( x_0\w)}$.
%Fix $1\geq \zeta>0$, $J_0\in \sI{n_0}{q}$; we assume that:
%$$\q|\det_{n_0\times n_0} X_{J_0}\q( x_0\w)\w|\geq \zeta \sup_{J\in \sI{n_0}{q}} \q|\det_{n_0\times n_0} X_J\q( x_0\w)\w|$$
%Recall, if $J=\q( j_1, \ldots, j_{n_0}\w)$, $X_J$ is the list of vector fields
%$X_{j_1},\ldots, X_{j_{n_0}}$.  
%Without loss of generality, we 
%reorder the vector fields so that $J_0=\q(1,\ldots, n_0\w)$.
Without loss of generality, we assume for the remainder of the section
that $J_0=\q( 1,\ldots, n_0\w)$.
We will also assume that:
\begin{itemize}
\item For $1\leq j\leq n_0$, $X_j$ is $C^2$ on $\Bsub{X}{d}{J_0}{x_0}{\xi}$ and satisfies $\Cjn{2}{\Bsub{X}{d}{J_0}{x_0}{\xi}}{X_j}<\infty$.
%\item For $n_0< j\leq q$, $X_j$ satisfies $\Cjn{1}{\B{X}{d}{x_0}{\xi}}{X_j}<\infty$.
\item For $\q| \alpha \w|\leq 2$, $1\leq i,j\leq n_0$, $1\leq k\leq q$, $X_{J_0}^{\alpha} c_{i,j}^k\in C^{0}\q(\B{X}{d}{x_0}{\xi}\w)$, and
$$\sum_{\q|\alpha\w|\leq 2} \Cjn{0}{\Bsub{X}{d}{J_0}{x_0}{\xi}}{X_{J_0}^\alpha c_{i,j}^k}<\infty.$$
\item For $\q| \alpha \w|\leq 1$, $1\leq i,j,k\leq q$, $X_{J_0}^{\alpha} c_{i,j}^k\in C^{0}\q(\B{X}{d}{x_0}{\xi}\w)$, and
$$\sum_{\q|\alpha\w|\leq 1} \Cjn{0}{\Bsub{X}{d}{J_0}{x_0}{\xi}}{X_{J_0}^\alpha c_{i,j}^k}<\infty.$$
\end{itemize}

We will say that $C$ is an admissible constant if $C$ can be chosen to
depend only on a fixed upper bound, $d_{max}<\infty$, for $d_1,\ldots, d_q$,
a fixed lower bound $d_{min}>0$ for $d_1,\ldots, d_q$, a fixed upper bound
for $n$ and $q$ (and therefore for $n_0$), a fixed lower bound, $\xi_0>0$,
for $\xi$, a fixed lower bound, $\zeta_0>0$, for $\zeta$, and a fixed
upper bound for the quantities:
\begin{equation*}
\begin{split}
\Cjn{2}{\Bsub{X}{d}{J_0}{x_0}{\xi}}{X_j}, &\quad 1\leq j\leq n_0, \\
%\Cjn{1}{\Bsub{X}{d}{J_0}{x_0}{\xi}}{X_j}, &\quad n_0< j\leq q \\
\sum_{\q|\alpha\w|\leq 2} \Cjn{0}{\Bsub{X}{d}{J_0}{x_0}{\xi}}{X_{J_0}^\alpha c_{i,j}^k}, & \quad 1\leq i, j\leq n_0, \quad 1\leq k\leq q,\\
\sum_{\q|\alpha\w|\leq 1} \Cjn{0}{\Bsub{X}{d}{J_0}{x_0}{\xi}}{X_{J_0}^\alpha c_{i,j}^k}, & \quad 1\leq i, j, k\leq q.\\
\end{split}
\end{equation*}

Furthermore, if we say that $C$ is an $m$-admissible constant, we mean
that in addition to the above, we assume that:
\begin{itemize}
\item $\Cjn{m}{\Bsub{X}{d}{J_0}{x_0}{\xi}}{X_j}<\infty$, for every $1\leq j\leq n_0,$
\item $\sum_{\q|\alpha\w|\leq m}\Cjn{0}{\Bsub{X}{d}{J_0}{x_0}{\xi}}{X_{J_0}^\alpha c_{i,j}^k}<\infty$, for every $1\leq i, j\leq n_0$, $1\leq k\leq q$, 
\item $\sum_{\q|\alpha\w|\leq m-1}\Cjn{0}{\Bsub{X}{d}{J_0}{x_0}{\xi}}{X_{J_0}^\alpha c_{i,j}^k}<\infty$, for every $1\leq i, j,k\leq q$. 
\end{itemize}
(in particular, the above partial derivatives exist and are continuous).  $C$ is allowed to depend on $m$, all
the quantities an admissible constant is allowed to depend on, and
a fixed upper bound for the above quantities.  
%As before,
%we write $A\lesssim_m B$ for $A\leq C B$ where $C$ is an $m$-admissible
%constant, and we write $A\approx_m B$ for $A\lesssim_m B$ and $B\lesssim_m A$.
Note that, as before, $\lesssim_0, \lesssim_1,\lesssim_2$, and $\lesssim$
all denote the same thing.

For $\eta>0$, a sufficiently small admissible constant, define the map:
$$\Phi:B_{n_0}\q(\eta\w) \rightarrow \Btsub{X}{d}{J_0}{x_0}{\xi}$$
by
$$\Phi\q(u\w) = \exp\q(u\cdot X_{J_0}\w)x_0.$$
The main results of this section are the following:

\begin{thm}\label{ThmMainUnitScale}
There exist admissible constants $\eta_1>0$, $\xi_1\geq \xi_2>0$, such that:
\begin{itemize}
\item $\Phi:B_{n_0}\q(\eta_1\w)\rightarrow \Btsub{X}{d}{J_0}{x_0}{\xi}$ is one-to-one.
\item For all $u\in B_{n_0}\q(\eta_1\w)$, $\q|\det_{n_0\times n_0} d\Phi\q( u\w)\w|\approx \q|\det_{n_0\times n_0} X\q( x_0\w)\w|$.
\item $\B{X}{d}{x_0}{\xi_2}\subseteq\Bsub{X}{d}{J_0}{x_0}{\xi_1}\subseteq\Phi\q(B_{n_0}\q(\eta_1\w)\w)\subseteq \Btsub{X}{d}{J_0}{x_0}{\xi}\subseteq \Bsub{X}{d}{J_0}{x_0}{\xi}\subseteq \B{X}{d}{x_0}{\xi}$.
\end{itemize}
Furthermore, if we let $Y_j$ ($1\leq j\leq q$) be the pullback of $X_j$ under the map
$\Phi$, then we have:
%$$\Cjn{2}{B_{n_0}\q(\eta_1\w)}{Y_j}\lesssim 1$$
%and
$$\Cjn{m}{B_{n_0}\q(\eta_1\w)}{Y_j}\lesssim_m 1$$
in particular,
$$\Cjn{2}{B_{n_0}\q(\eta_1\w)}{Y_j}\lesssim 1.$$
Finally, if for $u\in B_{n_0}\q( \eta_1\w)$ we define the $n_0\times n_0$ matrix
$A\q( u\w)$ by:\footnote{Recall,
we have, without loss of generality, assumed $J_0=\q( 1,\ldots,n_0\w)$.}
$$\q( Y_1,\ldots, Y_{n_0}\w) = \q( I+A\w) \grad_u$$
then,
$$\sup_{u\in B_{n_0}\q( \eta_1\w) } \q\| A\q( u\w)\w\|\leq \frac{1}{2}.$$
\end{thm}

\begin{cor}\label{CorMainUnitScaleTwice}
Let $\eta_1,\xi_1,\xi_2$ be as in Theorem \ref{ThmMainUnitScale}.  Then,
there exist admissible constants $0<\eta_2<\eta_1$, $0<\xi_4\leq \xi_3<\xi_2$ such that:
\begin{equation*}
\begin{split}
&\B{X}{d}{x_0}{\xi_4}\subseteq 
\Bsub{X}{d}{J_0}{x_0}{\xi_3}\subseteq 
\Phi\q(B_{n_0}\q(\eta_2\w)\w)\\&\subseteq 
\Btsub{X}{d}{J_0}{x_0}{\xi_2}\subseteq 
\Bsub{X}{d}{J_0}{x_0}{\xi_2}\subseteq
\B{X}{d}{x_0}{\xi_2}\\&\subseteq
\Bsub{X}{d}{J_0}{x_0}{\xi_1}\subseteq
\Phi\q(B_{n_0}\q(\eta_1\w)\w)\subseteq 
\Btsub{X}{d}{J_0}{x_0}{\xi}\\&\subseteq 
\Bsub{X}{d}{J_0}{x_0}{\xi}\subseteq 
\B{X}{d}{x_0}{\xi},
\end{split}
\end{equation*}
and $\Vol{\B{X}{d}{x_0}{\xi_2}}\approx \q|\det_{n_0\times n_0} X\q(x_0\w)\w|$,
where $\Vol{A}$ denotes the induced Lebesgue volume on the leaf generated by the
$X_j$s, passing through the point $x_0$.
\end{cor}

\begin{cor}\label{CorBumpFuncAtUnitScale}
Take $\xi_4$ as in Corollary \ref{CorMainUnitScaleTwice}.
Then, there exists $\phi\in C_0^2\q(\B{X}{d}{x_0}{\xi}\w)$ (here, we mean
$C^2$ as thought of as a function on the leaf), which
equals $1$ on $\B{X}{d}{x_0}{\xi_4}$ and satisfies:
$$\q|X^\alpha \phi\w|\lesssim_{\q(\q|\alpha\w|-1\w)\vee 0} 1$$
for every ordered multi-index $\alpha$.
\end{cor}

\begin{rmk}
Later in the paper we will apply Corollaries \ref{CorMainUnitScaleTwice}
and \ref{CorBumpFuncAtUnitScale} without explicitly saying what
$J_0$ and $\zeta$ are.  In these cases, we are choosing $\zeta=1$
and $J_0$ such that:
$$\q|\det_{n_0\times n_0} X_{J_0}\q( x_0\w)\w|_\infty = \q|\det_{n_0\times n_0} X\q( x_0\w)\w|_{\infty}.$$
%\sup_{J\in \sI{n_0}{q}}\q|\det_{n_0\times n_0} X_J\q( x_0\w)\w|_\infty$$
\end{rmk}

\begin{rmk}\label{RmkSymmetricUnitScaleAssump}
In our definition of admissible constants, we have assumed greater
regularity on $X_1,\ldots, X_{n_0}$ than on $X_{j}$, $n_0<j\leq q$.
In many applications, it is easier to just assume more symmetric regularity
assumptions, that imply the assumptions of this section.  Later in the paper
we sometimes assume the following, stronger hypotheses:
\begin{itemize}
\item $\q( X,d\w)$ satisfies $\sC\q( x_0,\xi\w)$, and
(\ref{EqnUnitScaleIntegCont}) holds on $\B{X}{d}{x_0}{\xi}$.
\item In addition to everything that they are allowed to depend on in
this section, $m$-admissible constants (for $m\geq 2$) can depend on a fixed upper bound
for the quantities:
$$\Cjn{m}{\B{X}{d}{x_0}{\xi}}{X_l},\quad \sum_{\q|\alpha\w|\leq m}\Cjn{0}{\B{X}{d}{x_0}{\xi}}{X^\alpha c_{i,j}^k}$$
where $1\leq i,j,k,l\leq q$, and these derivatives are assumed to
exist, and the norms are assumed to be finite.  Admissible constants are
defined to be $2$-admissible constants.
\end{itemize}
\end{rmk}

\begin{rmk}
Just as in Remark \ref{RmkdsDontMatter}, the $d_j$s do not play
an essential role in this section.
%, and Theorem \ref{ThmMainUnitScale}
%and its corollaries are more naturally considered in the case when
%$d_1=d_2=\ldots=d_q=1$.  However, we have included the $d_j$s
%to facilitate our applications in Section \ref{SectionMultiParamBalls}.
\end{rmk}

\begin{rmk}
As mentioned in the Section \ref{SectionResults}, ``at the unit scale'' in the title
of this section refers to the unit scale with respect to the
vector fields $X_j$.  Thus, if the vector fields $X_j$ are
very small, one can think of the results in this section as
taking place at a very small scale.
\end{rmk}

\begin{rmk}
The observant reader may have noticed that we made no {\it a priori}
bound on $X_j$, $n_0<j\leq q$.  However, we will see using Cramer's
rule that (\ref{EqnJ0IsMaximal}) implies a bound for $X_j$ at $x_0$.
In addition, we will be able to use Gronwall's inequality to obtain
bounds at points other than $x_0$ (see (\ref{EqnXkAsSumOfSubs})).
The reader may wonder, though, that since we have assumed no
{\it a priori} bound for the $C^1$ norm of $X_j$ ($n_0<j\leq q$),
do we need to insist that they are $C^1$?  The answer is partially
no, though our definitions only makes sense when the $X_j$ are
all assumed to be $C^1$.  We will see that the above assumptions
will show that $X_1,\ldots, X_{n_0}$ are integrable (see Proposition \ref{PropSubSatisHyp}), and from there
all we need is that $X_j$ ($n_0<j\leq q$) is $C^1$ on the leaf generated by
$X_1,\ldots, X_{n_0}$, passing through $x_0$.
This perspective is
taken up in Section \ref{SectionControlUnitScale}; in fact, 
the main reason we have been careful to {\it not} assume
a bound on the $C^1$ norm of $X_j$ ($n_0<j\leq q$) in this
section, is to make clear how these arguments also
work in the setup of Section \ref{SectionControlUnitScale}.
%the results in Section \ref{SectionControlUnitScale}
%are the main reason we have been careful to {\it not} assume
%a bound on the $C^1$ norm of $X_j$ ($n_0<j\leq q$) in this
%section.
\end{rmk}

Before we prove Theorem \ref{ThmMainUnitScale}, let us first see
how it implies the two corollaries.

\begin{proof}[Proof of Corollary \ref{CorMainUnitScaleTwice}]
We obtain $\eta_1,\xi_1, \xi_2$ from Theorem \ref{ThmMainUnitScale}.  Then,
apply Theorem \ref{ThmMainUnitScale} again with $\xi_2$ in place of
$\xi$ to complete the proof of the first part of the corollary.

By the above containments, we have:
$$\Vol{\Phi\q(B_{n_0}\q(\eta_2\w)\w)}\lesssim \Vol{\B{X}{d}{x_0}{\xi_2}}\lesssim \Vol{\Phi\q(B_{n_0}\q(\eta_1\w)\w)}.$$
Using (\ref{EqnChangeOfVars}) and the fact that
$$\q|\det_{n_0\times n_0} d\Phi\q(t\w)\w|\approx \q|\det_{n_0\times n_0} X\q(x_0\w)\w|$$
for all $t\in B_{n_0}\q(\eta_1\w)$, the estimate on the volume follows immediately.
\end{proof}

\begin{rmk}\label{RmkVolOfAllSmallBalls}
By a proof similar to the one of Corollary \ref{CorMainUnitScaleTwice},
we have that if $\xi'>0$, is a fixed admissible constant with $\xi'\leq \xi_2$, then,
$$\Vol{\B{X}{d}{x_0}{\xi'}}\approx  \q|\det_{n_0\times n_0} X\q(x_0\w)\w|.$$
\end{rmk}

\begin{proof}[Proof of Corollary \ref{CorBumpFuncAtUnitScale}]
Let $\psi\in C_0^\infty\q( B_{n_0}\q( \eta_1\w)\w)$, with $\psi=1$ on
$B_{n_0}\q(\eta_2\w)$.  Define
\begin{equation*}
\phi\q( x\w) = 
\begin{cases}
\psi\q(\Phi^{-1}\q( x\w)\w) & \text{if $x\in \Phi\q(B_{n_0}\q(\eta_1\w)\w)$,}\\
0 & \text{otherwise.}
\end{cases}
\end{equation*}
Then, we see:
\begin{equation*}
X^\alpha \phi\q(x\w) =
\begin{cases}
\q( Y^\alpha \psi\w) \q(\Phi^{-1}\q(x\w)\w) & \text{if $x\in \Phi\q(B_{n_0}\q(\eta_1\w)\w)$,}\\
0 & \text{otherwise.}
\end{cases}
\end{equation*}
Thus, to prove the corollary, it suffices to show that:
$$\q|Y^\alpha \psi\w|\lesssim_{\q(\q|\alpha\w|-1\w) \vee 0} 1$$
and this is obvious.
\end{proof}

We now turn to the proof of Theorem \ref{ThmMainUnitScale}.  The
main idea is to apply Theorem \ref{ThmMainFrobThm} to the vector fields
$\q( X,d\w)_{J_0}$.  
%To simplify notation, without loss of generality,
%we reorder the vector fields and the associated formal degrees
%to ensure $J_0=\q(1,2,\ldots, n_0\w)$.

\begin{lemma}\label{LemmaLieDerivOfDetCarnot}
Fix $1\leq n_1\leq n\wedge q$.  Then, for $1\leq j\leq n_0$, $I\in \sI{n_1}{n}$, $J\in \sI{n_1}{q}$, $x\in \Bsub{X}{d}{J_0}{x_0}{\xi}$,
$$\q| X_j \det X\q(x\w)_{I,J}\w|\lesssim \q|\det_{n_1\times n_1} X\q(x\w)\w|.$$
\end{lemma}
\begin{proof}
This can be proved by a simple modification of the proof for Lemma \ref{LemmaLieDerivOfDet}.  We leave the details to the reader.
\end{proof}

\begin{lemma}\label{LemmaDetsDontChangeCarnot}
For $y\in \Bsub{X}{d}{J_0}{x_0}{\xi}$, $1\leq n_1\leq q\wedge n$,
$$\q|\det_{n_1\times n_1} X\q( y\w)\w|\approx \q|\det_{n_1\times n_1} X\q( x_0\w)\w|.$$
In particular, for all $y\in \Bsub{X}{d}{J_0}{x_0}{\xi}$, $\dim \Span{X_1\q(y\w),\ldots X_q\q( y\w)} = n_0$.
\end{lemma}
\begin{proof}
This can be proved by a simple modification of the proof of Lemma \ref{LemmaDetsDontChange}, using Lemma \ref{LemmaLieDerivOfDetCarnot}.  We leave the details to the reader.
\end{proof}

Take $I_0\in \sI{n_0}{n}$ such that:
$$\q|\det X\q(x_0\w)_{I_0,J_0}\w|= \sup_{I\in \sI{n_0}{n}}\q|\det X\q(x_0\w)_{I,J_0}\w|.$$
\begin{lemma}\label{LemmaMaxDetDoesntChange}
There exists an admissible constant $\xi^1>0$, $\xi^1\leq \xi$ such that
for every $y\in \Bsub{X}{d}{J_0}{x_0}{\xi^1}$, we have:
\begin{equation*}
%\q\|\det X\q( y\w)_{I_0,J_0}\w|\gtrsim \sup_{\substack{I\in \sI{n_0}{n}\\ J\in \sI{n_0}{q}}} \q|\det X\q(y\w)_{I,J}\w|
\q|\det X\q( y\w)_{I_0,J_0}\w|\gtrsim \q|\det_{n_0\times n_0} X\q(y\w)\w|
\end{equation*}
\end{lemma}
\begin{proof}
Fix $I\in \sI{n_0}{n}$, $J\in \sI{n_0}{q}$.  Let $\gamma:\q[0,1\w]\rightarrow \Bsub{X}{d}{J_0}{x_0}{\xi}$ satisfy:
$$\gamma'\q(t\w) = \q( b\cdot X_{J_0}\w)\q(\gamma\q(t\w)\w)$$
with $b\in \Lpp{\infty}{\q[0,1\w]}^{n_0}$, $\Lppn{\infty}{\q[0,1\w]}{\q|\xi^{-d_{J_0}} b\w|}<1$.

By applying Lemmas \ref{LemmaLieDerivOfDetCarnot}, \ref{LemmaDetsDontChangeCarnot}, we see:
\begin{equation*}
\begin{split}
\frac{d}{dt} \q|\det X_{I,J}\q(\gamma\q(t\w)\w)\w|^2 &= \det X_{I,J}\q(\gamma\q(t\w)\w) \q(\q( b\cdot X_{J_0}\w)\det X_{I,J}\w)\q(\gamma\q(t\w)\w)\\
&\lesssim \q|\det_{n_0\times n_0} X\q(\gamma\q(t\w)\w)\w|^2\\
&\approx \q|\det_{n_0\times n_0} X\q( x_0\w)\w|^2\\
&\approx \q|\det_{n_0\times n_0} X\q( x_0\w)_{J_0}\w|^2\\
&\approx \q|\det X\q(x_0\w)_{I_0,J_0}\w|^2\\
%&\leq C \q|\det X\q(x_0\w)_{I_0,J_0}\w|^2
\end{split}
\end{equation*}
and therefore,
$$\frac{d}{dt} \q|\det X_{I,J}\q(\gamma\q(t\w)\w)\w|^2 \leq C \q|\det X\q(x_0\w)_{I_0,J_0}\w|^2$$
where $C$ is some admissible constant.  Thus, if $t\leq \frac{1}{2C}$,
we have:
$$\q|\det X\q(\gamma\q(t\w)\w)_{I_0,J_0}\w| \approx \q|\det X\q(x_0\w)_{I_0,J_0}\w|$$
and,
\begin{equation}
\begin{split}
\q| \det X\q(\gamma\q(t\w)\w)_{I,J}\w|&\lesssim \q|\det X\q(x_0\w)_{I,J}\w| + \q|\det X\q( x_0\w)_{I_0,J_0}\w|
\\&\lesssim \q|\det X\q(x_0\w)_{I_0,J_0}\w|
\\&\approx \q|\det X\q(\gamma\q(t\w)\w)_{I_0,J_0}\w|.
\end{split}
\end{equation}
We complete the proof by noting that there exists an admissible constant
$\xi^1>0$ such that for every point $y\in \Bsub{X}{d}{J_0}{x_0}{\xi^1}$ there is
a $\gamma$ of the above form and a $t\leq \frac{1}{2C}$ with $y=\gamma\q( t\w)$.
\end{proof}

\begin{lemma}\label{LemmaQuotientOfDets}
Fix $I\in \sI{n_0}{n}$, $J\in \sI{n_0}{q}$.  Then,
\begin{equation*}
\sum_{\q|\alpha\w|\leq m}\Cjn{0}{\Bsub{X}{d}{J_0}{x_0}{\xi^1}}{ X_{J_0}^{\alpha}\frac{\det X_{I,J}}{\det X_{I_0,J_0}  }}\lesssim_m 1.
\end{equation*}
\end{lemma}
\begin{proof}
For $m=0$ this follows from Lemma \ref{LemmaMaxDetDoesntChange}.
For $m>0$, we look back to the proof of Lemma \ref{LemmaLieDerivOfDet}.
There, it was shown that $X_j \det X_{I,J}$ could be written as a sum
of terms of the form
$$f\det X_{I',J'}$$
where $I'\in \sI{n_0}{n}$, $J'\in \sI{n_0}{q}$, and $f$ was either of the
form $c_{i,j}^k$ or $f$ was a derivative of a coefficient of $X_j$ ($1\leq j\leq n_0$).
From this, Lemma \ref{LemmaMaxDetDoesntChange}, and a simple induction,
the lemma follows easily.  We leave the proof to the interested reader.
\end{proof}

We now show that on $\Bsub{X}{d}{J_0}{x_0}{\xi^1}$, the vector fields $X_1,\ldots, X_{n_0}$ satisfy the hypotheses of Theorem \ref{ThmMainFrobThm}.  Recall,
we have assumed, without loss of generality, $J_0=\q( 1,\ldots, n_0\w)$.
\begin{prop}\label{PropSubSatisHyp}
For $1\leq i,j,k\leq n_0$, there exist functions $\ch_{i,j}^k\in C\q(\Bsub{X}{d}{J_0}{x_0}{\xi^1}\w)$ such that, for $1\leq i,j\leq n_0$:
\begin{equation*}
\q[X_i, X_j\w]= \sum_{k=1}^{n_0} \ch_{i,j}^k X_k.
\end{equation*}
These functions satisfy:
\begin{equation*}
\sum_{\q|\alpha\w|\leq m} \Cjn{0}{\Bsub{X}{d}{J_0}{x_0}{\xi^1}}{X_{J_0}^\alpha\ch_{i,j}^k} \lesssim_m 1.
\end{equation*}
\end{prop}
\begin{proof}
For $1\leq j,k\leq q$, let $X^{\q(j,k\w)}$ be the matrix obtained by
replacing the $j$th column of the matrix $X$ with $X_k$.  Note that:
$$\det X^{\q(j,k\w)}_{I_0,J_0} = \epsilon_{j,k} \det X_{I_0, J\q( j,k\w)}$$
where $\epsilon_{j,k}\in \q\{0,1,-1\w\}$, and $J\q( j,k\w)\in \sI{n_0}{q}$.
Thus, for any $1\leq k\leq q$, we may write, by Cramer's rule:
\begin{equation}\label{EqnXkAsSumOfSubs}
X_k = \sum_{l=1}^{n_0} \frac{\det X^{\q(l,k\w)}_{I_0,J_0}}{\det X_{I_0,J_0}} X_l
= \sum_{l=1}^{n_0} \epsilon_{l,k} \frac{\det X_{I_0,J\q(l,k\w)}}{\det X_{I_0,J_0}} X_l.
\end{equation}
Hence, we have, for $1\leq i,j\leq n_0$:
\begin{equation*}
\q[X_i, X_j\w] = \sum_{k=1}^q c_{i,j}^k X_k = \sum_{l=1}^{n_0} \q( \sum_{k=1}^q c_{i,j}^k \epsilon_{l,k} \frac{\det X_{I_0,J\q(l,k\w)}}{\det X_{I_0,J_0}} \w) X_l =: \sum_{l=1}^{n_0} \ch_{i,j}^l X_l.
\end{equation*}

Given the form of $\ch_{i,j}^k$, the desired estimates on the derivatives
follow immediately from Lemma \ref{LemmaQuotientOfDets}.
%We need only estimate $X^\alpha \ch_{i,j}^k$ where $0\leq\q|\alpha\w|$.
%For $\q|\alpha\w|=0$, this follows directly from Lemma \ref{LemmaMaxDetDoesntChange}.
%For $\q|\alpha\w|>0$, we look back at the proof of Lemma \ref{LemmaLieDerivOfDet}.
%There, it was shown that $X_j \det X_{I,J}$ could be written as a sum
%of terms of the form:
%$$f\det X_{I',J'}$$
%where $I'\in \sI{n_0}{n}$, $J'\in \sI{n_0}{q}$, and $f$ was either of the form $c_{i',j'}^{k'}$, or $f$ was a derivative of coefficients of $X_j$.  From this
%and Lemma \ref{LemmaMaxDetDoesntChange}
%it follows easily that:
%$$\sum_{\q|\alpha\w|\leq m} \Cjn{0}{\B{X}{d}{x_0}{\xi^1}}{X^\alpha \ct_{i,j}^k} \lesssim_m 1$$
\end{proof}

We now apply Theorem \ref{ThmMainFrobThm} to the list of vector fields
$\q( X,d\w)_{J_0}$ on the ball $\Bsub{X}{d}{J_0}{x_0}{\xi^1}$.  We obtain
$\xi_1$ and $\eta_1$ as in the statement of Theorem \ref{ThmMainUnitScale}.  We obtain 
$$\q|\det_{n_0\times n_0} d\Phi \q(t\w)\w|\approx \q|\det_{n_0\times n_0} X_{J_0}\q( x_0\w)\w|.$$
But, we know from our initial assumptions that
$$\q|\det_{n_0\times n_0} X_{J_0}\q( x_0\w)\w| \approx \q|\det_{n_0\times n_0} X\q( x_0\w)\w|.$$
For $1\leq j\leq n_0$, we have:
\begin{equation*}
\Cjn{m}{B_{n_0}\q(\eta_1\w)}{Y_j}\lesssim_m 1.
\end{equation*}
Hence, to complete the proof of Theorem \ref{ThmMainUnitScale}, we need
to show the existence of $\xi_2$ and prove the estimates on $Y_j$ for
$n_0<j\leq q$.  We begin with the latter:
\begin{prop}\label{PropRegularityOfYUnitScale}
$\Cjn{m}{B_{n_0}\q(\eta_1\w)}{Y_k}\lesssim_m 1$, for $n_0<k\leq q$.
\end{prop}
\begin{proof}
By (\ref{EqnXkAsSumOfSubs}), we see that we may write:
$$Y_k = \sum_{l=1}^{n_0} \epsilon_{l,k} \q(\frac{\det X_{I_0, J\q(l,k\w)}}{\det X_{I_0,J_0}}\circ \Phi\w) Y_l.$$
Since we already know the result for $Y_l$, $1\leq l \leq n_0$, it suffices to
show that:
$$\Cjn{m}{B_{n_0}\q(\eta_1\w)}{\frac{\det X_{I_0, J\q(l,k\w)}}{\det X_{I_0,J_0}}\circ \Phi}\lesssim_m 1.$$
By Proposition \ref{PropEstCmOfA}, it suffices to show that for
$\q|\alpha\w|\leq m$,
$$\Cjn{0}{B_{n_0}\q(\eta_1\w)}{Y_{J_0}^{\alpha} \frac{\det X_{I_0, J\q(l,k\w)}}{\det X_{I_0,J_0}}\circ \Phi}\lesssim_m 1.$$
%here $\alpha$ is an ordered multi-index, involving just the indices $1,\ldots, n_0$.
But,
$$Y_{J_0}^{\alpha} \frac{\det X_{I_0, J\q(l,k\w)}}{\det X_{I_0,J_0}}\circ \Phi = \q(X_{J_0}^{\alpha} \frac{\det X_{I_0, J\q(l,k\w)}}{\det X_{I_0,J_0}}\w)\circ \Phi.$$
From here, the result follows immediately from an application of Lemma \ref{LemmaQuotientOfDets}.
\end{proof}

We now conclude our proof of Theorem \ref{ThmMainUnitScale}, with
the following proposition:
\begin{prop}\label{PropBigBallInSmallerBall}
There exists an admissible constant $\xi_2>0$ such that:
$$\B{X}{d}{x_0}{\xi_2}\subseteq \Bsub{X}{d}{J_0}{x_0}{\xi_1}.$$
\end{prop}
\begin{proof}
Suppose $y\in \B{X}{d}{x_0}{\xi_2}$, where $\xi_2\leq \xi_1\leq \xi^1$ will be chosen
at the end of the proof.  Thus there exists a path $\gamma:\q[0,1\w]\rightarrow \B{X}{d}{x_0}{\xi_2}$, $\gamma\q( 0\w) = x_0$, $\gamma\q( 1\w) =y$,
$$\gamma'\q(t\w) = \q( b\cdot X\w) \q( \gamma\q(t\w)\w)$$
where $b\in \Lpp{\infty}{\q[0,1\w]}^q$ with $\Lppn{\infty}{\q[0,1\w]}{\q|\xi_2^{-d} b\w|}<1$.
Then, applying (\ref{EqnXkAsSumOfSubs}), we have:
\begin{equation*}
\begin{split}
\gamma'\q(t\w) &= \sum_{k=1}^q b_k\q(t\w) X_k\q(\gamma\q(t\w)\w) \\
&= \sum_{l=1}^{n_0} \q(\sum_{k=1}^q \epsilon_{l,k} b_k\q(t\w)\frac{\det X\q(\gamma\q(t\w)\w)_{I_0,J\q(l,k\w)}}{\det X\q(\gamma\q(t\w)\w)_{I_0,J_0}}\w) X_l\q(\gamma\q(t\w)\w)\\
&=: \sum_{l=1}^{n_0} a_l\q(t\w) X_l\q(\gamma\q(t\w)\w),
\end{split}
\end{equation*}
and if $\xi_2>0$ is admissibly small enough, by Lemma \ref{LemmaMaxDetDoesntChange}, we have that:
$$\Lppn{\infty}{\q[0,1\w]}{\sqrt{\sum_{l=1}^{n_0} \xi_1^{-2d_l} \q|a_l\w|^2}}<1,$$
proving that $y=\gamma\q(1\w)\in \Bsub{X}{d}{J_0}{x_0}{\xi_1}$.
\end{proof}

%\begin{rmk}\label{RmkSymmetricUnitScaleAssump}
%In our definition of admissible constants, we have assumed greater
%regularity on $X_1,\ldots, X_{n_0}$ than on $X_{j}$, $n_0<j\leq q$.
%In many applications, it is easier to just assume more symmetric regularity
%assumptions, that imply the assumptions of this section.  Later in the paper
%we sometimes assume the following, stronger hypotheses:
%\begin{itemize}
%\item $\q( X,d\w)$ satisfies condition $\q( x_0,\xi\w)$, and
%(\ref{EqnUnitScaleIntegCont}) holds on $\B{X}{d}{x_0}{\xi}$.
%\item In addition to everything that they are allowed to depend on in
%this section, $m$-admissible constants (for $m\geq 2$) can depend on a fixed upper bound
%for the quantities:
%$$\Cjn{m}{\B{X}{d}{x_0}{\xi}}{X_l},\quad \sum_{\q|\alpha\w|\leq m}\Cjn{0}{\B{X}{d}{x_0}{\xi}}{X^\alpha c_{i,j}^k}$$
%where $1\leq i,j,k,l\leq q$, and these derivatives are assumed to
%exist, and the norms are assumed to be finite.  Admissible constants are
%defined to be $2$-admissible constants.
%\end{itemize}
%\end{rmk}

%% file: controlunitscale.tex
We take all the same notation as in Section \ref{SectionUnitScale},
and define ($m$-)admissible constants in the same way.\footnote{We are
still assuming $J_0=\q( 1,\ldots, n_0\w)$.}  
The goal
of this section is to understand when we can add an additional
vector field with a formal degree $\q( X_{q+1}, d_{q+1}\w)$ ($d_{q+1}\in \q( 0,\infty\w)$) to the list of vector fields $\q( X,d\w)$ without
``adding anything new.''
In particular, we wish to not significantly increase the size of $\B{X}{d}{x_0}{\tau}$,
where $\tau$ is thought of as a fixed constant $\leq \xi$.
%Let $\dm>0$ be a fixed lower bound for $d_{q+1}$.

Let $X_{q+1}$ be a $C^1$ vector field on $\Bsub{X}{d}{J_0}{x_0}{\xi}$ 
(here we mean that $X_{q+1}$ is $C^1$ thought of as a function on
the leaf in which $\Bsub{X}{d}{J_0}{x_0}{\xi}$ lies; but it need not be tangent to
the leaf), 
and assign to it a formal degree $d_{q+1}\in \q( 0,\infty\w)$.  
Let $\q( \hX, \hd\w)$ denote the list of vector fields with formal
degrees:
$$\q( \q( X_1,d_1\w), \ldots, \q( X_{q+1},d_{q+1}\w)\w).$$ 
For an integer $m\geq 1$ we define three conditions which will
turn out to be equivalent (all parameters below are considered to
be elements of $\q( 0,\infty\w)$):
\begin{enumerate}
\item $\sPo{m}{\kappa_1}{\tau_1}{\sigma_1}{\sigma_1^m}$:
\begin{itemize}
\item $\q|\det_{n_0\times n_0} X\q( x_0\w)\w|_{\infty} \geq \kappa_1 \q|\det_{n_0\times n_0} \hX\q( x_0\w)\w|_{\infty}$
\item $\q|\det_{j\times j} \hX\q( x_0\w) \w|=0$, $n_0<j\leq n$.
\item There exist $c_{i,q+1}^j\in C^0\q(\Bsub{X}{d}{J_0}{x_0}{\tau_1}\w)$ such that
$$\q[X_i,X_{q+1}\w]=\sum_{j=1}^{q+1} c_{i,q+1}^j X_j, \quad \text{on }\Bsub{X}{d}{J_0}{x_0}{\tau_1}$$
with:
$$\sum_{\q|\alpha\w|\leq m-1} \Cjn{0}{\Bsub{X}{d}{J_0}{x_0}{\tau_1}}{X^\alpha c_{i,q+1}^j}\leq \sigma_1^m, \quad \Cjn{0}{\Bsub{X}{d}{J_0}{x_0}{\tau_1}}{c_{i,q+1}^j}\leq \sigma_1.$$
\end{itemize}
\item $\sPt{m}{\tau_2}{\sigma_2}{\sigma_2^m}$:  There exist $c_j\in C^0\q( \Bsub{X}{d}{J_0}{x_0}{\tau_2} \w)$
such that:
\begin{itemize}
\item $X_{q+1}=\sum_{j=1}^{n_0} c_j X_j$, on $\Bsub{X}{d}{J_0}{x_0}{\tau_2}$.
\item $\sum_{\q|\alpha\w|\leq m} \Cjn{0}{\Bsub{X}{d}{J_0}{x_0}{\tau_2}}{X^\alpha c_j}\leq \sigma_2^m$.
\item $\sum_{\q|\alpha\w|\leq 1} \Cjn{0}{\Bsub{X}{d}{J_0}{x_0}{\tau_2}}{X^\alpha c_j}\leq \sigma_2$.
\end{itemize}
\item $\sPr{m}{\tau_3}{\sigma_3}{\sigma_3^m}$:  There exist $c_j\in C^0\q( \Bsub{X}{d}{J_0}{x_0}{\tau_3} \w)$
such that:
\begin{itemize}
\item $X_{q+1}=\sum_{j=1}^{q} c_j X_j$, on $\Bsub{X}{d}{J_0}{x_0}{\tau_3}$.
\item $\sum_{\q|\alpha\w|\leq m} \Cjn{0}{\Bsub{X}{d}{J_0}{x_0}{\tau_3}}{X^\alpha c_j}\leq \sigma_3^m$.
\item $\sum_{\q|\alpha\w|\leq 1} \Cjn{0}{\Bsub{X}{d}{J_0}{x_0}{\tau_3}}{X^\alpha c_j}\leq \sigma_3$.
\end{itemize}
\end{enumerate}

\begin{thm}\label{ThmUnitScaleEquivCond}
$\sPn{1}{m}\Rightarrow \sPn{2}{m} \Rightarrow \sPn{3}{m}\Rightarrow \sPn{1}{m}$
in the following sense:
\begin{enumerate}
\item $\sPo{m}{\kappa_1}{\tau_1}{\sigma_1}{\sigma_1^m} \Rightarrow$ there
exist admissible constants $\tau_2=\tau_2\q(\kappa_1,\tau_1,\sigma_1\w)$,
$\sigma_2=\sigma_2\q( \kappa_1, \sigma_1\w)$, and an $m$-admissible
constant $\sigma_2^m=\sigma_2^m\q( \kappa_1, \sigma_1^m \w)$ such that
$\sPt{m}{\tau_2}{\sigma_2}{\sigma_2^m}$.
\item $\sPt{m}{\tau_2}{\sigma_2}{\sigma_2^m}\Rightarrow \sPr{m}{\tau_2}{\sigma_2}{\sigma_2^m}$.
\item $\sPr{m}{\tau_3}{\sigma_3}{\sigma_3^m} \Rightarrow$ there exist admissible
constants $\kappa_1=\kappa_1\q( \sigma_3\w)$, $\sigma_1=\sigma_1\q( \sigma_3\w)$
and an $m$-admissible constant $\sigma_1^m=\sigma_1^m\q( \sigma_3^m\w)$, such that $\sPo{m}{\kappa_1}{\tau_3}{\sigma_1}{\sigma_1^m}$.
\end{enumerate}
\end{thm}
\begin{proof}
$\sPn{1}{m}\Rightarrow \sPn{2}{m}$ follows just as in the proof
of (\ref{EqnXkAsSumOfSubs}).  This can be seen by noting that
$\q( X,d\w)$ can be replaced by $\q( \hX, \hd\w)$ in the 
proofs of Lemmas \ref{LemmaLieDerivOfDetCarnot}, \ref{LemmaDetsDontChangeCarnot},
\ref{LemmaMaxDetDoesntChange}, and \ref{LemmaQuotientOfDets}.  
The reader might worry that in the definition of $\sPn{2}{m}$ we
are using $X^\alpha$ instead of $X_{J_0}^\alpha$; however, there
is no real difference between the two, due to (\ref{EqnXkAsSumOfSubs}).
From there,
the proof follows easily, and we leave the details to the interested
reader.
$\sPn{2}{m}\Rightarrow \sPn{3}{m}$ and $\sPn{3}{m}\Rightarrow \sPn{1}{m}$
are both trivial.
\end{proof}

Let $\dm$ be a fixed lower bound for $d_{q+1}$.  We have:
\begin{prop}\label{PropUnderCondsBallContain}
Suppose $\sPt{1}{\tau_2}{\sigma_2}{\sigma_2^1}$ holds.
Then, there exists an admissible constant $\tau'=\tau'\q( \dm, \tau_2, \sigma_2 \w)$ such that:
$$\B{X}{d}{x_0}{\tau'}\subseteq \B{\hX}{\hd}{x_0}{\tau'}\subseteq \Bsub{X}{d}{J_0}{x_0}{\tau_2}.$$
\end{prop}
\begin{proof}
The first containment is trivial.  The second follows just as in
the proof of Proposition \ref{PropBigBallInSmallerBall}.
\end{proof}

\begin{prop}\label{PropPullBackUnderConds}
Suppose $\sPt{m}{\tau_2}{\sigma_2}{\sigma_2^m}$ holds.  Let
$\eta'\leq \eta_1$ be small enough
that $\Phi\q( B_{n_0}\q( \eta'\w)\w) \subseteq \Bsub{X}{d}{J_0}{x_0}{\tau_2}$.
Let $Y_{q+1}$ be the pullback of $X_{q+1}$ under $\Phi$ to $B_{n_0}\q( \eta'\w)$.
Then,
$$\Cjn{m}{B_{n_0}\q( \eta'\w)}{Y_{q+1}}\leq \sigma_4^m$$
where $\sigma_4^m=\sigma_4^m \q( \sigma_2^m\w)$ is an $m$-admissible constant.
\end{prop}
\begin{proof}
This follows just as in Proposition \ref{PropRegularityOfYUnitScale}.
\end{proof}

\begin{rmk}\label{RmkCommutatorInTermsOfControlUnitScale}
Our assumption on the commutator $\q[ X_i, X_j\w]$ in Section \ref{SectionUnitScale}
was essentially just that $\q( \q[X_i,X_j\w], d_i+d_j\w)$ satisfied
condition $\sPn{3}{m}$ for appropriate $m$.
\end{rmk}

%% file: unitopsnew.tex
In this section, we study the compositions of certain ``unit operators,''
which will be the core of our study of maximal functions
in Section \ref{SectionMaximalFuncs}--see Section \ref{SectionStreet} for some
motivation for the study of these operators.  
%In fact, we believe
%these operators are interesting in their own right, see 
%the discussion in Section \ref{SectionStreet} and Remark \ref{RmkUnitOpsInteresting}.

Let $X_1,\ldots, X_q$, $d_1,\ldots, d_q$, $x_0$, $\xi$, $\zeta$,
$c_{i,j}^k$, $n_0$, and $J_0$ be as in
Section \ref{SectionUnitScale}, in addition (for simplicity), we assume the stronger
assumptions of Remark \ref{RmkSymmetricUnitScaleAssump}.  
We again suppose, without loss
of generality, that $J_0=\q( 1,\ldots, n_0 \w)$.
In addition, suppose we are given $\numsub$ subsets of
$\q\{\q(X_1,d_1\w),\ldots, \q( X_q, d_q\w)\w\}$:
$$\q\{\q(Z_1^\mu, d_1^\mu\w),\ldots, \q(Z^\mu_{q_\mu}, d^\mu_{q_\mu}\w)\w\}\subseteq \q\{ \q(X_1,d_1\w),\ldots, \q( X_q, d_q\w) \w\}$$
with $1\leq \mu\leq \numsub$.  Suppose these subsets satisfy:
\begin{equation}\label{EqnAllTheXjAppear}
\q\{ \q(X_1,d_1\w),\ldots, \q(X_{n_0},d_{n_0}\w)\w\}\subseteq \bigcup_{1\leq \mu\leq \numsub} \q\{ \q(Z_1^{\mu},d_1^{\mu}\w),\ldots,\q(Z_{q_\mu}^\mu, d_{q_\mu}^\mu\w)\w\}.
\end{equation}

We say $C$ is a pre-admissible constant if $C$ can be chosen to depend
only on those parameters an admissible constant could depend on
in Remark \ref{RmkSymmetricUnitScaleAssump}, plus a fixed upper bound for $\numsub$.
%and a fixed upper bound for $\sup_{1\leq \mu\leq \numsub} q_{\mu}$.
We will write $A\prel B$ for $A\leq C B$ where $C$ is a pre-admissible
constant.  Also, we write $A\prea B$ for $A\prel B$ and $B\prel A$.

If we say that $C$ is an admissible constant, it means that we furthermore
assume that:
\begin{equation*}
\q[ Z_i^\mu, Z_j^\mu \w] = \sum_{k=1}^{q_\mu} c_{i,j}^{k,\mu} Z_k^\mu
\end{equation*}
and $C$ is allowed to depend on everything a pre-admissible constant
is allowed to depend on, plus a fixed upper bound for the quantities:
$$\sum_{\q|\alpha\w|\leq 2} \Cjn{0}{\B{X}{d}{X_0}{\xi}}{\q( Z^\mu\w)^\alpha c_{i,j}^{k,\mu}}, \quad 1\leq \mu\leq \numsub$$
which we assume to exist and are finite.

Given a function $f$ defined on a set $U$, and given for each $x\in U$
a set $V_x$, we define for those $y\in U$ such that $V_y\subseteq U$:
$$A_{U,V_{\cdot}} f\q( y\w) = \frac{1}{\Vol{V_y}} \int_{V_y} f\q( z\w) dz.$$
Here we are being ambiguous about what we mean by $\Vol{V_y}$ and $dz$.
Below, $V$ will be replaced by sets lying in the leaf generated
by one of the $Z^\mu$s (or by $X$).  We then mean for $\Vol{\cdot}$ and
$dz$ to refer to the Lebesgue measure on that leaf.  Below, we will
drop the $U$ from the subscript $A_{U,V_{\cdot}}$, and it is understood
to be the domain of $f$.

If we let $\xi_2$ be 
as in the statement of Corollary \ref{CorMainUnitScaleTwice}\footnote{Note that all of the constants in Corollary \ref{CorMainUnitScaleTwice} are pre-admissible in the sense of this section.} the main result of this section is:
\begin{thm}\label{ThmUnitOpsMainThm}
There exist admissible constants $0<\lambda_3, \lambda_2, \lambda_1\leq \xi_2$
such that for every $f\in\Cj{0}{\B{X}{d}{x_0}{\xi}}$ with $f\geq 0$, we have:
\begin{equation*}
\begin{split}
A_{\B{X}{d}{\cdot}{\lambda_3}} f\q(x_0\w) &\lesssim A_{\B{Z^\numsub}{d^\numsub}{\cdot}{\lambda_2}} A_{\B{Z^{\numsub-1}}{d^{\numsub-1}}{\cdot}{\lambda_2}}\cdots A_{\B{Z^1}{d^1}{\cdot}{\lambda_2}} f\q( x_0\w)\\
&\lesssim A_{\B{X}{d}{\cdot}{\lambda_1}} f\q(x_0\w).
\end{split}
\end{equation*}
\end{thm}

Define $n_1=\sum_{\mu=1}^{\numsub} q_{\mu}$.
To prove Theorem \ref{ThmUnitOpsMainThm}, we need a preliminary result:
\begin{prop}\label{PropUnitOpsPrelimProp}
There exist pre-admissible constants $0<l_3, l_2, l_1\leq \xi_2$ such that
for all $f\geq 0$, we have:
\begin{equation*}
\begin{split}
A_{\B{X}{d}{\cdot}{l_3}} f\q( x_0\w) &\prel \int_{Q_{n_1}\q( l_2\w)} f\q( e^{u_{1}\cdot Z^1}e^{u_{2}\cdot Z^{2}}\cdots e^{u_\numsub\cdot Z^\numsub} x_0\w) d u_1 \ldots d u_{\numsub} du_\numsub\\
&\prel A_{\B{X}{d}{\cdot}{l_1}}f\q(x_0\w).
\end{split}
\end{equation*}
Recall, $Q_{n_1}\q( l_2\w)$ denotes the $\q|\cdot\w|_\infty$ ball in
$\R^{n_1}$ of radius $l_2$.
\end{prop}

To prove Proposition \ref{PropUnitOpsPrelimProp},
we need some preliminary results.  Set $l_1=\xi_2$, take $\eta_2$ and $\Phi$ as in
Corollary \ref{CorMainUnitScaleTwice}.
\begin{lemma}\label{LemmaUnitOpsDefinel1}
For $f\geq 0$,
$$\int_{B_{n_0}\q( \eta_2\w)} f\circ \Phi\q( u\w) du \prel A_{\B{X}{d}{\cdot}{l_1}} f\q( x_0\w).$$
\end{lemma}
\begin{proof}
Note that:
$$\Vol{\Phi\q(B_{n_0}\q( \eta_2\w)\w)} \prea \q|\det_{n_0\times n_0} X\q( x_0\w)\w| \prea \Vol{\B{X}{d}{x_0}{\lambda_1}}$$
and so we have:
$$\frac{1}{\Vol{\Phi\q(B_{n_0}\q( \eta_2\w)\w)}} \int_{\Phi\q(B_{n_0}\q( \eta_2\w)\w)}f\q( y\w) dy \prel A_{\B{X}{d}{\cdot}{l_1}} f\q(x_0\w).$$
Now applying a change of variables as in (\ref{EqnChangeOfVars}) and
using that by Theorem \ref{ThmMainUnitScale} for every $u\in B_{n_0}\q( \eta_2\w)$,
$$\q|\det_{n_0\times n_0} d\Phi\q( u\w) \w| \prea \q|\det_{n_0\times n_0} X\q( x_0\w)\w|\prea \Vol{\B{X}{d}{x_0}{l_1}}$$
we see that:
\begin{equation*}
\int_{B_{n_0}\q( \eta_2\w)} f\circ \Phi\q( u\w) du \prel A_{\B{X}{d}{\cdot}{l_1}} f\q( x_0\w)
\end{equation*}
completing the proof.
\end{proof}

Let $Y^\mu_j$ be the pullback of $Z^\mu_j$ under the map $\Phi$.  As before,
we let $Y_1,\ldots, Y_{n_0}$ denote the pullbacks of $X_1,\ldots, X_{n_0}$.
Note that, by (\ref{EqnAllTheXjAppear}) each of $Y_1,\ldots, Y_{n_0}$ appears
as at least one of the $Y^\mu_j$.
\begin{lemma}\label{LemmaUnitOpsDefinel2}
There exists pre-admissible constants $l_2>0$, $\eta_3>0$ such that for all
$f\in \Cj{0}{B_{n_0}\q( \eta_2\w)}$, $f\geq 0$,
\begin{equation}\label{EqnToProveCompIntsUnitScale}
\begin{split}
\int_{B_{n_0}\q( \eta_3 \w)} f\q( u\w) du &\prel \int_{Q_{n_1}\q( l_2\w)} f\q( e^{u_1\cdot Y^1}e^{u_2\cdot Y^2}\cdots e^{u_{\numsub}\cdot Y^\numsub} 0\w) du_1\cdots du_\numsub \\&\prel \int_{B_{n_0}\q( \eta_2 \w)} f\q( u\w) du.
\end{split}
\end{equation} 
\end{lemma}
\begin{proof}
Let $\Psi\q( u_1,\ldots u_\numsub\w)$ denote the map:
$$\Psi\q( u_1,\ldots, u_\numsub\w) = e^{u_1 \cdot Y^1}e^{u_2\cdot Y^2}\cdots e^{u_\numsub \cdot Y^\numsub}0.$$
Note that $\Psi\in C^2\q( Q_{n_1}\q(\eta'\w) \w)$, provided $\eta'$ is
a sufficiently small pre-admissible constant.  Moreover, the $C^2$
norm is bounded by a pre-admissible constant (by Theorem \ref{ThmExpReg},
using that $\Cjn{2}{B_{n_0}\q(\eta_2\w)}{Y_j^\mu}\prel 1$, by
Theorem \ref{ThmMainUnitScale}).

Recalling that each $Y_j$ ($1\leq j \leq n_0$) appears at least once
in some $Y_k^\mu$, for each $1\leq j\leq n_0$ we pick one
such occurrence.  Write $\Psi$ as a function of two variables:
$$\Psi\q( u^1, u^2\w), \quad u^1\in Q_{n_0}\q(\eta'\w), u^2\in Q_{n_1-n_0}\q(\eta'\w)$$
where $u^1$ denotes the coefficients of the above chosen $Y_j$, and $u^2$
denotes the remaining coefficients.  For each fixed $u^2$, think of $\Psi$
as a function of one variable:
$$\Psi_{u^2}\q( u^1\w)$$
Note $d\Psi_{0}\q(0\w)=I$, and so by the $C^2$ estimates of $\Psi$, we
see that if $l_2$ if a pre-admissible constant that is small enough,
for every $u^2\in Q_{n_1-n_0}\q( l_2\w)$, we have:
\begin{equation}\label{EqnEstdForInv}
\q\| d\Psi_{u^2}\q(0\w) -I\w\|\leq \frac{1}{2}.
\end{equation} 
Hence, by the inverse function theorem (Theorem \ref{ThmInverseFunctionThm}),
we may pre-admissibly shrink $l_2$ such that:
\begin{itemize}
\item For every $u^2\in Q_{n_1-n_0}\q( l_2\w)$, $\Psi_{u^2}$ is injective
on $Q_{n_0}\q( l_2\w)$.
\item $\Psi_{u^2}\q(Q_{n_0}\q(l_2\w)\w)\subseteq Q_{n_0}\q( \eta_2\w)$, for every $u^2\in Q_{n_1-n_0}\q( l_2\w)$.
\end{itemize}
Hence, by a simple change of variables, we have for $u^2\in Q_{n_1-n_0}\q(l_2\w)$:
$$\int_{Q_{n_0}\q( l_2\w)} f\q( \Psi_{u^2}\q( u^1\w) \w) du^1 \prel \int_{B_{n_0}\q(\eta_2\w)} f\q(u\w) du$$
for $f\geq 0$.  Applying:
$$\int_{Q_{n_1-n_0}\q(l_2\w)} du^2$$
to both sides of this expression proves the latter inequality in (\ref{EqnToProveCompIntsUnitScale}).

We now turn to the former inequality in (\ref{EqnToProveCompIntsUnitScale}).
Applying the inverse function theorem (Theorem \ref{ThmInverseFunctionThm})
and again using (\ref{EqnEstdForInv}) we have that there exist
pre-admissible constants $\eta'\leq l_2$ and $\eta_3\leq l_2$ such that
for every $u^2\in Q_{n_1-n_0}\q( \eta'\w)$ we have that:
$$B_{n_0}\q( \eta_3\w) \subseteq \Psi_{u^2}\q( Q_{n_0}\q(l_2\w)\w).$$
Thus a simple change of variables (using (\ref{EqnEstdForInv})) shows that,
for $u^2\in  Q_{n_1-n_0}\q( \eta'\w)$ and $f\geq 0$:
$$\int_{B_{n_0}\q(\eta_3\w)} f\q( u\w)du \prel \int_{Q_{n_0}\q(l_2\w)} f\q( \Psi_{u^2}\q(u^1\w) \w) du^1.$$
Integrating both sides in $u^2$, we obtain:
\begin{equation*}
\begin{split}
\int_{B_{n_0}\q(\eta_3\w)} f\q( u\w) du &\prel \int_{Q_{n_1-n_0}\q( \eta'\w)} \int_{Q_{n_0}\q(l_2\w)} f\q( \Psi_{u^2}\q( u^1\w) \w) du^1 du^2\\
&\prel \int_{Q_{n_1-n_0}\q( l_2\w)} \int_{Q_{n_0}\q(l_2\w)} f\q( \Psi_{u^2}\q( u^1\w) \w) du^1 du^2.
\end{split}
\end{equation*}
Where in the last line, we used that $f\geq 0$ and that $\eta'\leq l_2$.
This proves the first inequality in (\ref{EqnToProveCompIntsUnitScale})
and completes the proof.
\end{proof}

\begin{lemma}\label{LemmaUnitOpsDefinel3}
There exists a pre-admissible constant $l_3>0$ such that for all $f\geq 0$:
\begin{equation*}
A_{\B{X}{d}{\cdot}{l_3}}f\q(x_0\w)\prel \int_{B_{n_0}\q( \eta_3\w)} f\circ \Phi\q( u\w) du.
\end{equation*}
\end{lemma}
\begin{proof}
Proceeding as in the proofs of Propositions \ref{PropCanPutBallInPhi} and \ref{PropBigBallInSmallerBall},
we may find a pre-admissible constant $l_3>0$ such that:
$$\B{X}{d}{x_0}{l_3}\subseteq \Phi\q( B_{n_0}\q(\eta_3\w)\w).$$
Note that, by Remark \ref{RmkVolOfAllSmallBalls}, we have:
$$\Vol{\B{X}{d}{x_0}{l_3}}\prea \q|\det_{n_0\times n_0} X\q( x_0\w)\w| \prea \Vol{\Phi\q( B_{n_0}\q(\eta_3\w)\w)}$$
and it follows that
\begin{equation*}
A_{\B{X}{d}{\cdot}{l_3}} f\q( x_0\w) \prel \frac{1}{\Vol{\Phi\q( B_{n_0}\q(\eta_3\w)\w)}} \int_{\Phi\q( B_{n_0}\q(\eta_3\w)\w)} f\q( y\w) dy.
\end{equation*}
Applying a change of variables as in (\ref{EqnChangeOfVars}) and using
that for all $u\in  B_{n_0}\q(\eta_3\w)$, we have:
\begin{equation*}
\q|\det_{n_0\times n_0} d\Phi\q(t\w)\w| \prea\q| \det_{n_0\times n_0} X\q(x_0\w)\w|\prea \Vol{\Phi\q( B_{n_0}\q(\eta_3\w)\w)}
\end{equation*}
it follows that:
\begin{equation*}
A_{\B{X}{d}{\cdot}{l_3}} f\q( x_0\w) \prel \int_{B_{n_0}\q(\eta_3\w)} f\q( u\w) du
\end{equation*}
completing the proof.
\end{proof}

\begin{proof}[Proof of Proposition \ref{PropUnitOpsPrelimProp}]
Fix $f\geq 0$ as in the statement of Proposition \ref{PropUnitOpsPrelimProp}.
Apply Lemmas \ref{LemmaUnitOpsDefinel1}, \ref{LemmaUnitOpsDefinel3}
to $f$ and Lemma \ref{LemmaUnitOpsDefinel2} to $f\circ \Phi$ to obtain:
\begin{equation*}
\begin{split}
A_{\B{X}{d}{\cdot}{l_3}} f\q( x_0\w) &\prel \int_{Q_{n_1}\q( l_2\w)} f\circ \Phi\q( e^{u_{1}\cdot Y^1}e^{u_{2}\cdot Y^{2}}\cdots e^{u_\numsub\cdot Y^\numsub} 0\w) d u_1 \ldots d u_{\numsub} du_\numsub\\
&\prel A_{\B{X}{d}{\cdot}{l_1}}f\q(x_0\w).
\end{split}
\end{equation*}
Using that:
\begin{equation*}
f\circ \Phi\q( e^{u_{1}\cdot Y^1}e^{u_{2}\cdot Y^{2}}\cdots e^{u_\numsub\cdot Y^\numsub} 0\w) = f\q( e^{u_{1}\cdot Z^1}e^{u_{2}\cdot Z^{2}}\cdots e^{u_\numsub\cdot Z^\numsub} x_0\w)
\end{equation*}
completes the proof.
\end{proof}

\begin{proof}[Proof of Theorem \ref{ThmUnitOpsMainThm}]
Let $0<\xi'\leq \xi_2$ be an admissible constant so small that:
\begin{equation*}
\begin{split}
\Omega_0&:=
\bigcup_{\substack{x_\numsub\in \\\B{Z^\numsub}{d^\numsub}{x_0}{\xi'}}} \bigcup_{\substack{x_{\numsub-1}\in \\\B{Z^{\numsub-1}}{d^{\numsub-1}}{x_\numsub}{\xi'}}}\cdots \bigcup_{\substack{x_{2}\in \\\B{Z^{2}}{d^2}{x_3}{\xi'}}} \B{Z^1}{d^1}{x_2}{\xi'}\\
&\Subset \B{X}{d}{x_0}{\frac{\xi}{2}}
\end{split}
\end{equation*}
where $A\Subset B$ denotes that $A$ is a relatively compact subset of $B$.
It is easy to see that this is possible, and we leave the details to the
reader.  Further, we take $0<\xi''\leq \xi'$ to be an admissible constant
so small that for every $y\in \Omega_0$,
$$\B{Z^\mu}{d^\mu}{y}{\xi''}\Subset \B{X}{d}{x_0}{\xi}, \quad 1\leq \mu\leq \nu.$$
We apply Proposition \ref{PropUnitOpsPrelimProp} to each $y\in \Omega_0$
with $\xi''$ in place of $\xi$ and $\q( Z^\mu, d^\mu\w)$ in place of
$\q( X,d\w)$ (and taking $\numsub=1$) to find admissible constants
$l_1, l_2, l_3$ such that for every $y\in \Omega_0$, and every $f\geq 0$
\begin{equation}\label{EqnFirstPropApp}
\begin{split}
A_{\B{Z^\mu}{d^\mu}{\cdot}{l_3}} f\q( y\w) &\lesssim \int_{Q_{q_\mu}\q(l_2\w)} f\q( e^{u_\mu \cdot Z^\mu} y\w) du_\mu
\\
&\lesssim A_{\B{Z^\mu}{d^\mu}{\cdot}{l_1}} f\q( y\w)
\end{split}
\end{equation}
and also applying Proposition \ref{PropUnitOpsPrelimProp} as it is
stated we may ensure that:
\begin{equation}\label{EqnSecondPropApp}
\begin{split}
A_{\B{X}{d}{\cdot}{l_3}} f\q( x_0\w) &\lesssim \int_{Q_{n_1}\q( l_2\w)} f\q( e^{u_{1}\cdot Z^1}e^{u_{2}\cdot Z^{2}}\cdots e^{u_\numsub\cdot Z^\numsub} x_0\w) d u_1 \ldots d u_{\numsub} du_\numsub\\
&\lesssim A_{\B{X}{d}{\cdot}{l_1}}f\q(x_0\w).
\end{split}
\end{equation}
Let $\lambda_1=l_1$ and $\lambda_2=l_3$.  Then, applying (\ref{EqnFirstPropApp})
$\numsub$ times, we see that:
\begin{equation*}
\begin{split}
&A_{\B{Z^\numsub}{d^\numsub}{\cdot}{\lambda_2}}A_{\B{Z^{\numsub-1}}{d^{\numsub-1}}{\cdot}{\lambda_2}} \cdots A_{\B{Z^1}{d^1}{\cdot}{\lambda_2}}f\q( x_0\w) 
\\&\quad
\lesssim \int_{Q_{n_1}\q( l_2\w)} f\q( e^{u_{1}\cdot Z^1}e^{u_{2}\cdot Z^{2}}\cdots e^{u_\numsub\cdot Z^\numsub} x_0\w) d u_1 \ldots  du_\numsub.
\end{split}
\end{equation*}
Applying (\ref{EqnSecondPropApp}) yields the second inequality in the
statement of Theorem \ref{ThmUnitOpsMainThm}.

We apply Proposition \ref{PropUnitOpsPrelimProp} to each $y\in \Omega_0$
with $\lambda_2$ in place of $\xi$ and $\q( Z^\mu, d^\mu\w)$ in place of
$\q( X,d\w)$ (and taking $\numsub=1$) to find admissible constants
$l_3',l_2',l_1'\leq \lambda_2$ such that for every $y\in \Omega_0$, and every $f\geq 0$:
\begin{equation}\label{EqnThirdPropApp}
\begin{split}
A_{\B{Z^\mu}{d^\mu}{\cdot}{l_3'}} f\q( y\w) &\lesssim \int_{Q_{q_\mu}\q(l_2'\w)} f\q( e^{u_\mu \cdot Z^\mu} y\w) du_\mu
\\
&\lesssim A_{\B{Z^\mu}{d^\mu}{\cdot}{l_1'}} f\q( y\w)
\end{split}
\end{equation}
and also applying Proposition \ref{PropUnitOpsPrelimProp} as it is
stated (with $\lambda_2$ in place of $\xi$) we may ensure that:
\begin{equation}\label{EqnFourthPropApp}
\begin{split}
A_{\B{X}{d}{\cdot}{l_3'}} f\q( x_0\w) &\lesssim \int_{Q_{n_1}\q( l_2'\w)} f\q( e^{u_{1}\cdot Z^1}e^{u_{2}\cdot Z^{2}}\cdots e^{u_\numsub\cdot Z^\numsub} x_0\w) d u_1 \ldots d u_{\numsub} du_\numsub\\
&\lesssim A_{\B{X}{d}{\cdot}{l_1'}}f\q(x_0\w).
\end{split}
\end{equation}

Set $\lambda_3=l_3'$.  We first claim that, for all $f\geq 0$:
\begin{equation}\label{EqnCompl3andnewl1}
A_{\B{Z^\mu}{d^\mu}{\cdot}{l_1'}} f\q( y\w) \lesssim A_{\B{Z^\mu}{d^\mu}{\cdot}{\lambda_2}} f\q( y\w), \quad y\in \Omega_0,\quad 1\leq\mu\leq\numsub.
\end{equation}
Indeed, we already have that $l_1'\leq \lambda_2$.  Moreover, we have
by Remark \ref{RmkVolOfAllSmallBalls}:
$$\Vol{\B{Z^\mu}{d^\mu}{y}{l_1'}}\approx \Vol{\B{Z^\mu}{d^\mu}{y}{\lambda_2}}, \quad y\in \Omega_0$$
and (\ref{EqnCompl3andnewl1}) immediately follows.

Thus we have:
\begin{equation*}
\begin{split}
&A_{\B{Z^\numsub}{d^\numsub}{\cdot}{\lambda_2}}A_{\B{Z^{\numsub-1}}{d^{\numsub-1}}{\cdot}{\lambda_2}} \cdots A_{\B{Z^1}{d^1}{\cdot}{\lambda_2}}f\q( x_0\w)\\
&\quad \gtrsim  
A_{\B{Z^\numsub}{d^\numsub}{\cdot}{l_1'}}A_{\B{Z^{\numsub-1}}{d^{\numsub-1}}{\cdot}{l_1'}} \cdots A_{\B{Z^1}{d^1}{\cdot}{l_1'}}f\q( x_0\w) 
\\&\quad \gtrsim
\int_{Q_{n_1}\q( l_2'\w)} f\q( e^{u_{1}\cdot Z^1}e^{u_{2}\cdot Z^{2}}\cdots e^{u_\numsub\cdot Z^\numsub} x_0\w) d u_1 \ldots d u_{\numsub} du_\numsub
\\&\quad \gtrsim
A_{\B{X}{d}{\cdot}{\lambda_3}} f\q( x_0\w)
\end{split}
\end{equation*}
where in the second to last line, we have applied (\ref{EqnThirdPropApp})
$\numsub$ times, and in the last line we have applied (\ref{EqnFourthPropApp}).
This completes the proof.
\end{proof}

%% file: multiintronew.tex
In this section, we discuss multi-parameter Carnot-Carath\'eodory balls.
In Section \ref{SectionBallsAtAPoint} we state the main theorem
regarding multi-parameter balls (Theorem \ref{ThmMultiBallsAtPoint}).
In Section \ref{SectionExamplesOfMainThm} we discuss four examples/applications
where Theorem \ref{ThmMultiBallsAtPoint} applies, one of which
is the ``weakly-comparable'' balls of \cite{TaoWrightLpImprovingBoundsForAverages}.
Finally, in Section \ref{SectionControlEveryScale} we discuss
a notion of ``controlling'' vector fields, which we hope will
elucidate the complicated assumptions in Section \ref{SectionBallsAtAPoint}.

Before we begin, we need one new piece of notation.
Suppose we are given formal degrees $d_1,\ldots, d_q\in \q[0,\infty\w)^\nu$.
If $\alpha$ is an ordered multi-index, we define the formal degree
$$d\q( \alpha\w)=\sum_{j=1}^q k_j d_j$$
where $k_j$ denotes the number of times that $j$ appears in the list $\alpha$.
Thus if $\delta\in \q[0,\infty\w)^\nu$, we may define
$\delta^{d\q( \alpha\w)}\in \q[0,\infty\w)$ and $\delta^{-d\q( \alpha\w)}\in \q[0,\infty\w]$ in the usual way.

%% file: ballsatpointthird.tex
Suppose $X_1,\ldots, X_q$ are $q$ $C^1$ vector fields with associated formal
degrees $0\ne d_1,\ldots, d_q\in \q[0,\infty\w)^\nu$.
Let $K\subset \Omega$ (think of $K=\q\{x_0\w\}$ or, more generally, $K$ compact).
Suppose that $\xi\in \q(0,1\w]^\nu$ is such that
$\q(  X, d\w)$ satisfies $\sC\q( x,\xi\w)$, for every $x\in K$.
%Suppose $\q( X,d\w)$, $K\subset \Omega$, and $\xi\in \q(0,1\w]^{\nu}$ are
%as in Section \ref{SectionMultiParamBalls}.  
The goal in this section
is to apply Theorem \ref{ThmMainUnitScale} and Corollaries
\ref{CorMainUnitScaleTwice} and \ref{CorBumpFuncAtUnitScale}
to the vector fields $\q( \delta X,\sd\w)$ at each point $x\in K$,
where $\delta\in \q[0,1\w)^\nu$ is small.

%In this section, we will assume the necessary and sufficient
%conditions on $\q( X,d\w)$ so that we may apply Corollaries
%\ref{CorMainUnitScaleTwice} and \ref{CorBumpFuncAtUnitScale}
%to $\q( \delta X,d\w)$ for all $\q|\delta\w|$ sufficiently small.
%These conditions are somewhat complicated, and we will offer
%a more detailed inspection of them (along with certain special cases) in Section ??.

Fix a subset $\sA$:
$$\sA\subseteq \q\{\delta\in \q[0,1\w]^\nu: \delta\ne 0, \delta\leq \xi \w\}$$
to be the set of ``allowable'' $\delta$s.  Recall, $\delta\leq \xi$ means
that the inequality holds coordinatewise.
\begin{rmk}\label{RmkFlagdelta}
We will be restricting our attention to balls $\B{X}{d}{x}{\delta}$,
where $\delta\in \sA$, $x\in K$.  For many applications, one would take:
\begin{equation}\label{EqnConnonicalA}
\sA= \q\{\delta\in \q[0,1\w]^\nu: \delta\ne 0, \delta\leq \xi\w\}
\end{equation}
and we encourage the reader to keep this particular choice of $\sA$
in mind throughout this section.  However, other choices of $\sA$
do arise in applications.  For instance, the choice:
\begin{equation}\label{EqnFlagA}
\sA= \q\{\delta\in \q[0,1\w]^\nu: \delta\ne 0, \delta\leq \xi, \delta_1\geq\delta_2\geq \cdots \geq \delta_{\nu}\w\}
\end{equation}
arises in the study of flag kernels, as in \cite{NagelRicciSteinSingularIntegralsWithFlagKernels}.
Also, the results in Section \ref{SectionWeaklyCompBalls} use
yet another
choice of $\sA$.
\end{rmk}

In this section, we assume that for every $\delta\in \sA$, $x\in K$, we have:
$$\q[\delta^{d_i} X_i, \delta^{d_j} X_j\w] = \sum_k c_{i,j}^{k,\delta,x} \delta^{d_k} X_k$$
on $\B{X}{d}{x}{\delta}$.
In addition, we assume:
\begin{itemize}
\item The $X_j$s are $C^2$ on $\B{X}{d}{x}{\xi}$, for every $x\in K$, and satisfy $\sup_{x\in K}\Cjn{2}{\B{X}{d}{x}{\xi}}{X_j}<\infty$.
\item For all $\q|\alpha\w|\leq 2$, $x\in K$, we have $\q(\delta^d X\w)^\alpha c_{i,j}^{k,\delta,x}\in \Cj{0}{\B{X}{d}{x}{\delta}}$, for every $i,j,k$, and every $\delta\in \sA$, and moreover:
$$\sup_{\substack{\delta\in A\\x\in K}} \sum_{\q|\alpha\w|\leq 2} \Cjn{0}{\B{X}{d}{x}{\delta}}{\q(\delta^d X\w)^\alpha c_{i,j}^{k,\delta,x}}<\infty.$$
\end{itemize}
Finally, let 
$$n_0\q( x,\delta\w)=\dim \Span{\delta^{d_1}X_1\q( x\w), \ldots, \delta^{d_q}X_q\q( x\w)}.$$

We say $C$ is an admissible constant if $C$ can be chosen to depend only
on fixed upper and lower bounds $d_{max}<\infty$, $d_{min}>0$, for
the coordinates of $\sd$, 
%a fixed lower bound $\xi_0>0$ for the
%coordinates of $\xi$, 
a fixed upper bound for $n,q,\nu$
and a fixed upper bound for the quantities:
\begin{equation*}
\sup_{x\in K}\Cjn{2}{\B{X}{d}{x}{\xi}}{X_j}, \quad \sup_{\substack{\delta\in A\\x\in K}} \sum_{\q|\alpha\w|\leq 2} \Cjn{0}{\B{X}{d}{x}{\delta}}{\q(\delta^d X\w)^\alpha c_{i,j}^{k,\delta,x}}.
\end{equation*}
Furthermore, if we say $C$ is an $m$-admissible constant, we mean that in addition
to the above, we assume that:
\begin{itemize}
\item $\sup_{x\in K}\Cjn{m}{\B{X}{d}{x}{\xi}}{X_j}<\infty$, for every $1\leq j\leq q$.
\item $\sup_{\substack{\delta\in A\\x\in K}} \sum_{\q|\alpha\w|\leq m} \Cjn{0}{\B{X}{d}{x}{\delta}}{\q(\delta^d X\w)^\alpha c_{i,j}^{k,\delta,x}}<\infty$, for every $i,j,k$.
\end{itemize}
(in particular, the above partial derivatives exist and are continuous).
$C$ is allowed to depend on $m$, all the quantities an admissible
constant is allowed to depend on, and a fixed upper bound for the
above two quantities.  
%As before, we write $A\lesssim_m B$ for $A\leq CB$
%where $C$ is an $m$-admissible constant.

\begin{rmk}
The assumptions in this section are somewhat complicated.  The
reader might hope that special cases of these assumptions
might be enough for applications.  %The extent to which this
%is true 
%This is discussed in Section \ref{SectionControlEveryScale}.
Unfortunately, this seems to not be the case, and is discussed in Section \ref{SectionControlEveryScale}.
\end{rmk}

For each $\delta\in \sA$, $x\in K$, let $J\q( x,\delta\w) =\q(J\q( x,\delta\w)_1,\ldots,J\q(x,\delta\w)_{n_0\q(x,\delta\w)}\w)\in \sI{n_0\q(x,\delta\w)}{q}$ be such that:
$$\q|\det_{n_0\q( x,\delta\w)\times n_0\q(x,\delta\w)} \q( \delta^d X \q( x\w)\w)_{J\q( x,\delta \w)}\w|_\infty = \q|\det_{n_0\q( x,\delta\w)\times n_0\q( x,\delta\w)} \delta^d X\q( x\w)\w|_\infty,$$
and define, for $u\in \R^{n_0\q(x,\delta\w)}$ with $\q|u\w|$ sufficiently small:
$$\Phi_{x,\delta}\q( u\w) = e^{u\cdot \q(\delta^d X\w)_{J\q( x,\delta\w)}}x.$$
The main result of this section is:
\begin{thm}\label{ThmMultiBallsAtPoint}
There exist admissible constants $\eta_1,\eta_2>0$, $0<\xi_4\leq \xi_3<\xi_2\leq \xi_1$ such that, for all $\delta\in \sA$, $x\in K$:
\begin{equation*}
\begin{split}
&\B{X}{d}{x}{\xi_4\delta}\subseteq 
\Bsub{X}{d}{J\q( x,\delta\w)}{x}{\xi_3 \delta}\subseteq 
\Phi_{x,\delta}\q( B_{n_0\q( x,\delta\w)}\q( \eta_2\w)\w) \\
&\subseteq \Btsub{X}{d}{J\q( x,\delta\w)}{x}{\xi_2\delta}
\subseteq \Bsub{X}{d}{J\q( x,\delta\w)}{x}{\xi_2\delta}
\subseteq \B{X}{d}{x}{\xi_2\delta} \\
&\subseteq \Bsub{X}{d}{J\q( x,\delta\w)}{x}{\xi_1\delta}
\subseteq \Phi_{x,\delta}\q( B_{n_0\q( x,\delta\w)}\q( \eta_1\w)\w) 
\subseteq \Btsub{X}{d}{J\q( x,\delta\w)}{x}{\delta}\\
&\subseteq \Bsub{X}{d}{J\q( x,\delta\w)}{x}{\delta}
\subseteq \B{X}{d}{x}{\delta},
\end{split}
\end{equation*}
and
\begin{itemize}
\item $\Phi_{x,\delta}:B_{n_0\q( x,\delta\w)}\q( \eta_1\w) \rightarrow \Btsub{X}{d}{J\q( x,\delta\w)}{x}{\delta}$ is one-to-one.
\item For all $u\in B_{n_0\q( x,\delta\w)}\q( \eta_1\w)$, $\q|\det_{n_0\q( x,\delta\w)\times n_0\q( x,\delta\w)} d\Phi_{x,\delta}\q( u\w)\w|\approx \q|\det_{n_0\q( x,\delta\w)\times n_0\q( x,\delta\w)} \delta^d X\q( x\w)\w|$.
\item $\Vol{\B{X}{d}{x}{\xi_2\delta}}\approx \q|\det_{n_0\q(x,\delta\w) \times n_0\q( x,\delta\w)} \delta^d X\q( x\w)\w|$.
\item There exists $\phi_{x,\delta} \in C_0^2\q( \B{X}{d}{x}{\delta}\w)$, which
equals $1$ on $\B{X}{d}{x}{\xi_4\delta}$ and satisfies:
$$\q|X^{\alpha} \phi_{x,\delta}\w|\lesssim_{\q(\q|\alpha\w|-1\w)\vee 0}\delta^{-d\q( \alpha\w)}.$$
\end{itemize}
Furthermore, if we let $Y_j^{x,\delta}$ be the pullback of $\delta^{d_j} X_j$ under the 
map $\Phi_{x,\delta}$ to $B_{n_0\q( x,\delta\w)}\q(\eta_1\w)$, we have that:
$$\Cjn{m}{B_{n_0\q( x\w)}\q(\eta_1\w)}{Y_j^{x,\delta}}\lesssim_m 1.$$
Finally, if for each $x\in K$, $u\in B_{n_0\q( x,\delta\w)}\q( \eta_1\w)$, and
$\delta\in \sA$, we define the $n_0\q( x,\delta\w) \times n_0\q( x,\delta\w)$ matrix
$A\q( x,u\w)$ by:
$$\q( Y_{J\q( x,\delta\w)_1}^{x,\delta},\ldots, Y_{J\q( x,\delta\w)_{n_0\q( x,\delta\w)}}^{x,\delta} \w) = \q( I+A\q( x,\cdot\w) \w) \grad_u$$
then,
\begin{equation}\label{EqmBallsAtAPointMatrixBound}
\sup_{u\in B_{n_0\q( x,\delta\w)}\q(\eta_1\w)} \q\|A\q( x,u\w)\w\|\leq \frac{1}{2}.
\end{equation}
%Here,
%$$J\q( x,\delta\w) = \q(J\q( x,\delta\w)_1, \ldots, J\q( x,\delta\w)_{n_0\q( x\w)}\w)$$
\end{thm}
\begin{proof}
For each $x\in K$ and $\delta\in \sA$, merely apply Theorem \ref{ThmMainUnitScale}
and Corollaries \ref{CorMainUnitScaleTwice} and \ref{CorBumpFuncAtUnitScale}
to $\q( \delta^d X, \sd\w)$, taking $\zeta=1$ and $J_0=J\q( x,\delta\w)$.
It is easy to see, by the assumptions in this section, that all of the constants admissible (respectively, $m$-admissible) in those results
are admissible (respectively, $m$-admissible) in the sense of this section.
\end{proof}

\begin{cor}
We assume, in addition to the other assumptions in this section, that for every $\delta\in \sA$ with $\q|\delta\w|$ sufficiently
small, $\xi_2^{-1}\delta\in \sA$
(in particular, this is true if
$\sA$ is given by (\ref{EqnConnonicalA}) or (\ref{EqnFlagA})).
We have, for $x\in K$, and all $\delta\in \sA$ with $\q|\delta\w|$ sufficiently small: 
\begin{equation}\label{EqnPrelimForDoubling}
\Vol{\B{X}{d}{x}{\delta}}\approx \q|\det_{n_0\q( x\w)\times n_0\q( x\w)} \q(\xi_2^{-1}\delta\w)^d X\q( x\w)\w|\approx \q|\det_{n_0\q( x\w) \times n_0\q( x\w)} \delta^d X\q( x\w)\w|
\end{equation}
and so if $\q|\delta\w|$ is sufficiently small and $2\delta\in \sA$,
\begin{equation}\label{EqnDoublingConditionInMainThm}
\Vol{\B{X}{d}{x}{2\delta}}\lesssim \Vol{\B{X}{d}{x}{\delta}}.
\end{equation}
\end{cor}
\begin{proof}
\eqref{EqnPrelimForDoubling} follows by replacing $\delta$ with $\xi_2^{-1}\delta$ in the
statement of Theorem \ref{ThmMultiBallsAtPoint}.
\eqref{EqnDoublingConditionInMainThm} follows since the RHS
of \eqref{EqnPrelimForDoubling} is the square root of a polynomial in $\delta$ (with positive coefficients).
\end{proof}

%% file: examples.tex
In this section, we present four applications/examples where
Theorem \ref{ThmMultiBallsAtPoint} applies.
The first two applications were both previously well understood, and in fact
both can be understood by the methods of \cite{NagelSteinWaingerBallsAndMetricsDefinedByVectorFields}.
The reason we include them here is to put them in
the context of Theorem \ref{ThmMultiBallsAtPoint}, and because they have been useful in the past.
The third example is included to provide a simple situation where
the methods of \cite{NagelSteinWaingerBallsAndMetricsDefinedByVectorFields}
do not apply but Theorem \ref{ThmMultiBallsAtPoint} does.
We close this section with most interesting of our applications.
In this application, we show how to lift results from the single
parameter case to the multi-parameter case.  In particular, we will
see how results like the Campbell-Hausdorff formula can
be applied even in the multi-parameter case--where, at first glance,
they seem totally inapplicable.
%Of course, there are many examples where Theorem \ref{ThmMultiBallsAtPoint}
%applies but the methods of \cite{NagelSteinWaingerBallsAndMetricsDefinedByVectorFields} do not.
%The interested reader should have no trouble developing such
%an example.

%% file: weaklycompnew.tex
In this section, we discuss the so-called ``weakly-comparable'' balls
that were used in \cite{TaoWrightLpImprovingBoundsForAverages}.
We do not attempt to proceed in the greatest possible generality,
and instead just try to present the main ideas.
Most of the conclusions of this section are contained in \cite{TaoWrightLpImprovingBoundsForAverages},
and the main purpose here is just to show how these results are a 
special case of Theorem \ref{ThmMultiBallsAtPoint}.

Let $X_1,\ldots, X_\nu$ be $\nu$ $C^\infty$ vector fields defined on $\Omega$,
with associated formal degrees $d_1,\ldots, d_\nu\in \q( 0,\infty\w)$.
Fix large constants $\kappa, N$.  Essentially, we will be considering
the balls generated by the vector fields $\delta_\mu^{d_\mu} X_\mu$,
where we restrict our attention to those $\delta=\q( \delta_1,\ldots, \delta_\nu\w)$
such that:
\begin{equation}\label{EqnDefWeaklyComp}
%\frac{1}{\kappa}\delta_{\mu_1}^N\leq 
\delta_{\mu_2}^N \leq \kappa \delta_{\mu_1}
\end{equation}
for every $\mu_1,\mu_2$.  We call a $\delta$ satisfying (\ref{EqnDefWeaklyComp}) a ``weakly comparable'' $\delta$.

We assume that $X_1,\ldots, X_\nu$ satisfy H\"ormander's condition.  That is,
$X_1,\ldots, X_\nu$, along with their commutators of all orders, span
that tangent space at every point of $\Omega$.  Fix $K\Subset \Omega$,
a compact subset of $\Omega$, and let $\Omega_0\Subset \Omega$ be such that
$K\Subset \Omega_0$.

Let $\hd_\mu\in \q[0,\infty\w)^\nu$ be the vector that is $d_\mu$ in the $\mu$th
component, and $0$ in the other components.  For a list (or a ``word'') $w=\q( w_1, \ldots, w_r\w)$
of integers $1,\ldots, \nu$ we define:
$$\hd\q( w\w) = \sum_{j=1}^r \hd_{w_j},$$
$$X_w = \ad{X_{w_1}}\ad{X_{w_2}}\cdots \ad{X_{w_{r-1}}} X_{w_r}.$$
As before, for a $\nu$ vector $e=\q( e_1, \ldots, e_\nu\w)\in \q[0,\infty\w)^\nu
$,
define $\delta^e = \prod_{\mu=1}^\nu \delta_\mu^{e_\mu}$.

By the assumption that $X_1,\ldots, X_\nu$ satisfy H\"ormander's condition,
and by the relative compactness of $\Omega_0$, there exist $l$ lists $w^1,\ldots, w^l$
such that, for every $x\in \Omega_0$:
$$T_x\Omega = \Span{X_{w^1}\q( x\w), \ldots, X_{w^l}\q(x\w)}.$$

Let $d_0=\sup_{1\leq m\leq l} \q| \hd\q( w^m\w)\w|_1$.  Recall, $\q|v\w|_1=\sum_j \q|v_j\w|$.  Let $\q( X,d\w)$ denote the 
finite list of vector fields along with associated
formal degrees given by $\q( X_w, \hd\q(w\w)\w)$, where
$w$ ranges over all lists satisfying $\q| \hd\q( w\w)\w|_1\leq Nd_0$.

Take $\xi\in \q(0,1\w]^\nu$ so small that $\q( X,d\w)$ satisfies 
$\sC\q( x, \xi\w)$ for every $x\in K$, with $\Omega_0$ taking
the place of $\Omega$ in the definition of $\sC\q( x,\xi\w)$.

In this section, we say that $C$ is an admissible constant if $C$ can
be chosen to depend only on a fixed upper bound for $n$, a fixed upper bound for $\nu$, a fixed
upper bound for $N$ and $\kappa$, fixed upper and lower bounds for
$d_\mu$ ($1\leq \mu\leq \nu$), a fixed upper bound for $d_0$, 
%a fixed
%lower bound for the coordinates of $\xi$, 
a fixed lower bound for:
$$\inf_{x\in \Omega_0} \q| \det_{n\times n} \q( X_{w^1}\q( x\w) | \cdots | X_{w^l}\q(x \w) \w)\w|,$$
and fixed upper bounds for a finite number of the norms:
$$\Cjn{m}{\Omega_0}{X_\mu}, \quad 1\leq \mu\leq \nu.$$

\begin{thm}
Let $\sA=\q\{\delta\in \q[0,1\w]^\nu: \delta\ne 0, \delta\leq \xi, \delta_{\mu_2}^N\leq \kappa\delta_{\mu_1}, \forall \mu_1,\mu_2\w\}$.
Then, with this choice of $A$, the list of vector fields $\q( X,d\w)$ satisfies
the assumptions of Section \ref{SectionBallsAtAPoint}, where all of the
constants that are admissible (or even $m$-admissible) in the sense of that section are admissible
in the sense of this section.  Hence, Theorem \ref{ThmMultiBallsAtPoint}
holds for $\q( X,d\w)$.
\end{thm}
\begin{proof}
We will show that if $w_1$ and $w_2$ are words with $\q|\hd\q( w_1\w)\w|_1, \q|\hd\q( w_2\w)\w|_1\leq Nd_0$,
we have for $\delta\in \sA$:
$$\q[\delta^{\hd\q( w_1\w)}X_{w_1}, \delta^{\hd\q( w_2\w)}X_{w_2}\w] = \sum_{\q|\hd\q( w_3\w)\w|_1\leq Nd_0} c_{w_1,w_2}^{w_3,\delta} \delta^{\hd\q( w_3\w)}X_{w_3},$$
with
$$\Cjn{m}{\Omega_0}{c_{w_1,w_2}^{w_3,\delta}}\lesssim 1.$$

If $\q|\hd\q( w_1\w)+\hd\q( w_2\w)\w|_1\leq Nd_0$, this follows
easily from the Jacobi identity.
We proceed, then, in the case when $\q| \hd\q( w_1\w)+\hd\q( w_2\w)\w|_1> Nd_0$.
Using that:
\begin{equation}\label{EqnHorComm}
\q[ X_{w_1}, X_{w_2}\w]=\sum_{k=1}^l c_{w_1,w_2}^k X_{w^k},
\end{equation}
with
$$\Cjn{m}{\Omega_0}{c_{w_1,w_2}^k}\lesssim 1,$$
and multiplying both sides of (\ref{EqnHorComm}) by:
$$\delta^{\hd\q( w_1\w)}\delta^{\hd\q( w_2\w)}$$
the result follows easily.
\end{proof}

\begin{example}\label{ExampleWeaklyCompNecessary}
An example to keep in mind where the weakly comparable hypothesis
is necessary is given by the following vector fields with formal degrees on $\R^2$:
\begin{equation*}
\q( \partial_x, \q(1,0,0\w)\w),\quad \q( e^{-\frac{1}{x^2}}\partial_y, \q( 0,1,0\w)\w),\quad \q( \partial_y, \q( 0,0,1\w)\w).
\end{equation*}
If we restrict our attention to the case when $\delta_3=0, \delta_1=\delta_2$
(which is impossible under the weakly comparable hypothesis, without
taking $\delta_1=0=\delta_2$)
then (without being precise about definitions), we are left
with the one-parameter ball of radius $\delta_1$ ``generated''
by the vector fields:
$$\partial_x,\quad e^{-\frac{1}{x^2}}\partial_y$$
and it is well known that this sort of ball cannot satisfy
any sort of doubling condition of the form (\ref{EqnDoublingConditionInMainThm}).
\end{example}

%% file: multiplelists.tex
In this section, we suppose we have $\nu$ lists of $C^\infty$ vector fields
on $\Omega\subseteq \R^n$
with associated formal degrees:
$$\q(X_1^\mu, d_1^\mu\w),\ldots, \q( X_{q_\mu}^\mu, d_{q_{\mu}}^\mu \w), d_j^\mu\in \q(0,\infty\w), 1\leq \mu\leq \nu$$
and we assume that for {\it each} $\mu$, the list
$$X_1^\mu, \ldots, X_{q_\mu}^\mu$$
spans the tangent space at each point in $\Omega$.
Our goal is to consider the balls generated by the vector fields:
$$\delta_\mu^{d_j^\mu}X_j^\mu,\quad 1\leq \mu\leq \nu,\quad 1\leq j\leq q_{\mu},$$
where $\delta=\q( \delta_1,\ldots, \delta_\nu\w)$ is small.

Fix $K\Subset \Omega$, a compact subset of $\Omega$, and take
$\Omega_0\Subset \Omega$ such that $K\Subset \Omega_0$.  Define:
$$d_0:=\max_{\substack{1\leq \mu\leq \nu\\1\leq j\leq q_{\mu}}} d_j^{\mu}.$$

We define $\hd_j^\mu\in \q[0,\infty\w)^\nu$ for $1\leq \mu\leq \nu$ and $1\leq j\leq q_{\mu}$
to be the vector that is $d_j^\mu$ in the $\mu$th component
and $0$ in all the other components.
For a list $w=\q( \q( w_1,\mu_1\w),\ldots, \q( w_r, \mu_r\w)\w)$ of pairs,
where $1\leq \mu_j\leq \nu$ and $1\leq w_j\leq q_{\mu_j}$ we define
(as in Section \ref{SectionWeaklyCompBalls}):
$$\hd\q( w\w) = \sum_{j=1}^r \hd_{w_j}^{\mu_j},$$
$$X_w = \ad{X_{w_1}^{\mu_1}}\ad{X_{w_2}^{\mu_2}}\cdots \ad{X_{w_{r-1}}^{\mu_{r-1}}} X_{w_r}^{\mu_r}.$$
Let $\q( X,d\w)$ denote the list of vector fields with associated
formal degrees given by $\q( X_w, \hd\q( w\w)\w)$ where $w$
ranges over all those lists with $\q|\hd \q( w \w)\w|_\infty\leq d_0$.

Take $\xi\in \q( 0,1\w]^\nu$ so small that $\q(  X, d\w)$ satisfies
$\sC\q( x,\xi\w)$ for every $x\in K$, with $\Omega_0$ taking the place
of $\Omega$ in the definition of $\sC\q( x,\xi\w)$.

In this section, we say that $C$ is an admissible constant if $C$ can
be chosen to depend only on a fixed upper bound for $n$, fixed
upper and lower bounds for $d_j^\mu$ ($1\leq \mu\leq \nu$, $1\leq j\leq q_\mu$), 
%a fixed lower bound for the coordinates of $\xi$, 
a fixed upper bound for $\nu$,
a fixed lower bound for:
$$\inf_{\substack{x\in \Omega_0\\1\leq \mu\leq \nu}} \q|\det_{n\times n} \q(X_1^\mu\q(x\w) | \cdots| X_{q_\mu}^\mu\q( x\w)\w)\w|,$$
and fixed upper bounds for a finite number of the norms:
$$\Cjn{m}{\Omega_0}{X_{j}^\mu}, \quad 1\leq \mu\leq \nu, \quad 1\leq j \leq q_{\mu}.$$

\begin{thm}\label{ThmMultipleListsThatSpan}
Let $\sA$ be given by (\ref{EqnConnonicalA}).  Then, with this choice of $\sA$,
the list of vector fields $\q( X,d\w)$ satisfies the assumptions
of Section \ref{SectionBallsAtAPoint}, where all of the constants
that are admissible (or even $m$-admissible) in the sense of that section are admissible in the
sense of this section.  Hence, Theorem \ref{ThmMultiBallsAtPoint}
holds for $\q(X,d\w)$.
\end{thm}
\begin{proof}
We will show that if $w_1$ and $w_2$ are lists with $\q|\hd\q( w_1\w)\w|_\infty, \q|\hd\q(w_2\w)\w|_\infty \leq d_0$, we have for $\delta\in \sA$:
\begin{equation}\label{EqnToShowMultipleLists}
\q[\delta^{\hd\q( w_1\w)} X_{w_1}, \delta^{\hd\q( w_2\w)} X_{w_2} \w] = \sum_{\q|\hd\q( w_3\w)\w|_\infty\leq d_0} c_{w_1,w_2}^{w_3,\delta} \delta^{\hd\q( w_3\w)} X_{w_3},
\end{equation}
with
$$\Cjn{m}{\Omega_0}{c_{w_1,w_w}^{w_3,\delta}}\lesssim 1.$$

If $\q|\hd\q( w_1\w) +\hd\q( w_2\w)\w|_\infty\leq d_0$, (\ref{EqnToShowMultipleLists})
follows easily from the Jacobi identity.  We proceed, therefore, in the
case when $\q|\hd\q( w_1\w) +\hd\q( w_2\w)\w|_\infty >d_0$.
Let us assume that the $\mu$th coordinate of $\hd\q( w_1\w)+\hd\q( w_2\w)$ is greater than $d_0$.  Using that:
\begin{equation}\label{EqnIntegMultipleLists}
\q[X_{w_1},X_{w_2}\w]=\sum_{j=1}^{q_\mu} c_{w_1,w_2}^j X_{j}^\mu,
\end{equation}
and multiplying both sides of (\ref{EqnIntegMultipleLists}) by:
$$\delta^{\hd\q( w_1\w)}\delta^{\hd\q( w_2\w)}$$
(\ref{EqnToShowMultipleLists}) follows easily.
\end{proof}

\begin{rmk}
Theorem \ref{ThmMultipleListsThatSpan} also follows from the results in Section 4
of \cite{StreetAdvances} (which used the methods of \cite{NagelSteinWaingerBallsAndMetricsDefinedByVectorFields}).  In fact, the more general results
in Section 4 of \cite{StreetAdvances} 
are clearly a special case Theorem \ref{ThmMultiBallsAtPoint}.
\end{rmk}

%% file: badnsw.tex
As was already discussed in Section \ref{SectionNSW}, the methods of
\cite{NagelSteinWaingerBallsAndMetricsDefinedByVectorFields}
fail to prove Theorem \ref{ThmMultiBallsAtPoint}.  The main
issue is that the error term given by the Campbell-Hausdorff
formula cannot be {\it a priori} controlled using the methods
of \cite{NagelSteinWaingerBallsAndMetricsDefinedByVectorFields}
(see Section \ref{SectionNSW}).
Thus, if one wishes to develop an example where the methods
of \cite{NagelSteinWaingerBallsAndMetricsDefinedByVectorFields}
do not apply, one must use vector fields where the error term
is not obviously controllable.  As shown in
Section \ref{SectionLiftingSingleParam} (see also Section \ref{SectionNSW}), the results of this paper imply
that the error term is controllable.  The point of this section
is to offer an example where the methods of \cite{NagelSteinWaingerBallsAndMetricsDefinedByVectorFields} do not prove this fact.

In particular, one needs that the error term of the Campbell-Hausdorff
formula not be zero, so the main aspect of the example that follows
is that the iterated brackets of the vector fields we present
are not eventually zero (this rules out vector fields with polynomial
coefficients\footnote{As a consequence, if one is only interested
in vector fields with polynomial coefficients, then the methods
of \cite{NagelSteinWaingerBallsAndMetricsDefinedByVectorFields}
(with some adjustments) are sufficient for most purposes.}).

We work in the two-parameter situation, with $\sA$
given by \eqref{EqnConnonicalA}.  We consider the list of vector fields
on $\R^4$ with formal degrees
``generated'' by the vector fields
$$\q(\partial_x+\cos\q( s\w) \partial_y ,\q( 1,0\w)\w), \quad \q( \partial_s+\cos\q( x\w) \partial_t, \q( 0,1\w) \w).$$
More specificly, we consider the list of vector fields with formal degrees:
\begin{equation*}
\q(\partial_x+\cos\q( s\w) \partial_y ,\q( 1,0\w)\w), \: \q( \partial_s+\cos\q( x\w) \partial_t, \q( 0,1\w) \w)\: \q( \sin\q( s\w) \partial_y-\sin\q( x\w) \partial_t,\q( 1,1\w)\w),
\end{equation*}
\begin{equation*}
\q( \cos\q( x\w)\partial_t,\q( 2,1\w)\w), \: \q( \cos\q( s\w) \partial_y,\q( 1,2\w)\w),\: \q(\sin\q( x\w)\partial_t, \q( 3,1\w) \w),\: \q(\sin\q( s\w) \partial_y,\q( 1,3\w)\w).
\end{equation*}
It is immediate to verify that these vector fields satisfy
the assumptions of Theorem \ref{ThmMultiBallsAtPoint}, but (for the reasons
mentioned above) the methods of \cite{NagelSteinWaingerBallsAndMetricsDefinedByVectorFields}
are insufficient to study the balls generated by these vector fields.

%% file: lifting.tex
In this section, we discuss a general method whereby one may lift
many results from the single parameter setting
of \cite{NagelSteinWaingerBallsAndMetricsDefinedByVectorFields}
to the multi-parameter setting in this paper.

To make this methodology clear, we present a concrete example
where it applies.  Indeed, this example is interesting
in its own right.

We suppose that we are given generating $C^\infty$ vector fields on $\Omega\subseteq \R^n$,
 with
$\nu$ parameter formal degrees,
$$\q( W_1,d_1\w),\ldots, \q( W_r,d_r\w).$$
For a word $w=\q( w_1,\ldots, w_l\w)$, $w_j\in \q\{1,\ldots, r\w\}$,
we define:
$$\hd\q( w\w) = \sum_{j=1}^l d_{w_j},$$
$$X_w = \ad{X_{w_1}}\cdots \ad{X_{w_{l-1}}} X_{w_l}.$$
Let $\q(X,d\w)=\q(X_1,d_1\w),\ldots, \q(X_q,d_q\w)$ denote the list of vector fields with formal
degrees given by $\q( X_w, \hd\q( w\w)\w)$ where $w=\q( w_1,\ldots, w_l\w)$
and $l\leq M$ for some fixed large $M$.  Our goal is to show, under
the smooth version of the hypotheses of Section \ref{SectionBallsAtAPoint}, that the balls
$$B_{\delta^d W}\q( x\w)$$
are comparable to the balls
$$\B{X}{d}{x}{\delta}.$$

More specificly, fix $x_0\in \Omega$, and assume $\q(X,d\w)$ satisfies
$\sC\q( x_0,\xi\w)$.
We assume that we have, for every $\delta\in\q[0,1\w)^\nu$ with $\delta\leq \xi$,
$$\q[\delta^{d_i} X_i, \delta^{d_j}X_j\w]= \sum_k c_{i,j}^{k,\delta} \delta^{d_k} X_k,$$
on $\B{X}{d}{x_0}{\delta}$.  In what follows, an admissible constant
may depend on upper bounds for $q$ and $n$, lower and upper bounds for the 
$\q|\cdot\w|_1$ norms of the
formal degrees,
upper bounds for a finite number of the norms $\Cjn{m}{\B{X}{d}{x_0}{\xi}}{X_j}$
and upper bounds for a finite number of the norms:
$$\sup_{\delta\leq \xi} \sum_{\q|\alpha\w|\leq m}\Cjn{0}{\B{X}{d}{x_0}{\delta}}{\q( \delta^d X\w)^{\alpha}c_{i,j}^{k,\delta}},$$
which we assume to be finite--in fact, we only need the above bounds
for $m$ which can be chosen to depend only on $M$ and $q$.
%, though we will
%not be explicit about this fact below. 

We have,
\begin{thm}\label{ThmChowBall}
There exists an admissible constant $\eta'>0$ such that for every $\delta\leq \xi$, we have:
\begin{equation*}
\B{X}{d}{x_0}{\eta'\delta}\subseteq B_{\delta^d W}\q( x_0\w) \subseteq \B{X}{d}{x_0}{\delta}.
\end{equation*}
\end{thm}
The second containment in Theorem \ref{ThmChowBall} is obvious, and so
the theorem is really a statement about the first containment.
In the single parameter case, Theorem \ref{ThmChowBall} was shown in
\cite{NagelSteinWaingerBallsAndMetricsDefinedByVectorFields}.
Specificly, we have:
\begin{thm}[Theorem 4 of \cite{NagelSteinWaingerBallsAndMetricsDefinedByVectorFields}]\label{ThmSingleChowBall}
In the case $\nu=1$ and when $X_1,\ldots, X_q$ span the tangent space,
Theorem \ref{ThmChowBall} holds--so long as we allow admissible
constants to also depend on a lower bound for:
$$\q|\det_{n\times n} X\q( x_0\w)\w|.$$
\end{thm}
Actually, in Theorem 4 of \cite{NagelSteinWaingerBallsAndMetricsDefinedByVectorFields},
$W_1,\ldots, W_r$ are each given the formal degree $1$, but this is not an essential point, and the methods there immediately generalize to give
Theorem \ref{ThmSingleChowBall}.
It is worth noting that the proof in \cite{NagelSteinWaingerBallsAndMetricsDefinedByVectorFields}
uses heavily the Campbell-Hausdorff formula, and therefore use of
a lower bound for $\q|\det_{n\times n} X\q( x_0\w)\w|$ is essential
for those methods.

\begin{proof}[Proof of Theorem \ref{ThmChowBall}]
Apply Theorem \ref{ThmMultiBallsAtPoint}, to obtain $\Phi_\delta$,
$\eta_1$ and $\xi_2$ as in that theorem.  To prove Theorem \ref{ThmChowBall},
it suffices to construct an admissible constant $\eta'>0$ such that:
$$\B{X}{d}{x_0}{\eta'\delta}\subseteq B_{\q( \xi_2\delta\w)^d W}\q( x_0\w);$$
rephrasing this, it suffices to show,
\begin{equation}\label{EqnChowSTS}
\B{\delta^d X}{\sd}{x_0}{\eta'}\subseteq B_{\q( \xi_2\delta\w)^d W}\q( x_0\w),
\end{equation}
for some admissible $\eta'>0$.
Let $Y$ denote the list of vector fields given by the pullback
of $\delta^d X$ under the map $\Phi_\delta$ to $B_{n_0\q( \delta\w)}\q( \eta_1\w)$, and let $W'$ denote the list
of vector fields given by the pullback of $\delta^d W$ under $\Phi_\delta$.
Pulling back \eqref{EqnChowSTS} via $\Phi_\delta$, we see that it suffices to
show that,
\begin{equation}\label{EqnChowPullBackSTS}
\B{Y}{\sd}{0}{\eta'}\subseteq B_{\xi_2^{\sd}W'}\q(0\w).
\end{equation}
However, using that the $W'$ generate the $Y$ (since this is just the pullback
of the statement that the $W$ generate the $X$), using
that $\q|\det_{n_0\q( \delta\w)\times n_0\q( \delta\w)} Y\q( 0\w)\w|\gtrsim 1$
(this follows from \eqref{EqmBallsAtAPointMatrixBound}), and using
$\xi_2\approx 1$, we may apply
Theorem \ref{ThmSingleChowBall} (in the special case when $\delta\approx 1$)
to deduce \eqref{EqnChowPullBackSTS}, completing the proof.
\end{proof}

In conclusion, if one can prove a result in the single-parameter setting
of \cite{NagelSteinWaingerBallsAndMetricsDefinedByVectorFields},
one often gets a multi-parameter result ``for free,'' merely by
pulling the multi-parameter vector fields back under the
scaling map $\Phi_\delta$ and applying the single-parameter
result.  In particular, this allows one to use the Campbell-Hausdorff
formula to prove results in the multi-parameter setting.  This same
proof method shows that the error term for the Campbell-Hausdorff formula
as discussed in Section \ref{SectionNSW} can be controlled in an
appropriate sense, even in the multi-parameter setting.

%% file: controleveryscale.tex
In Section \ref{SectionControlUnitScale}, we saw that
the conditions imposed on the commutators $\q[ X_i, X_j\w]$
in Section \ref{SectionUnitScale} were closely
related to three equivalent conditions that were defined
in Section \ref{SectionControlUnitScale} (see Remark \ref{RmkCommutatorInTermsOfControlUnitScale}).
The goal in this section is to understand the conditions
imposed on the commutators in Section \ref{SectionBallsAtAPoint}
in a similar way.  To do so, we will lift two of the three
equivalent conditions from Section \ref{SectionControlUnitScale}
into the setting of Section \ref{SectionBallsAtAPoint}.
These equivalent conditions are interesting in their own right,
and will play a role in future work.

We take all the same notation as in Section \ref{SectionBallsAtAPoint},
and define ($m$-)admissible constants in the same way.  Let
$X_{q+1}$ be a $C^1$ vector field on $\Omega$, with an
associated formal degree $0\ne d_{q+1}\in \q[0,\infty\w)^\nu$.
We will introduce conditions on $\q( X_{q+1},d_{q+1}\w)$
which will imply (informally) that one does not
``get anything new'' if $\q( X_{q+1},d_{q+1}\w)$
is added to the list $\q( X,d\w)$.
Let $\q( \hX, \hd\w)$ denote the list of vector fields
with formal degrees
$\q( X_1,d_1\w),\ldots, \q( X_{q+1},d_{q+1}\w)$.
For an integer $m\geq 1$, we define two conditions (all parameters
below are considered to be elements of $\q( 0,\infty\w)$):
\begin{enumerate}
\item $\sPo{m}{\kappa_1}{\tau_1}{\sigma_1}{\sigma_1^m}$:
\begin{itemize}
\item $\forall \delta\in \sA, x\in K, \q|\det_{n_0\q( x,\delta\w)\times n_0\q( x,\delta\w)} \q(\delta X\w)\q( x\w)\w|_{\infty} \geq \kappa_1 \q|\det_{n_0\q(x,\delta\w)\times n_0\q(x,\delta\w)} \q(\delta\hX\w)\q( x\w)\w|_{\infty}.$
\item $\forall x\in K, \q|\det_{j\times j} \hX\q( x\w) \w|=0$, $n_0\q( x,\delta\w)<j\leq n$.
\item $\forall \delta\in \sA, x\in K, \exists c_{i,q+1}^{j,x,\delta}\in C^0\q(\B{X}{d}{x}{\tau_1\delta}\w)$ such that
$$\q[\delta^{d_i} X_i,\delta^{d_{q+1}}X_{q+1}\w]=\sum_{j=1}^{q+1} c_{i,q+1}^{j,x,\delta} \delta^{d_j}X_j, \quad \text{on }\B{X}{d}{x}{\tau_1\delta},$$
with
$$\sum_{\q|\alpha\w|\leq m-1} \Cjn{0}{\B{X}{d}{x}{\tau_1\delta}}{\q(\delta X\w)^\alpha c_{i,q+1}^{j,x,\delta}}\leq \sigma_1^m, \quad \Cjn{0}{\B{X}{d}{x}{\tau_1\delta}}{c_{i,q+1}^{j,x,\delta}}\leq \sigma_1.$$
\end{itemize}
\item $\sPr{m}{\tau_3}{\sigma_3}{\sigma_3^m}$:  For every $x\in K, \delta\in \sA$, there exist $c_j^{x,\delta}\in C^0\q( \B{X}{d}{x_0}{\tau_3\delta} \w)$
such that:
\begin{itemize}
\item $\delta^{d_{q+1}}X_{q+1}=\sum_{j=1}^{q} c_j^{x,\delta} \delta^{d_j}X_j$, on $\B{X}{d}{x}{\tau_3\delta}$.
\item $\sum_{\q|\alpha\w|\leq m} \Cjn{0}{\B{X}{d}{x}{\tau_3\delta}}{\q(\delta X\w)^\alpha c_j^{x,\delta}}\leq \sigma_3^m$.
\item $\sum_{\q|\alpha\w|\leq 1} \Cjn{0}{\B{X}{d}{x}{\tau_3\delta}}{\q(\delta X\w)^\alpha c_j^{x,\delta}}\leq \sigma_3$.
\end{itemize}
\end{enumerate}

\begin{thm}
$\sPn{1}{m}\Leftrightarrow \sPn{3}{m}$
in the following sense:
\begin{enumerate}
\item $\sPo{m}{\kappa_1}{\tau_1}{\sigma_1}{\sigma_1^m} \Rightarrow$ there 
      exist admissible constants $\tau_3=\tau_3\q(\kappa_1,\tau_1,\sigma_1\w)$,
       $\sigma_3=\sigma_3\q( \kappa_1, \sigma_1\w)$, and an $m$-admissible
             constant $\sigma_3^m=\sigma_3^m\q( \kappa_1, \sigma_1^m \w)$ 
such that          $\sPr{m}{\tau_3}{\sigma_3}{\sigma_3^m}$.
\item $\sPr{m}{\tau_3}{\sigma_3}{\sigma_3^m} \Rightarrow$ there exist admissible
constants $\kappa_1=\kappa_1\q( \sigma_3\w)$, $\sigma_1=\sigma_1\q( \sigma_3\w)$
and an $m$-admissible constant $\sigma_1^m=\sigma_1^m\q( \sigma_3^m\w)$, such that 
$\sPo{m}{\kappa_1}{\tau_3}{\sigma_1}{\sigma_1^m}$.
\end{enumerate}
Furthermore, if $0<\dm$ is a fixed lower bound for $\q| d_{q+1}\w|_1$, then
under the condition $\sPr{m}{\tau_3}{\sigma_3}{\sigma_3^m}$,
we have that there exists an admissible constant
$\tau'=\tau'\q( \dm, \tau_3, \sigma_3 \w)$
such that:
$$\B{X}{d}{x}{\tau'\delta}\subseteq \B{\hX}{\hd}{x}{\tau'\delta}\subseteq \B{X}{d}{x}{\tau_3\delta}$$
for every $x\in K, \delta\in \sA$.
Finally, if $\eta'\leq \eta_1$ is small enough so that
$\Phi_{x,\delta}\q( B_{n_0\q( x,\delta\w)}\q( \eta'\w)\w) \subseteq \B{X}{d}{x_0}{\tau_3\delta}$
and we define $Y_{q+1}^{x,\delta}$ to be the pullback of $\delta^{d_{q+1}}X_{q+1}$ under
$\Phi_{x,\delta}$ to $B_{n_0\q(x,\delta\w)}\q( \eta'\w)$, then,
$$\Cjn{m}{B_{n_0\q( x,\delta\w)}\q( \eta'\w)}{Y_{q+1}^{x,\delta}}\leq \sigma_4^m$$
where $\sigma_4^m=\sigma_4^m\q( \sigma_3^m\w)$ is an $m$-admissible constant.
\end{thm}
\begin{proof}
Merely apply Theorem \ref{ThmUnitScaleEquivCond} and Propositions
\ref{PropUnderCondsBallContain} and \ref{PropPullBackUnderConds} for each $x\in K, \delta\in A$,
to the list of vector fields $\q( \delta X, \sd\w)$, taking $x_0=x$ and
$J_0=J\q( x,\delta\w)$.
\end{proof}

\begin{rmk}\label{RmkControlCommEveryScale}
Our assumption on the commutator $\q[ X_i, X_j\w]$ in Section
\ref{SectionBallsAtAPoint} was essentially that
$\q( \q[ X_i, X_j\w], d_i+d_j\w)$ satisfied condition
$\sPn{3}{m}$ for appropriate $m$.
\end{rmk}

\begin{defn}\label{DefnControl}
We say a vector field with a formal degree $\q( X_{q+1}, d_{q+1}\w)$ is
$m$-controlled by the list of vector fields $\q( X,d\w)$ provided
either of the two equivalent conditions $\sPn{1}{m}$ or $\sPn{3}{m}$
holds.  We say $\q( X_{q+1},d_{q+1}\w)$ is $\infty$-controlled by
$\q( X,d\w)$ if $\sPo{m}{\kappa_1}{\tau_1}{\sigma_1}{\sigma_1^m}$ holds
for every $m$, with $\kappa_1, \tau_1$, and $\sigma_1$ independent
of $m$ (equivalently if $\sPr{m}{\tau_3}{\sigma_3}{\sigma_3^m}$ holds
with $\tau_3$ and $\sigma_3$ independent of $m$).
\end{defn}

%% file: exampofcontrol.tex
For this section, we take all the same notation as in 
Section \ref{SectionBallsAtAPoint}, and assume that 
$\sA$ is given by (\ref{EqnConnonicalA}).
As was mentioned in Section \ref{SectionControlEveryScale} (see Remark \ref{RmkControlCommEveryScale})
our main assumption in Theorem \ref{ThmMultiBallsAtPoint}
is essentially that the commutator 
$\q( \q[ X_i, X_j\w], d_i+d_j\w)$ is ``controlled'' by $\q( X, d\w)$
in the sense of Definition \ref{DefnControl}.

In \cite{NagelSteinWaingerBallsAndMetricsDefinedByVectorFields}, a stronger
assumption was used in the single parameter case (see Section \ref{SectionNSW}).  The most obvious multi-parameter
analog of this assumption is the following:
\begin{equation}\label{EqnNSWIntegCond}
\q[ X_i, X_j\w] = \sum_{d_k\leq d_i+d_j} c_{i,j}^k X_k,
\end{equation}
where the inequality is meant coordinatewise, and the $c_{i,j}^k$
are assumed to be sufficiently smooth.  It is easy
to see that this assumption is a special case of the assumptions
in Section \ref{SectionBallsAtAPoint}:  indeed, one can take
$$c_{i,j}^{k,x,\delta}:=
\begin{cases}
\delta^{d_i+d_j-d_k} c_{i,j}^k & \text{if }d_k\leq d_i+d_j,
\\
0 & \text{otherwise}.
\end{cases}$$
One may wonder whether it is possible to get away with such
simple assumptions in applications.  This seems to not be the case,
and to exemplify the possible difficulties, in this section, we give
examples where the closely related notion of control takes a more
complicated form.

\begin{example}
This example takes place on $\R^2$ with the vector fields:
$$X_1=\partial_x, \quad X_2=x\partial_y, \quad X_3=\partial_y.$$
Create two copies of these vector fields:
$$X_1^j, X_2^j, X_3^j$$
$j=1,2,$ both copies acting on the {\it same} space.  We take $\nu=2$
and assign the formal degrees in $\q[0,\infty\w)^2$ as follows:
$$\q(X_1^1,\q(1,0\w)\w), \q(X_2^1,\q(1,0\w)\w), \q( X_3^1,\q(2,0\w)\w),\q(X_1^2,\q(0,1\w)\w),\q(X_2^2,\q(0,1\w)\w),\q( X_3^2,\q(0,2\w)\w).$$
It is clear that:
$$\q(\q[X_1^1, X_2^2\w], \q(1,0\w)+\q( 0,1\w)\w) = \q( \partial_y, \q( 1,1\w)\w)$$
is $\infty$-controlled by the other vector fields.  However, it is
easy to see that it cannot be written as in (\ref{EqnNSWIntegCond}).
In this case, one could just throw in the vector field $\q( \partial_y, \q( 1,1\w)\w)$, and then the list of vector fields would satisfy (\ref{EqnNSWIntegCond}), 
but this is not the case in the Example \ref{ExampleControlRadon}, below.
Furthermore, this process of adding in vector fields is counter to
the way in which we proceed in Section \ref{SectionMaximalFuncs}.
\end{example}

\begin{example}
Consider the vector fields with single-parameter formal degrees on $\R$ given by:
$$\q(\partial_x, 2\w), \q( x^2\partial_x, 1\w), \q( x\partial_x,1.5\w).$$
Denote them by $\q( X_j, d_j\w)$, $j=1,2,3$.
We restrict our attention to $\q|x\w|\leq 1$.  It is clear that for
every $\q|\delta\w|\leq 1$,
\begin{equation}\label{EqnExampleOfControl}
\q[\delta^{d_i}X_i, \delta^{d_j}X_j\w] = \sum_k c_{i,j}^{k,\delta} \delta^{d_k} X_k
\end{equation}
with $c_{i,j}^{k,\delta}\in C^\infty$ uniformly in $\delta$.
We claim that $\q( x\partial_x, 1.5\w)$ is $\infty$-controlled by the other two vector
fields.
Indeed, fix $x_0,\delta$, with $\delta,\q|x_0\w|\leq 1$.
By (\ref{EqnExampleOfControl}) it suffices to show that:
\begin{equation}\label{EqnExampleofControlToShow}
\q|\q(\delta^2, x_0^2\delta\w)\w|_{\infty}\geq \q|\q(\delta^2, x_0^2\delta, x_0\delta^{1.5}\w)\w|_{\infty}.
\end{equation}
Suppose $\delta^{1.5}\q|x_0\w| \geq \delta^2$.  Then $\q|x_0\w|\geq \sqrt{\delta}$.  Hence, $\delta^{1.5} \q|x_0\w| \leq \delta \q|x_0^2\w|$,
completing the proof of (\ref{EqnExampleofControlToShow}).

What this example shows is that we can write
$$\delta^{1.5} x\partial_x = c_1^{x_0,\delta} \delta^2 \partial_x + c_2^{x_0,\delta} \delta x^2 \partial_x,\text{ on } \B{X}{d}{x_0}{\tau_3\delta},$$
where $\tau_3$ can be chosen independent of $x_0,\delta$.  Note that the
choice of $c_1^{x_0,\delta},c_2^{x_0,\delta}$ depends on $x_0$ and $\delta$
in a way which is more complicated than arises from \eqref{EqnNSWIntegCond}:
it depends on the ratio of $\q|x_0\w|$ and $\sqrt{\delta}$.
\end{example}

\begin{example}\label{ExampleControlRadon}
Consider the vector fields with formal degrees on $\R$:
$$\q( \partial_x, \q( a,0\w)\w), \quad \q( \partial_x, \q( 0,b\w)\w)$$
here $a,b>0$.  In this example, we take $A=\q\{0\ne \delta, \q|\delta\w|<1\w\}$.
We will show that the above two vector fields $\infty$-control
$\q( \partial_x, \q( c,d\w)\w)$ if and only if the point $\q( c,d\w)$ lies
on or above the line going through $\q( a,0\w)$ and $\q( 0, b\w)$.  Here,
$c,d$ are any two non-negative real numbers, at least one of which
is non-zero.

By replacing $\delta=\q( \delta_1, \delta_2\w)$ with $\q( \delta_1^{\frac{1}{a}},\delta_2^{\frac{1}{b}}\w)$, it is easy to see that it suffices to
prove the result for $a=1=b$.  Hence we need to show that
$\q( \partial_x,\q( c,d\w)\w)$ is $\infty$-controlled by the above
two vector fields if and only if $c+d\geq 1$.  However, it is easy to see
that:
$$\q( \delta_1,\delta_2\w)^{\q( c,d\w)}\leq C \max\q\{\delta_1, \delta_2\w\}$$
for all $\delta$ sufficiently small, if and only if $c+d\geq 1$.
The result follows easily.
\end{example}

%% file: maxfuncsnew.tex
In this section, we wish to study maximal operators
associated to a special case of the multi-parameter balls from
Section \ref{SectionBallsAtAPoint}.

Suppose we are given $\nu$ families of $C^1$ vector fields on
$\Omega$ with associated single-parameter formal degrees:
\begin{equation*}
\q( X^\mu, d^\mu \w)=\q( \q(X^\mu_1, d^\mu_1\w), \ldots, \q(X^\mu_{q_\mu}, d^\mu_{q_\mu}\w)\w), \quad d^\mu_j\in \q( 0,\infty\w), \quad 1\leq \mu\leq \nu.
\end{equation*}
%We assume that the $X^\mu$ satisfy the following integrability conditions:
%\begin{equation}\label{EqnMaxFunInteg1}
%\q[ X_i^\mu, X_j^\mu\w]=\sum_{d_k^\mu\leq d_i^\mu+d_j^\mu} c_{i,j}^{k,\mu} X_k^\mu
%\end{equation}
%and if $\mu_1\ne \mu_2$:
%\begin{equation}\label{EqnMaxFuncInteg2}
%\q[X_i^{\mu_1},X_j^{\mu_2}\w] = \sum_{d_k^{\mu_1}\leq d_i^{\mu_1}+d_j^{\mu_2}} c_{i,j}^{k,\mu_1,\mu_2} X_k^{\mu_1} + \sum_{d_k^{\mu_2}\leq d_j^{\mu_2}} \ct_{i,j}^{k,\mu_1,\mu_2} X_k^{\mu_2}
%\end{equation}
%\begin{rmk}
%Notice that the LHS of (\ref{EqnMaxFuncInteg2}) is anti-symmetric,
%however the RHS is not.  Thus, by switching the roles of $\mu_1$ and
%$\mu_2$ one obtains a possibly {\it different} equation.  This is
%a special case of symmetry under $S_{\nu}$ that was used in
%Section \ref{SectionNoRestrictions}.
%\end{rmk}

We may associate to the $X^\mu$s and $d^\mu$s a family of vector fields
with (multi-parameter) formal degrees.  Indeed, let $\q(X,d\w)$
denote the list of vector fields $X_j^\mu$, $1\leq \mu\leq \nu$, $1\leq j\leq q_{\mu}$,
with the degree of $X_j^\mu$ the element of $\q[ 0,\infty\w)^\nu$ which
is $d_j^\mu$ in the $\mu$th coordinate and is $0$ in all
other coordinates.
%Note that by (\ref{EqnMaxFunInteg1}) and (\ref{EqnMaxFuncInteg2}) we
%have that $\q( X,d\w)$ satisfies (\ref{EqnBallsOnCompactInteg}).
Define $K\Subset \Omega$ and $\xi$ as in Section \ref{SectionMultipleLists},
in terms of $\q( X,d\w)$.  
We assume that the list of vector fields $\q( X,d\w)$ satisfies
all of the assumptions of Section \ref{SectionBallsAtAPoint} (with
$\sA$ given by \eqref{EqnConnonicalA}), {\it without adding
any new vector fields to the list $\q( X,d\w)$}.\footnote{In particular,
this implies that the (one-parameter) list of vector fields $\q( X^{\mu_0}, d^{\mu_0}\w)$ satisfies the assumptions of Section \ref{SectionBallsAtAPoint},
for each $\mu_0$.
This can be seen by taking $\delta_\mu=0$ for every $\mu\ne \mu_0$.}
We 
define
admissible constants in the same way as they were defined in 
Section \ref{SectionBallsAtAPoint}.  We define,
for $\delta\in \q(0,1\w]^\nu$, $\delta\leq \xi$, $x\in K$:
$$B_{\q(X^1, d^1\w),\ldots, \q(X^\nu, d^\nu\w)}\q(x, \delta\w):=\B{X}{d}{x}{\delta}.$$
We present two interesting examples that satisfy the hypotheses of this section:
\begin{example}\label{ExHormanderWithEuclid}
Suppose $X_1,\ldots, X_m$ are $C^\infty$ vector fields satisfy H\"ormander's condition.  Use
these to generate a list of vector fields with formal degrees 
$\q( X,d\w)$ as in \cite{NagelSteinWaingerBallsAndMetricsDefinedByVectorFields} (see Section \ref{SectionNSW}).  Let $\nu=2$ and let $\q(X,d\w)$
be the list of vector fields corresponding to $\mu=1$.  Then let
$\q( \partial_1,1\w), \ldots, \q(\partial_n,1\w)$ denote the list
of vector fields corresponding to $\mu=2$.  I.e., $\mu=2$ corresponds
to the usual Euclidean vector fields.  These then satisfy
the hypotheses of the section.  The main idea in this example
is that one may write:
\begin{equation}\label{EqnEculidWithHormander}
\q[X_i, \partial_j\w]=\sum_k a_{i,j}^k \partial_k,
\end{equation}
where the $a_{i,j}^k\in C^\infty$.
One must be careful with this example.
It is tempting to think, given the results in Section \ref{SectionMultipleLists},
that one could take any two lists satisfying H\"ormander's condition,
removing the assumption that one of the lists corresponds to the
usual Euclidean vector fields.  This is not the case, since the
procedure in Section \ref{SectionMultipleLists} involved
adding more vector fields to the list $\q( X,d\w)$, namely the commutators
involving vector fields from both lists.  In this case, though,
this procedure is not necessary, due to (\ref{EqnEculidWithHormander}).
\end{example}
\begin{example}
Let $\q( X^\mu, d^\mu\w)$ be $\nu$ lists of vector fields with
formal degrees such that for each fixed $\mu$, $\q( X^\mu, d^\mu\w)$
satisfies the hypotheses of Section \ref{SectionBallsAtAPoint} (with $\nu=1$ and
$\sA$ given by \eqref{EqnConnonicalA}).
Suppose further that for $\mu_1\ne \mu_2$,
$\q[X_i^{\mu_1}, X_j^{\mu_2}\w]=0$.
Then these vector fields satisfy the hypotheses of this section.
In particular, when working on a homogeneous group, one could take
$\nu=2$, and let the $\mu=1$ vector fields correspond to
a homogeneous basis of the left invariant vector fields (with
degrees corresponding to their homogeneity) and $\mu=2$ be a similar
list but instead with the right invariant vector fields.
This was the setup in \cite{StreetAdvances}, and is discussed
in Section \ref{SectionStreet}.
\end{example}

To motivate the results in this section, let us consider
a classical example.  In this case $\Omega=\R^\nu$, $q_\mu=1$ for all $1\leq \mu\leq \nu$,
and $\q(X^\mu_1, d^\mu_1\w)=\q(\partial_\mu,1\w)$.
We have the classical ``strong'' maximal
function in $\R^{\nu}$.  This is given by:
\begin{equation}\label{EqnDefnMax}
\sM f\q(x\w) :=\sup_{\substack{\delta=\q(\delta_1,\ldots, \delta_n\w)\\ \delta_j>0}} \frac{1}{\Vol{B_{\q(\partial_1,1\w), \ldots, \q(\partial_\nu, 1\w)} \q(x,\delta\w)}}
\int_{B_{\q(\partial_1,1\w), \ldots, \q(\partial_\nu, 1\w)} \q(x,\delta\w)} \q|f\q(z\w)\w| dz.
\end{equation}
Rewriting \eqref{EqnDefnMax} in the notation of Section \ref{SectionUnitScale},
we have:
\begin{equation*}
\sM f\q( x\w) = \sup_{\delta} A_{B_{\q(\partial_1,1\w), \ldots, \q(\partial_\nu, 1\w)} \q(\cdot ,\delta\w)}\q|f\w| \q( x\w).
\end{equation*}
Perhaps the easiest way to deduce $L^p$ boundedness ($1<p\leq \infty$)
for $\sM$ is the idea of Jessen, Marcinkiewicz, and Zygmund \cite{JessenMarcinkiewiczZygmundNoteOnDifferentiability}
to bound $\sM$ by a product of the one-dimensional maximal functions,
whose $L^p$ boundedness is already understood.
To do this, one proves the simple inequality, that there exists a
$\lambda>0$ such that for every $\delta$, and every $f\geq 0$, we have:
$$A_{B_{\q(\partial_1,1\w), \ldots, \q(\partial_\nu, 1\w)} \q(\cdot,\lambda\delta\w)} f\leq C A_{\B{\partial_\nu}{1}{\cdot}{\delta_\nu}}\cdots A_{\B{\partial_1}{1}{\cdot}{\delta_1}} f$$
And then it follows immediately, that:
\begin{equation}\label{EqnMaxClassicBound}
\sM f\q( x\w) \leq C \sM_\nu \cdots \sM_1 f\q(x\w),
\end{equation}
where
$$\sM_\mu f\q(x\w) = \sup_{\delta_\mu>0} A_{\B{\partial_\mu}{1}{\cdot}{\delta_\mu}} \q|f\w|.$$
In this section, we wish to generalize (\ref{EqnMaxClassicBound}).
%\begin{rmk}
%Of course (\ref{EqnMaxClassicBound}) does not yield the best results for
%$\sM$.  As is well known, $\sM$ satisfies certain end-point estimates when $p=1$, see \cite{CordobaFeffermanAGeometricProofOfTheStrongMaximalTheorem}, which
%do not follow from (\ref{EqnMaxClassicBound}).  However,
%in this section we are interested in generalizing the above proof, and
%do not worry about end-point estimates.
%\end{rmk}

\begin{thm}\label{ThmMainMaxFuncs}
There exist admissible constants, $0<\tau_2<\tau_1<1$, $\sigma>0$ such
that for all $\q|\delta\w|\leq \sigma$, we have, for $f\in C\q( \Omega\w)$,
$f\geq 0$, $x\in K$:
\begin{equation*}
\begin{split}
A_{\Bmus{\cdot}{\tau_2 \delta}} f\q( x\w) &\lesssim 
A_{\B{X^\nu}{d^\nu}{\cdot}{\tau_1 \delta}} \cdots A_{\B{X^1}{d^1}{\cdot}{\tau_1 \delta}} f\q( x\w)\\
&\lesssim
A_{\Bmus{\cdot}{\delta}} f\q( x\w).
\end{split}
\end{equation*}
\end{thm}
\begin{proof}
Apply Theorem \ref{ThmUnitOpsMainThm} with $\q( Z^\mu, d^\mu\w) = \q( \delta_\mu X^\mu, d^\mu\w)$,
$\q( X,d\w)= \q( \delta X, d\w)$, and $x_0=x$, where $x\in K$.
We obtain admissible constants $\lambda_1,\lambda_2,\lambda_3$ independent
of $x,\delta$ as in that theorem.

To conclude the proof, merely take $\tau_1=\frac{\lambda_2}{\lambda_1}$,
$\tau_2=\frac{\lambda_3}{\lambda_1}$ and replace $\delta$ with
$\lambda_1\delta$.
\end{proof}

\begin{cor}\label{CorMaxIneq}
There exist admissible constants $0<\tau_2<\tau_1<1$, $\sigma>0$ such that
if we define, for $x\in K$, $f\in C\q( \Omega\w)$,
\begin{equation*}
\begin{split}
\sM f\q( x\w) &= \sup_{\q|\delta\w|\leq \tau_2\sigma} A_{\Bmus{\cdot}{\delta}}\q|f\w| \q(x\w),\\
\end{split}
\end{equation*}
and for all $x\in \Omega$ such that $\B{X^\mu}{d^\mu}{x}{\tau_1\sigma}\subset \Omega$,
\begin{equation*}
\begin{split}
\sM_{\mu} f\q( x\w) &= \sup_{0<\delta_\mu\leq \tau_1\sigma} A_{\B{X^\mu}{d^\mu}{\cdot}{\delta_\mu}} \q|f\w|\q( x\w).
\end{split}
\end{equation*}
Then we have:
\begin{equation}\label{EqnMaxIneq}
\sM f\q( x\w) \lesssim \sM_\nu \sM_{\nu-1}\cdots \sM_{1} f\q( x\w).
\end{equation}
%and as a consequence, $\sM$ extends to a bounded operator $L^p\q( \Omega\w) \rightarrow L^p\q( K\w)$, ($1<p<\infty$).
\end{cor}
\begin{proof}
This follows directly from Theorem \ref{ThmMainMaxFuncs}.
%(\ref{EqnMaxIneq}) directly from Theorem \ref{ThmMainMaxFuncs}.
%To see that $\sM$ is bounded on $L^p$ ($1<p< \infty$),
%it suffices to show that $\sM_{\mu}$ is bounded on $L^p$ ($1<p< \infty$, $1\leq \mu\leq \nu$); however, the $L^p$ boundedness of $\sM_{\mu}$ is well-known,
%since the {\it one}-parameter balls $\B{X^\mu}{d^\mu}{\cdot}{\delta_\mu}$ satisfy
%the doubling condition (\ref{EqnDoublingConditionInMainThm}).
\end{proof}

\begin{cor}\label{CorLpMaxIneq}
Let $\sM$ be defined as in Corollary \ref{CorMaxIneq}.
Then, by possibly admissibly shrinking $\tau_2$, we have that
$\sM$ extends to a bounded map $L^p\q( \Omega\w) \rightarrow L^p\q( K\w)$,
for every $1<p<\infty$.
\end{cor}
\begin{proof}
This would follow from \eqref{EqnMaxIneq}, provided we have that
$\sM_1,\ldots,\sM_\nu$ extend to bounded operators on $L^p$ ($1<p<\infty$).
Intuitively, this is simple, since the {\it one}-parameter balls
$\B{X^\mu}{d^\mu}{\cdot}{\delta_\mu}$ satisfy
the doubling condition
\eqref{EqnDoublingConditionInMainThm}, and we expect to be able to apply
the theory of spaces of homogeneous type to conclude the desired $L^p$
boundedness.  There is a slight technicality, though, since if (for
a fixed $\mu$), the vector fields $X^\mu$ do not span the tangent space,
then the balls $\B{X^\mu}{d^\mu}{\cdot}{\delta_\mu}$ do not endow
$\Omega$ with the structure of a space of homogeneous type:  rather, they foliate
$\Omega$ into leaves, each of which is (locally) a space of homogeneous type.
The technical details to deal with this difficulty are covered
in Section \ref{SectionHomogeneousFoliations}.
\end{proof}

\begin{rmk}
Notice, in Corollary \ref{CorLpMaxIneq}, we have left out $p=\infty$.  This is because if the vector fields
do not span the tangent space, the maximal operators may only be
{\it a priori} defined on $C\q( \Omega\w)$.
\end{rmk}

\begin{rmk}
As mentioned in Section \ref{SectionStreet}, it is likely that $\sM$
is bounded on $L^p$ ($1<p<\infty$) for a larger class of balls than is
discussed in this section.  However, Theorem \ref{ThmMainMaxFuncs}
is very tied to the assumptions of this section.
\end{rmk}

\begin{cor}\label{CorSchwartzKer}
There exists an admissible constant $\sigma_1>0$ such that
if for $\q|\delta\w|\leq \sigma_1$, we let $T_\delta$ denote
the Schwartz kernel for the operator:
\begin{equation*}
A_{\B{X^\nu}{d^\nu}{\cdot}{\delta_\nu}} \cdots A_{\B{X^1}{d^1}{\cdot}{\delta_1}},
\end{equation*}
then, for $x\in K$, $T_\delta\q( x,y\w)$ is supported on
those points $\q( x,y\w)$ such that:
\begin{equation}\label{EqnFirstSKEqn}
\inf \q\{ \tau>0 : y\in \Bmus{x}{\tau \delta}\w\}\lesssim 1,
\end{equation}
and moreover,
\begin{equation}\label{EqnSecondSKEqn}
\sup_{y} T_\delta\q( x,y\w) \approx \frac{1}{\Vol{\Bmus{x}{\delta}}}.
\end{equation}
Furthermore, there exists an admissible constant $\sigma_2>0$ such
that:
\begin{equation}\label{EqnThirdSKEqn}
T_\delta\q( x,y\w) \approx \frac{1}{\Vol{\Bmus{x}{\delta}}}, \quad y\in \Bmus{x}{\sigma_2\delta}.
\end{equation}
These inequalities are understood to be taking place on the leaf
generated by $\delta X$ passing through $x$.
\end{cor}
\begin{proof}
(\ref{EqnFirstSKEqn}) and the $\lesssim$ part of (\ref{EqnSecondSKEqn})
follow by applying Theorem \ref{ThmMainMaxFuncs} and using the
inequality (for $f\geq 0$):
\begin{equation*}
A_{\B{X^\nu}{d^\nu}{\cdot}{\tau_1 \delta}} \cdots A_{\B{X^1}{d^1}{\cdot}{\tau_1 \delta}} f\q( x\w)
\lesssim
A_{\Bmus{\cdot}{\delta}} f\q( x\w)
\end{equation*}
and renaming $\tau_1\delta$, $\delta$.
With this new $\delta$, the other half of Theorem \ref{ThmMainMaxFuncs}
now reads:
\begin{equation*}
A_{\Bmus{\cdot}{\frac{\tau_2}{\tau_1} \delta}} f\q( x\w) \lesssim 
A_{\B{X^\nu}{d^\nu}{\cdot}{\delta}} \cdots A_{\B{X^1}{d^1}{\cdot}{\delta}} f\q( x\w),
\end{equation*}
thereby establishing (\ref{EqnThirdSKEqn}) and therefore the
$\gtrsim$ part of (\ref{EqnSecondSKEqn}).  This completes the proof.
\end{proof}

Corollary \ref{CorSchwartzKer} has an interesting corollary, to which we
now turn.  For the statement of this corollary, we restrict our
attention to the case $\nu=2$, but otherwise keep the same assumptions and
notation as in the rest of the section.

\begin{cor}
Suppose the leaf generated by $X^1$ passing through $x_0$ is the same
as the leaf generated by $X^2$ passing through $x_0$ (call this common
leaf $L$).
Then, there exists an admissible constant $\sigma_1>0$ such that
for every $x_0\in K$ and every $\delta:=\q(\delta_1,\delta_2\w)$ with $\q|\delta\w|\leq \sigma_1$, we have:
\begin{equation*}
\Vol{\B{X^1}{d^1}{x_0}{\delta_1}\cap \B{X^2}{d^2}{x_0}{\delta_2}} \approx \frac{\Vol{\B{X^1}{d^1}{x_0}{\delta_1}}\Vol{\B{X^2}{d^2}{x_0}{\delta_2}}} { \Vol{ \Bmust{x_0}{\delta}  }  }.
\end{equation*}
Here, $\Vol{\cdot}$ on the left hand side denotes the induced Lebesgue
volume on $L$.
\end{cor}
\begin{proof}
In the following, $dy$ will denote
the induced Lebesgue measure on $L$, and for a set $A$,
$\chi_A$ will denote the characteristic function of $A$.
\begin{equation}\label{EqnVolInter1}
\begin{split}
&\Vol{\B{X^1}{d^1}{x_0}{\delta_1}\cap \B{X^2}{d^2}{x_0}{\delta_2}} = \int \chi_{\B{X^1}{d^1}{x_0}{\delta_1}}\q( y\w) \chi_{\B{X^2}{d^2}{y}{\delta_2}}\q(x_0\w) dy\\
&\quad = \Vol{\B{X^1}{d^1}{x_0}{\delta_1}}\Vol{\B{X^2}{d^2}{x_0}{\delta_2}} 
\\&\quad\quad\times\int \frac{1}{\Vol{\B{X^1}{d^1}{x_0}{\delta_1}}}\chi_{\B{X^1}{d^1}{x_0}{\delta_1}}\q( y\w) \frac{1}{\Vol{\B{X^2}{d^2}{x_0}{\delta_2}}}\chi_{\B{X^2}{d^2}{y}{\delta_2}}\q(x_0\w) dy.
\end{split}
\end{equation}
In the above, we have used that $y\in \B{X^2}{d^2}{x_0}{\delta_2}$ if and
only if $x_0\in \B{X^2}{d^2}{y}{\delta_2}$.
We use the fact that if $y\in \B{X^2}{d^2}{x_0}{\delta_2}$, then
$$\Vol{\B{X^2}{d^2}{x_0}{\delta_2}}\approx \Vol{\B{X^2}{d^2}{y}{\delta_2}},$$
which follows from Theorem \ref{ThmMultiBallsAtPoint}, in particular (\ref{EqnDoublingConditionInMainThm}).
We then have that the RHS of (\ref{EqnVolInter1}) is:
\begin{equation*}
\begin{split}
&\approx \Vol{\B{X^1}{d^1}{x_0}{\delta_1}}\Vol{\B{X^2}{d^2}{x_0}{\delta_2}}
\\&\quad\times\int \frac{1}{\Vol{\B{X^1}{d^1}{x_0}{\delta_1}}}\chi_{\B{X^1}{d^1}{x_0}{\delta_1}}\q( y\w) \frac{1}{\Vol{\B{X^2}{d^2}{y}{\delta_2}}}\chi_{\B{X^2}{d^2}{y}{\delta_2}}\q(x_0\w) dy
\\&= \Vol{\B{X^1}{d^1}{x_0}{\delta_1}}\Vol{\B{X^2}{d^2}{x_0}{\delta_2}}T_{\delta}\q( x_0,x_0\w).
\end{split}
\end{equation*}
Where $T_\delta$ is as in Corollary \ref{CorSchwartzKer}.  Now
the result immediately follows from Corollary \ref{CorSchwartzKer}.
\end{proof}

%% file: metrics.tex
Corollary \ref{CorSchwartzKer} has a corollary which can be phrased in terms of metrics,
and may serve to give the reader some intuition for these results.
We devote this section to this corollary, and maintain all the
same notation as in Section \ref{SectionMaximalFuncs}.

Fix $r=\q(r_1, \ldots, r_\nu\w)\in \q(0,1\w]^\nu$, and assume 
$r_\mu=1$ for some $\mu$.  Corresponding to each such $r$, we
obtain a {\it one}-parameter family of balls:
$$\B{X}{d}{x}{\delta r},$$
for $x\in \Omega$.  This one-parameter family of balls
is associated to the Carnot-Carath\'eodory metric $\rho_r$, associated to the vector
fields
$$\q\{ \q(r_\mu^{d^\mu_j}X^\mu_j, d^\mu_j\w): 1\leq \mu\leq \nu, 1\leq j\leq q_\mu \w\}.$$
This metric is defined by:
$$\rho_r\q( x,y\w):=\inf \q\{\delta>0: y\in \B{X}{d}{x}{\delta r}\w\}.$$
\begin{rmk}\label{RmkMightNotBeAMetric}
Actually, it could be that $\rho_r$ is not a metric, in that if the
$X^\mu_j$ do not span the tangent space, the distance between two points
might be $\infty$.  This will not affect any of the results in this section.
\end{rmk}
%\begin{rmk}\label{RmkNotUsingContx0xi}
%In the definition of $\rho_r$, we have been a little loose with our assumptions.
%Indeed, we have used the ball $\B{X}{d}{x}{\delta r}$ for $\q|\delta r\w|$
%as large as we like, ignoring the restriction imposed by Definition \ref{DefnCondx0xi}.  This will cause us no problems in this section.
%\end{rmk}
\begin{rmk}
We assumed that $\max_{\mu}r_\mu=1$ since, if we drop this assumption,
we have:
$$\delta \rho_r = \rho_{\delta^{-1} r}$$
and so every choice of $r$ can be reduced to the case when $\max_{\mu}r_\mu=1$.
\end{rmk}

For each $\mu$ we also obtain a metric, the Carnot-Carath\'eodory metric
associated to the vector fields
$$\q\{\q(X^\mu_j,d^\mu_j\w):1\leq j\leq q_\mu\w\}$$
given by
$$\rho_{\mu}\q(x,y\w):=\inf \q\{\delta>0:y\in \B{X^\mu}{d^\mu}{x}{\delta}\w\},$$
where we have the same caveat as in Remark \ref{RmkMightNotBeAMetric}.% and \ref{RmkNotUsingContx0xi}.

Given two functions $\Delta_1,\Delta_2:\Omega\times \Omega\rightarrow \q[0,\infty\w]$
(one should think of $\Delta_1, \Delta_2$ as metrics) we obtain a new function:
$$\q(\Delta_1\circ\Delta_2\w)\q(x,z\w):=\inf_{y\in \Omega} \Delta_1\q( x,y\w) + \Delta_2\q( y,z\w).$$
One should think of $\Delta_1\circ \Delta_2$ as the ``distance'' between
$x$ and $z$ if one is first allowed to travel in the $\Delta_1$ metric
and then in the $\Delta_2$ metric.  Of course, even if $\Delta_1$ and
$\Delta_2$ are metrics, $\Delta_1\circ \Delta_2$ may not be
symmetric, and therefore will not be a metric.

However, in the case of the $\rho_\mu$ above, we do end up with a
quasi-metric.  Indeed, we have:
\begin{cor}\label{CorApproxMetrics}
%For $x,y\in K$, we have:
There is an admissible constant $\sigma_2>0$ such that for every
$x\in K$ and $y\in \Omega$ such that $\rho_r\q( x,y\w)<\sigma_2$, we have:
$$\q[\q(r_1^{-1} \rho_1\w)\circ\q(r_2^{-1} \rho_2\w)\circ\cdots \circ\q(r_\nu^{-1} \rho_\nu\w)\w]\q(x,y\w)\approx \rho_r\q(x,y\w).$$
\end{cor}
\begin{proof}
$\gtrsim$ is obvious.  $\lesssim$ follows from Corollary \ref{CorSchwartzKer}.
\end{proof}

\begin{rmk}
When $\q[X_j^{\mu_1}, X_k^{\mu_2}\w]=0$ for $\mu_1\ne \mu_2$
(and with a slight modification in the definition of our metrics),
 one actually
obtains equality in Corollary \ref{CorApproxMetrics} (as opposed to
$\approx$).  This follows from the proof method in Section 4.1 of \cite{StreetAdvances}.
\end{rmk}

%% file: homogfoli.tex
Let $\q( X_1,d_1\w),\ldots, \q( X_q,d_q\w)$ be vector fields on an open
set $\Omega\subseteq \R^n$ with
single-parameter formal degrees $d_1,\ldots, d_q\in \q( 0,\infty\w)$.
Under the (single-parameter version of) the hypotheses in
Section \ref{SectionBallsAtAPoint}, the balls $\B{X}{d}{x}{\delta}$
satisfy the doubling property that is crucial to the theory
of spaces of homogeneous type:
\begin{equation}\label{EqnDoublingCondFol}
\Vol{\B{X}{d}{x_0}{2\delta}}\lesssim \Vol{\B{X}{d}{x_0}{\delta}},
\end{equation}
see \eqref{EqnDoublingConditionInMainThm}.
This leads one to consider the maximal operator given by,
$$\sM f\q( x\w) = \sup_{\delta>0} A_{\B{X}{d}{\cdot}{\delta}}\q|f\w|\q( x\w)
= \sup_{\delta>0} \frac{1}{\Vol{\B{X}{d}{x}{\delta}}} \int_{\B{X}{d}{x}{\delta}} \q|f\q( y\w)\w|\: dy,$$
where the supremum is only taken over $\delta$ sufficiently small.
If the vector fields $X_1,\ldots, X_q$ spanned the tangent space
at every point of $\Omega$, the balls $\B{X}{d}{x}{\delta}$ would be
open sets of positive Lebesgue measure and \eqref{EqnDoublingCondFol}
would imply that they do, in fact, endow $\Omega$ with the structure
of a space of homogeneous type.  Classical arguments then show
that $\sM$ extends to a bounded operator on $L^p$ ($1<p\leq \infty$)--this
is essentially the situation covered in
\cite{NagelSteinWaingerBallsAndMetricsDefinedByVectorFields}.

However, if the vector fields do not span the tangent space
at each point, then the balls $\B{X}{d}{x}{\delta}$ do not turn
$\Omega$ into space of homogeneous type.  Indeed, at a point $x_0$
where $X_1,\ldots, X_q$ do not span the tangent space,
the ball $\B{X}{d}{x_0}{\delta}$ does not even have positive $n$ dimensional Lebesgue
measure:  it lies on the leaf passing through $x_0$ generated
by $X_1,\ldots, X_q$.
This does not prevent $\sM$ from extending to a bounded operator
on $L^p$ ($1<p<\infty$) however, as we shall see.
Informally, the idea is that $X_1,\ldots, X_q$ foliate $\Omega$
into leaves, where each leaf (endowed with the induced Lebesgue measure)
is locally a space of homogeneous type, and the standard
theory of maximal functions may be applied.  It is crucial, here,
that we are working locally (i.e., that we are restricting
our attention to $\delta>0$ small).  If we had not restricted
our attention to local results, the space of leaves might be quite
complicated, to the extent that it would be difficult (if not impossible)
to lift the $L^p$ boundedness of $\sM$ from each leaf to $\Omega$.\footnote{Consider,
for instance, the vector field $\partial_x+\theta\partial_y$ on the manifold
$M=\R^2/\Z^2$, where $\theta\in \R\setminus \Q$.  In this case, if we denote by
$\sL$ the space of leaves, we have $L^p\q( \sL\w)=\C$.  Locally, however,
$M$ with this foliation just looks like a product space, and the corresponding
maximal function is just the standard maximal function along
one of the variables.}

\begin{rmk}
Near a non-singular point\footnote{$x_0\in \Omega$ is said to be a non-singular
point if $\dim \Span{X_1\q( x\w) ,\ldots, X_q\q( x\w)}$ is constant
in a neighborhood of $x_0$.}
of the involutive distribution spanned by $X_1,\ldots, X_q$, the boundedness of
$\sM$ follows immediately.  Indeed, in this case, the foliation
looks locally like a product space.  The maximal function
just acts on one of the product variables, and in this
variable the balls form a space of homogeneous type.
Thus, the main point of this section is to demonstrate an
easy way to deal with singular points;
however, it will not be necessary for us to make
any distinction between non-singular and singular points in our
argument.
\end{rmk}

We now turn to a formal statement of our results.  We are given a compact
set $K\Subset \Omega$, and $\xi>0$ such that $\q( X,d\w)$ satisfies
$\sC\q( x,\xi\w)$ for every $x\in K$.
We assume for every $\delta\leq \xi$, $x\in K$, we have:
$$\q[\delta^{d_i}X_i, \delta^{d_j} X_j\w]=\sum_k c_{i,j}^{k,\delta,x} \delta^{d_k} X_k,$$
on $\B{X}{d}{x}{\delta}$.  In addition, we assume:
\begin{itemize}
\item The $X_j$s are $C^2$ on $\B{X}{d}{x}{\xi}$, for every $x\in K$, and satisfy $\sup_{x\in K}\Cjn{2}{\B{X}{d}{x}{\xi}}{X_j}<\infty$.
\item For all $\q|\alpha\w|\leq 2$, $x\in K$, we have $\q(\delta^d X\w)^\alpha c_{i,j}^{k,\delta,x}\in \Cj{0}{\B{X}{d}{x}{\delta}}$, for every $i,j,k$, and every $\delta\in \sA$, and moreover:
$$\sup_{\substack{\delta\in A\\x\in K}} \sum_{\q|\alpha\w|\leq 2} \Cjn{0}{\B{X}{d}{x}{\delta}}{\q(\delta^d X\w)^\alpha c_{i,j}^{k,\delta,x}}<\infty.$$
\end{itemize}
Finally, let 
$$n_0\q( x,\delta\w)=\dim \Span{\delta^{d_1}X_1\q( x\w), \ldots, \delta^{d_q}X_q\q( x\w)}.$$

We say $C$ is an admissible constant if $C$ can be chosen to depend only
on fixed upper and lower bounds $d_{max}<\infty$, $d_{min}>0$, 
for
$d_1,\ldots, d_q$, 
a fixed upper bound for $n,q$
and a fixed upper bound for the quantities:
\begin{equation*}
\sup_{x\in K}\Cjn{2}{\B{X}{d}{x}{\xi}}{X_j}, \quad \sup_{\substack{\delta\in A\\x\in K}} \sum_{\q|\alpha\w|\leq 2} \Cjn{0}{\B{X}{d}{x}{\delta}}{\q(\delta^d X\w)^\alpha c_{i,j}^{k,\delta,x}}.
\end{equation*}

\begin{thm}\label{ThmSingleParamMaxFunc}
There exists an admissible constant $\xi'>0$, $\xi'\leq \xi$, such that
if we define, for $f\in C\q( \Omega\w)$, $x\in K$,
\begin{equation}\label{EqnDefnSingleMax}
\sM f\q( x\w) = \sup_{0<\delta\leq \xi'} A_{\B{X}{d}{\cdot}{\delta}} \q|f\w| \q( x\w),
\end{equation}
then for every $f\in C\q( \Omega\w)\cap L^p\q( \Omega\w)$,
$$\q\|\sM f\w\|_{L^p\q( K\w)}\leq C_p \q\|f\w\|_{L^p\q( \Omega\w)},$$
for every $1<p\leq \infty$.  Here, $C_p$ is an admissible constant
which may also depend on $p$.
\end{thm}

The main assumptions of this section are equivalent to saying
that Theorem \ref{ThmMultiBallsAtPoint} applies to the vector fields
$\q( X,d\w)$.  Let $\xi_1,\eta_1$ be admissible constants
as in the conclusions of Theorem \ref{ThmMultiBallsAtPoint}.
Define,
$$\Omega' = \bigcup_{x\in K} \B{X}{d}{x}{\frac{\xi_1}{2}},$$
$$\Omega'' = \bigcup_{x\in K} \B{X}{d}{x}{\frac{\xi_1}{4}}.$$
Theorem \ref{ThmSingleParamMaxFunc} will follow from the following
two propositions.
\begin{prop}\label{PropAvgIsInt}
There exists an admissible constant $\xi_0>0$, $\xi_0<\xi$, such that
for every $\xi'\leq \xi_0$ and every $f\in C\q( \Omega\w)$ with 
$f\geq 0$, we have:
$$\int_K f\q( x\w)\: dx \lesssim \int_{\Omega''} A_{\B{X}{d}{\cdot}{\xi'}} f\q( x\w) \: dx\lesssim \int_{\Omega'} f\q( x\w) \: dx,$$
where the implicit constants are admissible but also allowed to depend on a lower bound for $\xi'$.
\end{prop}

\begin{prop}\label{PropLeafwiseMax}
%There exists an admissible constant $\xi_1>0$ such that for every $\xi'<\frac{\xi_1}{2}$, 
We have the 
pointwise
bound, for $1<p\leq \infty$, $0<\xi'\leq \frac{\xi_1}{4}$, and $x\in \Omega''$:
$$A_{\B{X}{d}{\cdot}{\xi'}} \q|\sM f\w|^p \q( x\w) \lesssim A_{\B{X}{d}{\cdot}{2\xi'}} \q|f\w|^p\q( x\w),$$
where the implicit constant is admissible and can also depend on $p$ and
a lower bound for $\xi'$, but
not on $x$.
\end{prop}

\begin{proof}[Proof of Theorem \ref{ThmSingleParamMaxFunc} given Propositions \ref{PropAvgIsInt} and \ref{PropLeafwiseMax}]
Fix $p>1$.  Let $\xi_0$ be as in Propositions \ref{PropAvgIsInt}.
Fix $\xi'>0$ an admissible
constant, such that $\xi'<\min\q\{\frac{\xi_0}{2}, \frac{\xi_1}{4}\w\}$.
Take $f\in C\q( \Omega\w)\cap L^p\q( \Omega\w)$, and consider:
\begin{equation*}
\begin{split}
\q\|\sM f\w\|_{L^p\q( K\w)}^p &= \int_K \q(\sM f\q( x\w)\w)^p \: dx\\
&\lesssim \int_{\Omega''} A_{\B{X}{d}{\cdot}{\xi'}} \q|\sM f\w|^p \q( x\w)\: dx\\
&\lesssim \int_{\Omega''} A_{\B{X}{d}{\cdot}{2\xi'}} \q|f\w|^p \q( x\w) \: dx\\
&\lesssim \int_{\Omega'} \q|f\q( x\w)\w|^p \: dx\\
&\lesssim \q\|f\w\|_{L^p\q( \Omega\w)}^p,
\end{split}
\end{equation*}
completing the proof.
\end{proof}

We now prove Proposition \ref{PropAvgIsInt}.  To do so, we need
two lemmas.

\begin{lemma}\label{LemmaAvgIsIntO}
There exists an admissible constant $\eta^0>0$ such that for every
$0<\eta'\leq \eta^0$, and every $f\in C\q( \Omega\w)$ with $f\geq 0$,
we have:
$$\int_K f\q( x\w) \: dx \lesssim \frac{1}{\q(2\eta'\w)^q} \int_{\q|t\w|\leq \eta'} \int_{\Omega''} f\q( e^{t\cdot X}x\w)\: dx\: dt\lesssim \int_{\Omega'}f\q( x\w) \: dx.$$
\end{lemma}
\begin{proof}
Note, for $\q|t\w|\leq \frac{\xi_1}{4}$, we have
$$\Omega'\supseteq e^{t\cdot X}\Omega''\supseteq K.$$
To make use of this, we choose $\eta^0\leq \frac{\xi_1}{4}$.  Furthermore,
by taking $\eta^0>0$ admissible small enough, we have for all $\q|t\w|\leq \eta^0$, and all $x\in \Omega''$,
$$\q|\det \frac{\partial}{\partial_x} e^{t\cdot X}x\w|\geq \frac{1}{2}.$$
This follows since the $C^2$ norm of $e^{t\cdot X}x-x$ is admissibly
bounded (see Theorem \ref{ThmExpReg}) and because when
$t=0$, $e^{t\cdot X}x=x$.

Putting these results together, we have from a simple change of variables,
for $\q|t\w|\leq \eta^0$,
$$\int_K f\q( x\w) \: dx \lesssim \int_{\Omega''} f\q( e^{t\cdot X} x\w) \: dx \lesssim \int_{\Omega'} f\q( x\w) \: dx.$$
Averaging both sides over $\q|t\w|\leq \eta'$ yields the proof.
\end{proof}

\begin{lemma}\label{LemmaAvgIsIntT}
Let $\eta^0$ be as in Lemma \ref{LemmaAvgIsIntO}.  There exists an
admissible constant $\xi_0>0$ such that for every
$0<\xi'\leq \xi_0$, there exist admissible constants
$0<\eta''=\eta''\q( \xi'\w)$, $0<\eta'$, $\eta'',\eta'\leq \eta^0$ such that\footnote{Here, $\eta''$ can be chosen to depend only on a fixed lower bound
for $\xi'$.}
for every $f\in C\q( \Omega\w)$, $f\geq 0$, we have:
$$\int_{\q|t\w|\leq \eta''} f\q( e^{t\cdot X}x\w) \: dt\lesssim A_{\B{X}{d}{\cdot}{\xi'}} f\q( x\w) \lesssim \int_{\q|t\w|\leq \eta'} f\q( e^{t\cdot X} x\w)\: dt,$$
for every $x\in \Omega'$.  Here, the implicit constants are allowed to depend
on a lower bound for $\xi'$.
\end{lemma}
\begin{proof}
This follows just as in Proposition \ref{PropUnitOpsPrelimProp}.  The
straightforward modifications are left to the reader.
\end{proof}

\begin{proof}[Proof of Proposition \ref{PropAvgIsInt}]
Take $\eta^0$, $\xi_0$ as in Lemmas \ref{LemmaAvgIsIntO} and \ref{LemmaAvgIsIntT}.
For $\xi'\leq \xi_0$, let $\eta',\eta''$ be as in the conclusion
of Lemma \ref{LemmaAvgIsIntT} (here and in the rest of the proof, all constants
are allowed to depend on a lower bound for $\xi'$--so that, in particular,
$\eta''\gtrsim 1$).
We then have, using Lemmas \ref{LemmaAvgIsIntO} and \ref{LemmaAvgIsIntT}
freely, for $f\in C\q( \Omega\w)$, with $f\geq 0$,
\begin{equation*}
\begin{split}
\int_K f\q( x\w) \: dx &\lesssim \int_{\q|t\w|\leq \eta''} \int_{\Omega''} f\q( e^{t\cdot X}x\w) \: dx\: dt\\
&\lesssim \int_{\Omega''} A_{\B{X}{d}{\cdot}{\xi'}} f\q( x\w) \: dx\\
&\lesssim \int_{\q|t\w|\leq \eta'} \int_{\Omega''} f\q( e^{t\cdot X}x\w) \: dx\: dt\\
&\lesssim \int_{\Omega'} f\q( x\w) \: dx,
\end{split}
\end{equation*}
which completes the proof.
\end{proof}

\begin{proof}[Proof of Proposition \ref{PropLeafwiseMax}]
Fix $x\in \Omega''$ and $\xi'$ as in the statement of the proposition.
In what follows all implicit admissible constants are also
allowed to depend on a lower bound for $\xi'$.
We define the maximal function $\sM$ in terms of this fixed $\xi'$,
as in \eqref{EqnDefnSingleMax}.
By definition of $\Omega''$, there exists $x_0\in K$ such that
$\B{X}{d}{x}{\xi'}\subseteq \B{X}{d}{x_0}{\xi_1}$.
Let $n_0=\dim \Span{X_1\q(x_0\w),\ldots, X_q\q( x_0\w)}$, and
let $\Phi:B_{n_0}\q(\eta_1\w) \rightarrow \B{X}{d}{x_0}{\xi}$
be the map guaranteed by Theorem \ref{ThmMultiBallsAtPoint} where
we take $\delta=\xi$ and $x=x_0$.
Note that,
$$\B{X}{d}{x}{\xi'}\subseteq \B{X}{d}{x_0}{\xi_1}\subseteq \Phi\q(B_{n_0}\q( \eta_1\w)\w).$$

Let $Y_1,\ldots, Y_q$ be the pullbacks of $X_1,\ldots, X_q$ via the
map $\Phi$, to
$B_{n_0}\q( \eta\w)$.
Note, for $u\in B_{n_0}\q( \eta\w)$ and $\delta>0$ small enough that
$\B{Y}{d}{u}{\delta}\Subset B_{n_0}\q( \eta\w)$, we have
$$ \Phi\q( \B{Y}{d}{u}{\delta} \w)  = \B{X}{d}{\Phi\q( u\w)}{\delta}.$$
Using that $\q|\det_{n_0\times n_0} d\Phi\q( u\w)\w|\approx \q| \det_{n_0\times n_0} X\q( x_0\w) \w|$,
and applying a change of variables as in \eqref{EqnChangeOfVars},
we see that
$$\Vol{\B{X}{d}{\Phi\q( u\w)}{\delta}}\approx \q|\det_{n_0\times n_0} X\q( x_0\w) \w| \Vol{\B{Y}{d}{u}{\delta}}.$$
It follows, for $f\in C\q( \Omega\w)$, with $f\geq 0$,
\begin{equation*}
\begin{split}
A_{\B{X}{d}{\cdot}{\delta}} \q( f\circ\Phi^{-1}\w) \q( \Phi\q( u\w)\w)
&= \frac{1}{\Vol{\B{X}{d}{\Phi\q( u\w)}{\delta}}} \int_{\B{X}{d}{\Phi\q( u\w)}{\delta}} f\q(\Phi^{-1}\q( y\w) \w) \: dy\\
&\approx \frac{1}{\Vol{\B{Y}{d}{u}{\delta}}} \int_{\B{Y}{d}{u}{\delta}} f\q( v\w) \:dv\\
&= A_{\B{Y}{d}{\cdot}{\delta}} f \q( u\w),
\end{split}
\end{equation*}
where, again, we have used \eqref{EqnChangeOfVars} and $dv$ denotes
Lebesgue measure on $B_{n_0}\q( \eta\w)$ and $dy$ denotes the induced
Lebesgue measure on the leaf in which $\B{X}{d}{\Phi\q( u\w)}{\delta}$ lies.
Consider, for $y\in \B{X}{d}{x_0}{\frac{\xi_1}{2}}$,
\begin{equation*}
\begin{split}
\sM \q(f\circ\Phi^{-1}\w) \q( y\w) &= \sup_{\xi'\geq \delta>0} A_{\B{X}{d}{\cdot}{\delta}} \q|f\circ \Phi^{-1}\w| \q( y\w) \\
&\approx \sup_{\xi'\geq \delta>0} A_{\B{Y}{d}{\cdot}{\delta}} \q|f\w| \q( \Phi^{-1}\q( y\w)\w)\\
&=: \sMt f \q( \Phi^{-1} \q( y\w)\w),
\end{split}
\end{equation*}
where $\sMt$ denotes the maximal function defined in terms of the Carnot-Carath\'eodory
balls defined by the vector fields $\q( Y,d\w)$.

Hence, we have,
\begin{equation*}
\begin{split}
A_{\B{X}{d}{\cdot}{\xi'}} \q|\sM \q(f\circ \Phi^{-1} \w)\w|^p \q( x\w) &\approx A_{\B{X}{d}{\cdot}{\xi'}} \q[\q|\sMt f\w|^p \circ \Phi^{-1}\w] \q(x\w)\\
&\approx A_{\B{Y}{d}{\cdot}{\xi}} \q|\sMt f\w|^p \q( \Phi^{-1}\q( x\w)\w).
\end{split}
\end{equation*}
Similarly, we have
\begin{equation*}
A_{\B{X}{d}{\cdot}{2\xi'}} \q| f\circ \Phi^{-1}\w|^p \q( x\w) \approx A_{\B{Y}{d}{\cdot}{2\xi'}} \q|f\w|^p \q( \Phi^{-1}\q( x\w)\w).
\end{equation*}
Thus, to complete the proof, it suffices to show the bound
\begin{equation*}
A_{\B{Y}{d}{\cdot}{\xi'}} \q|\sMt f\w|^p \q( \Phi^{-1}\q( x\w)\w) \lesssim 
A_{\B{Y}{d}{\cdot}{2\xi'}} \q|f\w|^p \q( \Phi^{-1}\q( x\w)\w).
\end{equation*}
Moreover, since $\Vol{\B{Y}{d}{\Phi^{-1}\q( x\w)}{\xi'}}\approx \Vol{\B{Y}{d}{\Phi^{-1}\q( x\w)}{2\xi'}}$ (in fact both are $\approx 1$, but we will not need this),
it suffices to show
\begin{equation*}
\q\|\sMt f\w\|_{L^p\q(\B{Y}{d}{\Phi^{-1}\q( x\w)}{\xi'}\w)}^p
\lesssim
\q\|f\w\|_{L^p\q(\B{Y}{d}{\Phi^{-1}\q( x\w)}{2\xi'}\w)}^p.
\end{equation*}
This is immediate from the classical theory of spaces of homogeneous
type, since the balls $\B{Y}{d}{\cdot}{\delta}$ satisfy all
the axioms of a space of homogeneous type, uniformly in
the relevant parameters.
\end{proof}

%% file: appendcalc.tex
In this appendix, we discuss two theorems from calculus that we will
use throughout the paper:  a uniform version of the
inverse function theorem, and how the smoothness of $e^{t_1X_1+t_2X_2+\cdots+t_{\nu}X_{\nu}}x_0$, as a function of $t_1,\ldots,t_\nu$, depends on the smoothness of
$X_1,\ldots, X_\nu$.  These results are surely familiar, in some
form or another, to the reader.  However, they play such a fundamental
role in our analysis, that we feel it is prudent to state them
in the precise form we shall use them.

For a $C^1$ vector field $Y$, one defines $E\q(t\w)=e^{tY}x_0$
to be the unique solution to the ODE $\frac{d}{dt} E\q(t\w) = Y\q(E\q(t\w)\w)$
satisfying $E\q( 0\w) =x_0$.  This unique solution always exists
for $\q|t\w|$ sufficiently small (depending on the $C^1$ norm of $Y$).  This allows us to define:
$$e^{t_1 X_1+\cdots + t_\nu X_\nu}x_0$$
for $\q|t\w|$ sufficiently small, where $t=\q( t_1,\ldots, t_\nu\w)$.
We have:

\begin{thm}\label{ThmExpReg}
Suppose $X_1,\ldots, X_\nu$ are $C^m$ vector fields ($m\geq 1$), defined on an
open set $\Omega\subseteq \R^n$.  Then, for $x_0$ fixed, the function:
$$u\q(t\w) = e^{t_1X_1+\cdots+t_\nu X_\nu}x_0-x_0$$
is $C^m$.  Moreover, the $C^m$ norm of this
function can be bounded in terms of $n,\nu$ and the $C^m$ norms of
$X_1,\ldots, X_\nu$.
\end{thm}
\begin{proof}
It is perhaps easiest to consider the function:
$$v\q( \epsilon, t\w) = e^{\epsilon\q( t_1X_1+\cdots+t_\nu X_\nu \w)}x_0.$$
Then, $u\q( t\w) = v\q( 1,t\w)-x_0$, and $v$ is defined by an ODE in
the $\epsilon$ variable.  That $v-x_0$ is $C^m$ (in both variables) is classical.
See \cite{DieudonneFoundationsOfModernAnalysis}, Chapter X.
Alternatively, one can modify the proof method in \cite{IzzoCrConvergenceOfPicardsSuccessiveApprox}
to this situation for a more elegant proof.
\end{proof}

\begin{rmk}\label{RmkExpMoreReg}
One could write $t$ in polar coordinates $r,\omega$, and consider
the function, $f\q( r,\omega \w) = u\q( r\omega\w)$.
Then, one has, $f\in C^m\q( r,\omega\w)$.  Moreover, one has,
for $a+\q|b\w|=m$ ($a\in \N$, $b$ a multi-index),
$$\partial_r \partial_r^a\partial_\omega^b f\q(r,\omega\w)$$
exists and is continuous.
\end{rmk}

We now turn to the inverse function theorem:
\begin{thm}\label{ThmInverseFunctionThm}
Fix an open set $U\subseteq R^n$, and fix $x_0\in U$.  Suppose
$K\subset C^1\q( U; R^n \w)$ is a compact set such that for all
$f\in K$, $\det{df}\q( x_0\w)\ne 0$, and hence, $\q| \det{df}\q( x_0\w) \w|$
is bounded away from $0$ uniformly for $f\in K$.  Then, there
exist constants $\delta_1,\delta_2>0$, such that for all $f\in K$,
\begin{itemize}
\item $f|_{B\q( x_0,\delta_1\w)}$ is a $C^1$ diffeomorphism onto its image.
\item $B\q( f\q( x_0\w),\delta_2\w) \subseteq f\q( B\q( x_0,\delta_1\w)\w)$.
\end{itemize}
here, $B\q( x_0,\delta\w)$ denotes the usual Euclidean ball centered at
$x_0$ of radius $\delta$.
\end{thm}
\begin{proof}
This follows from a straight-forward modification of the proof in
\cite{SpivakCalculusOnManifolds}, by using the Arzel\`a-Ascoli theorem.
Alternatively, since in our proofs, we will always show that the
relevant set is a pre-compact subset of $C^1$, by showing it is a
bounded subset of $C^2$, the derivatives of the functions in our set will actually
be uniformly Lipschitz, and in this case one may use
the theorem in \cite{HubbardHubbardVectorCalculusLinearAlgebra}.
\end{proof}

%% file: appendlinear.tex
In this appendix, we review the Cauchy-Binet formula and an associated
change of variables formula, that is essential to the work in this paper.
We first recall some notation from Section \ref{SectionResults}:  given two integers $1\leq m\leq n$, define $\sI{m}{n}$ 
to be the set of all lists of integers $\q( i_1,\ldots, i_m\w)$, such that:
$$1\leq i_1<i_2<\cdots<i_m\leq n.$$

Given an $n\times q$ matrix $A$, and $n_0\leq n\wedge q$, for $I\in \sI{n_0}{n}$,
$J\in \sI{n_0}{q}$ define the $n_0\times n_0$ matrix $A_{I,J}$ by
using the rows from $A$ which are listed in $I$ and the
columns from $A$ listed in $J$.  We define:
$$\det_{n_0\times n_0} A = \q( \det A_{I,J} \w)_{\substack{I\in \sI{n_0}{n}\\J\in \sI{n_0}{q}}}.$$
In particular, $\det_{n_0\times n_0} A$ is a vector (it will not be important to
us in which order the coordinates are arranged).

A special case arises when $q=n_0$.  Indeed, in this case, we have (\cite{ThrallThornheimVectorSpacesAndMatrices}, p. 127):
\begin{equation}\label{EqnDetsEqual}
\q|\det_{n_0\times n_0} A\w| = \sqrt{\det A^tA}
\end{equation}
and both of these quantities are equal to the volume of the $n_0$
dimensional parallelepiped with edges given by the columns of $A$.
This is a special case of the Cauchy-Binet formula.
Because of this, we obtain a change of variables formula which will
be of use to us.  Suppose $\Phi$ is a $C^1$ diffeomorphism from
an open subset $U$ in $\R^{n_0}$ mapping to an $n_0$ dimensional
submanifold of $\R^n$, where this submanifold is given the
Lebesgue measure, $dx$.  Then, we have:
\begin{equation}\label{EqnChangeOfVars}
\int_{\Phi\q( U\w)} f\q( x\w) dx = \int_U f\q(\Phi\q( t\w) \w) \q|\Det{n_0}{d\Phi\q(t\w)}\w| dt.
\end{equation}